\documentclass[aap,preprint]{imsart}

\RequirePackage[OT1]{fontenc}
\RequirePackage{amsthm,amsmath,amsfonts}
\RequirePackage[numbers]{natbib}
\RequirePackage[colorlinks,citecolor=blue,urlcolor=blue]{hyperref}

\startlocaldefs
\numberwithin{equation}{section}
\theoremstyle{plain}
\newtheorem{theorem}{Theorem}[section]
\newtheorem{lemma}[theorem]{Lemma}
\newtheorem{proposition}[theorem]{Proposition}

\newtheorem{corollary}[theorem]{Corollary}
\theoremstyle{definition}
\newtheorem{definition}[theorem]{Definition}
\newtheorem{remark}[theorem]{Remark}
\newtheorem{example}[theorem]{Example}
\endlocaldefs
\allowdisplaybreaks

\begin{document}

\begin{frontmatter}
\title{Abstract stochastic evolution  equations in M-type 2 Banach spaces}
\runtitle{T\d{a} Vi$\hat{\d{e}}$t T\^{o}n and Atsushi Yagi }

\begin{aug}
\author{T\d{a} Vi$\hat{\d{e}}$t T\^{o}n\thanksref{t2}\ead[label=e1]{taviet.ton@ist.osaka-u.ac.jp}}
and 
\author{ Atsushi Yagi\thanksref{t3}\ead[label=e2]{yagi@ist.osaka-u.ac.jp}}

\thankstext{t2}{This work is supported by the Japan Society for the Promotion of Science.}
\thankstext{t3}{This work is supported by Grant-in-Aid for Scientific Research (No. 20340035) of the Japan Society for the Promotion of Science.}
\runauthor{T\d{a} Vi$\hat{\d{e}}$t T$\hat{\rm o}$n and Atsushi Yagi}

\affiliation{Osaka University}

\address{Department of Information and Physical Science\\
 Graduate School of Information Science and Technology\\
  Osaka University\\
Suita Osaka 565-0871, Japan\\
\printead{e1}\\
\phantom{E-mail:\ }\printead*{e2}}

\end{aug}

\begin{abstract}
This paper devotes to studying abstract stochastic evolution  equations in M-type 2 Banach spaces. First, we handle  nonlinear evolution equations with multiplicative noise. The existence and uniqueness of local and global mild solutions under  linear growth and Lipschitz conditions on coefficients are presented.  The regular dependence of solutions on initial data is also studied. Second, we investigate   linear evolution equations with additive noise. The existence and uniqueness of strict  and mild solutions and  their regularity are shown.  Finally, we explore semilinear evolution equations  with additive noise. We concentrate on the existence, uniqueness and regular dependence of solutions on initial data.
\end{abstract}

\begin{keyword}[class=MSC]
\kwd[Primary ]{60H15}
\kwd{35R60}
\kwd[; secondary ]{47D06}
\end{keyword}

\begin{keyword}
\kwd{stochastic evolution equations}
\kwd{analytic semigroups}
\kwd{M-type 2 Banach spaces}
\end{keyword}
\end{frontmatter}

\section {Introduction}
The theory of stochastic partial differential equations in Hilbert spaces has been studied from 1970s. The basic theoretical problem on existence and uniqueness of solutions and the problem on regularity and regular dependence on initial data of solutions  are still of great interest today. 
Two main approaches are known to the abstract stochastic  evolution equations, namely, the variational methods and the semigroup methods. 
Some early work in the first approach are due to Bensoussan and Temam \cite{BensoussanTemam1,BensoussanTemam2}.
The fundamental work on monotone stochastic evolution equations is also due to Pardoux \cite{Pardoux1,Pardoux2}. There are many other important contributions in this approach, for examples Krylov-Rosovskii \cite{Krylov}, Pr\'{e}v\^{o}t and R\"{o}ckner \cite{PrevotRockner}, Viot \cite{Viot1,Viot2}, Gawarecki  and Mandrekar \cite {GawareckiMandrekar}  and references therein. 

Let us introduce the second approach. The semigroup methods, which were  initiated by the invention of the analytic semigroups in the middle of the last century, are characterized  by precise formulas representing the solutions of the Cauchy problem for deterministic evolution equations (see Hille \cite{Hille} and Yosida \cite{Yosida}). 
The analytical semigroup $S(t)=e^{-tA}$ generated by a linear operator $-A$ provides directly a fundamental solution to the Cauchy problem for an autonomous linear evolution equation
\begin{equation*}
\begin{cases}
\frac{dX}{dt}+AX=F(t), \quad\quad 0<t\leq T,\\
X(0)=X_0,
\end{cases}
\end{equation*}
and the solution is given by the formula $X(t)=S(t)X_0+\int_0^t S(t-s)F(s)ds.$ Similarly, a solution to the Cauchy problem for an autonomous nonlinear evolution equation
\begin{equation*}
\begin{cases}
\frac{dX}{dt}+AX=F(t,X), \quad\quad 0<t\leq T,\\
X(0)=X_0,
\end{cases}
\end{equation*}
can be obtained as a solution of an integral equation $X(t)=S(t)X_0+\int_0^t S(t-s)F(s,X(s))ds.$ For these problems, the solution formulas provide us important information on solutions such as uniqueness, regularity, smoothing effect and so forth. Especially, for nonlinear problems, one can derive Lipschitz continuity of solutions with respect to the initial values, even their Fr\'{e}chet differentiability. This powerful approach has been used in the study of stochastic evolution equations in Hilbert spaces.
Some early work  was proposed by Dawson  \cite{Dawson0,Dawson}, Curtain and Falb \cite{CurtainFalb}. More recent important contributions are due to Da Prato and his collaborations, see for examples \cite{prato-1,pratoFlandoliPriolaRockner,prato0,prato},  and references therein.

In this paper, we study abstract stochastic evolution equations in M-type 2 Banach spaces by using the semigroup methods. 
Let  $E$ be an M-type 2 real separable Banach space and $\mathcal B(E)$ be the Borel $\sigma$\,{-}\,field on  $E$.   Let $\{w_t, t\geq 0\}$ be a one-dimensional Brownian motion on a  complete probability space  $(\Omega, \mathcal F,\mathbb P)$ with a filtration $\{\mathcal F_t\}_{t\geq 0}$ satisfying the usual conditions  (see for example Arnold \cite{Arnold}, Friedman \cite{Fried}, Karatzas-Shreve \cite{IS}). 
We proceed to study abstract stochastic evolution  equations of the form 
\begin{equation} \label{E2}
\begin{cases}
dX+AXdt=F(t,X)dt+ G(t,X)dw_t,\\
X(0)=\xi,
\end{cases}
\end{equation}
where $A\,{:}\,\mathcal D(A)\subset E\to E$ is a densely defined, closed linear operator in $E$. The functions  $F$ and $ G$   are  $E$-valued random variables. The initial value $ \xi$ is an $\mathcal  F_0$-measurable  random variable.

A linear form of \eqref{E2} in an M-type 2 and UMD Banach space, i.e.  $F(t,x)=F(t)$ is depending only on $t$ and $G(t,x)=Bx$, where $B$ is a linear operator from the space to itself, is investigated in  Brze\'{z}niak \cite{Brzezniak}. The author showed a sufficient condition on the operator $A$ and the  linear operator $B$, which is assumed to be bounded as an operator from $\mathcal D(A)$ into $\mathcal D_A(\frac{1}{2},2)$, under which  \eqref{E2} has a strict solution. 

Our objective is to study  the existence and uniqueness, regularity and dependence on initial data of solutions of \eqref{E2}. Our contribution is threefold. First, we handle  nonlinear evolution equations with multiplicative noise. Second, we investigate linear evolution equations with additive noise. Finally, we concentrate on semilinear evolution equations with additive noise. 
The work on linear  and semilinear evolution equations with multiplicative noise is in preparation \cite{Ton1,Ton2}.

Let us describe the content of the present paper. In Section \ref{section2}, we introduce the stochastic integral in  the M-type 2 Banach space $E$, concepts of solutions of \eqref{E2}, function spaces with values in $E$ and an integral inequality of Volterra type.

In Section \ref{section3}, we show existence of solutions of \eqref{E2} as well as regular dependence of solutions on initial data under some conditions on the coefficients $F$ and $G$. This section contains three subsections. In the first subsection, we assume that $F$ and $G$ satisfy the linear growth  and  Lipschitz conditions.
Theorem \ref{linear growth and Lipschitz conditions} gives the existence of global mild solutions of \eqref{E2}. The proof of the theorem is similar to that in Da Prato-Zabczyk \cite{prato}. 
In the second subsection, we assume that $F$ and $G$ satisfy the local  linear growth  and  local Lipschitz conditions.
 Theorem \ref{local existence theorem} shows the existence of maximal local mild solutions of \eqref{E2}. The proof is based on Lemma \ref{lem1}, Lemma \ref{lem2} and Theorem \ref{linear growth and Lipschitz conditions}. Noting that these lemmas come from  earlier results in the theory of ordinary differential equations (see for example Friedman \cite{Fried}). 
 Theorem \ref{dependence of local mild solution on the initial data} gives the dependence of the maximal local mild solution on the initial data.
 In the last subsection, we assume that $F$ and $G$ satisfy the  linear growth  and   local Lipschitz conditions.
Theorem \ref{thm1} demonstrates the existence of global mild solutions of \eqref{E2}. Theorem \ref{dependence of global mild solution on the initial data} presents  the dependence of the global mild solutions on the initial data. 

In Section \ref{section4}, we concentrate on a class of equations of \eqref{E2}, namely,  linear evolution equations  with additive noise.  We assume that $F(t,x)=F(t)$ and $G(t,x)=G(t)$ are depending only on $t$. These  functions  are considered in the spaces $\mathcal F^{\beta,\sigma}((0,T];E) $ and $\mathcal B((0,T];E),$ which will be defined in Section \ref{section2}. Theorem \ref{StrictSolutions} gives the existence of  strict solutions.   
Theorem \ref{regularity theorem autonomous linear evolution equation} explores the regularity of mild solutions without  the assumption  
$|AS(t)|\leq c_\delta t^{-\delta}, t\in [0,T], \delta\in (0,\beta)$ of Theorem  \ref{StrictSolutions}.

In Section \ref{section5}, we set $F(t,x)=F_1(x)+F_2(t)$ and $G(t,x)=G(t),$ where $F_1, F_2$ and $G$  are depending only on $x$ and  $t$, respectively.   The  corresponding  equation \eqref{semilinear evolution equation} is then  the form of   semilinear evolution equations  with additive noise. We suppose  that the function $F_1$ is defined only on a subset of the space $E$, namely $\mathcal D(F_1)= \mathcal D(A^\eta),$ and that $F_1$ satisfies a Lipschitz condition on its domain (see \eqref{AbetaLipschitzcondition}). To treat  \eqref{semilinear evolution equation}  we require that the initial condition  takes values in a smaller space, say $\mathcal D(A^\beta)$.
Theorem \ref{semilinear evolution equationTheorem1} proves the existence of mild solutions in the function space $ \mathcal C((0,T_{F,G,\xi}];\mathcal D(A^\eta))\cap \mathcal C([0,T_{F,G,\xi}];\mathcal D(A^\beta))$. Theorem  \ref{semilinear evolution equationMoreRegular}  gives a more stronger regularity under more regular initial values. Theorem \ref{continuityofsolutionsininitialdata}   presents some results on regular  dependence of solutions on initial data. Theorem \ref{semilinear evolution equationTheorem2} shows the existence of solutions for a critical case of the Lipschitz condition on $F_1$. Finally, Theorem \ref{continuityofsolutionsininitialdata2} explores the regular dependence of solutions on initial data for this critical case.
\section{Preliminary}  \label{section2}
\subsection{Stochastic integrals in M-type 2 Banach spaces}
This subsection reviews the construction and some properties of the stochastic integral in M-type 2 real separable space $E$. All details in this subsection one can find in the work of  Dettweiler \cite{Dettweiler1,Dettweiler,Dettweiler3},  Pisier \cite{Pisier0,Pisier} and Brze\'{z}niak \cite{Brzezniak}.
\begin{definition}[Pisier \cite{Pisier}]  \label{MType2BanachSpace}
A Banach sapce $E$ is of $M$-type 2 (or martingale type 2) if there is a constant $c(E)$ such that for all $E$-valued martingale $\{M_n\}_n$ the  inequality 
$$\sup_n \mathbb E|M_n|^2 \leq c(E) \sum_{n\geq 0}\mathbb E |M_n-M_{n-1}|^2$$
 holds with the convention $M_{-1}=0.$
\end{definition}
\begin{example}
Every $L^p$ space with $p\in[2,\infty)$ is of $M$-type 2.
\end{example}
Let us first define the stochastic integral for step functions. Let $f:[0,T]\times \Omega\to E$ be an adapted random step function, i.e. there exist  sequences $\{t_i\}_0^n: 0=t_0<\cdots<t_n=T$  and $\{f_i\}_0^{n-1}: f_i\in L^2(\Omega,\mathcal F_{t_i}, \mathbb P; E)$ such that $f(s)=f_i$ a.e. for $s\in[t_i,t_{i+1})$. Then the stochastic integral of $f$ on $[0,T]$ with respect to $w_t$ is defined by
$$I_T(f):=\int_0^T f(t)dw_t=\sum_0^{n-1} (w_{t_{i+1}}-w_{t_i})f_i.$$
It is obvious  that $I_T(f)$ is $\mathcal F_T$-measurable. In addition, by putting 
$$M_k=\sum_{i=0}^k (w_{t_{i+1}}-w_{t_i})f_i,$$  
then $I_T(f)=M_{n-1}$ and $\{M_k\}_{k=0}^{n-1}$ is a martingale. In view of Definition \ref{MType2BanachSpace}, we have
\begin{align} 
\mathbb E|I_T(f)|^2\leq& \sup_k \mathbb E|M_k|^2
\leq  c(E) \sum_{k=0}^{n-1}\mathbb E| (w_{t_{k+1}}-w_{t_k})f_i |^2  \label{Eq-1}\\
= & c(E) \sum_{k=0}^{n-1}\mathbb E|w_{t_{k+1}}-w_{t_k}|^2 \mathbb E|f_i |^2  \notag\\
=& c(E) \sum_{k=0}^{n-1}(t_{k+1}-t_k) \mathbb E|f_i |^2 \notag\\
=& c(E) \int_0^T \mathbb E|f(t)|^2dt. \notag
\end{align}

Denote by $\mathcal P_\infty$ the predictable $\sigma$\,{-}\,field on $\Omega_{\infty}=[0,\infty)\times \Omega$  generated by sets of the form 
$$(s,t]\times K_1, \,\, 0\leq s<t<\infty, K_1\in \mathcal F_s  \text{ and } \{0\}\times K_2,\,\, K_2\in \mathcal F_0,$$
and denote by $\mathcal P_T$ the restriction of $\mathcal P_\infty$ to $\Omega_T=[0,T]\times \Omega.$ 
A process $\phi$ is called predictable if it is measurable from  $(\Omega_T, \mathcal P_T)$ into $(E,\mathcal B(E))$.
We then denote by $\mathcal N^2(0,T)$  the set of all $E$\,{-}\,valued predictable  processes $\phi$ such that $\mathbb E\int_0^T |\phi(t)|^2 dt<\infty$. Thanks to the inequality \eqref{Eq-1}, one can define the stochastic integral for  functions in $\mathcal N^2(0,T)$ (a set  $\mathcal N^2(a,b), a,b\in \mathbb R,$ and the integral for functions in  $\mathcal N^2(a,b)$   are defined similarly). Indeed, due to  \eqref{Eq-1}, it is not hard to show that the limit in the following definition exists and does not depend on the actual choice of step functions.
\begin{definition}[stochastic integrals]
Let $f\in \mathcal N^2(0,T)$. The stochastic integral of $f$ on $[0,t], 0\leq t\leq T$ is defined by
$$I_t(f):=\int_0^t f(s) dw_s=\lim_{n\to\infty} \int_0^t \phi_n(s)dw_s \hspace{1cm} (\text{limit in } L^2(\mathbb P))
$$
where $\{\phi_n\}_n$ is a sequence of step functions such that
$$\lim_{n\to\infty} \mathbb E \int_0^t|f(s)-\phi_n(s)|^2 ds=0.$$
\end{definition}
\begin{proposition}\label{IntegralInequality}
Let $f\in \mathcal N^2(0,\infty)$. Then
\begin{itemize}
  \item [{\rm (i)}] $\mathbb E|I_t(f)|^2\leq c(E) \int_0^t \mathbb E|f(s)|^2ds$, \quad $t\geq 0$.
  \item [{\rm (ii)}] $\{I(f)\}_t$ is an $E$-valued continuous martingale and $I(f)\in \mathcal N^2(0,\infty).$
  \item [{\rm (iii)}] For any $p>1, T>0$ 
$$\mathbb E \sup_{[0,T]} \Big|\int_0^t \phi(s) dw_s\Big|^p \leq \Big(\frac{p}{p-1}\Big)^p c_p(E) \mathbb E\Big[\int_0^T |\phi(s)|^2 ds\Big]^{\frac{p}{2}},$$
where $c_p(E)$ is some constant depending only on $p, E$.
\end{itemize}
\end{proposition}
\subsection{Concept of solutions}
 Throughout this paper,  if not specified, we always assume that $F$ and $ B$ are   measurable  from  $ (\Omega_T \times E, \mathcal P_T\times \mathcal B(E))$ into $(E,\mathcal B(E))$.  In addition, we assume that  $(-A)$ generates  a strongly continuous semigroup $S(t)=e^{-tA}, t\geq 0$ on $E.$ We then set  $M_t=\sup_{0\leq s\leq t} |S(s)|.$

Following the definition of  local solutions to stochastic differential equations (e.g. Arnold \cite{Arnold}, Friedman \cite{Fried}, Mao \cite{Mao}), we present a definition of local mild solutions for \eqref{E2}.
\begin{definition} \label{def0}
Let $\tau$ be a stopping time such that $\tau\leq T$ a.s. 
A predictable $E$-valued continuous process  $\{X(t) , t\in [0,\tau)\}$   is called a local mild solution  of \eqref{E2}  if  the followings are satisfied. \begin{itemize}
  \item  [\rm (i)] There exists a sequence $\{\tau_k\}_{k=1}^\infty$ of  stopping times such that $0\leq \tau_k\leq \tau_{k+1}, k\geq 1 $   and $\lim_{k\to\infty} \tau_k=\tau$ \, a.s.
  \item  [\rm (ii)] For  $ t\in[ 0,T] $ and $k\geq 1$, $S(t-\cdot)G(\cdot,X(\cdot))\in \mathcal N^2(0,t\wedge \tau_k)$, i.e.
            $$\mathbb E\int_0^{t\wedge \tau_k} |S(t-s)G(s,X(s))|^2 ds<\infty. $$
  \item  [\rm (iii)] For  $ t\in[ 0,T] $ and $k\geq 1$ 
$$\int_0^{t\wedge \tau_k} |S(t-s)F(s,X(s))| ds<\infty \hspace{2cm} \text{ a.s.  } $$
  \item   [\rm (iv)]  For  $ t\in[ 0,T] $ and $k\geq 1$ 
\begin{equation*}
\begin{aligned}
X(t\wedge \tau_k)=&S(t\wedge\tau_k)\xi +\int_0^{t\wedge\tau_k}S(t-s) F(s, X(s))ds\\
& +\int_0^{t\wedge\tau_k} S(t-s) G(s,X(s))dw_s \hspace{2cm}\text{a.s. } \end{aligned}
\end{equation*}
\end{itemize}
 If in addition  $\lim_{t\uparrow \tau} |X(t)|=\infty$ on $\{\tau< T\},$ then  $\{X(t) , t\in [0,\tau)\}$ is called a maximal local mild solution.  A maximal local mild solution $\{X(t),0\leq t<\tau\}$ is said to be unique if any other maximal local  mild solution $ \{\bar X(t),0\leq t<\bar \tau\}$  is indistinguishable from it, means that $\mathbb P\{\tau=\bar\tau\}=\mathbb P\{X(t)=\bar X(t) \text { for every } t\in [0,\tau)\}=1$.
\end{definition}
 The following definition of  global mild solutions is presented in Da Prato-Zabczyk \cite{prato} and Brze\'{z}niak \cite{Brzezniak}.
\begin{definition}\label{Def1}
A predictable $E$-valued continuous process $X(t), t\in [0,T]$ is called a global mild solution (briefly, a mild solution) of \eqref{E2} if $X\in \mathcal N(0,T)$ 
and for every $t\in[0,T]$
\begin{align}
 \label{DefinitionGlobalMildSolutions}X(t)=&S(t)\xi +\int_0^tS(t-s) F(s, X(s))ds\\
&+ \int_0^t S(t-s) G(s,X(s))dw_s \hspace{2cm} \text{a.s.}\notag
\end{align}
\end{definition}  
It is obvious that a maximal local mild solution $\{X(t),t\in [0,\tau)\}$ of \eqref{E2} is a  mild solution on $[0,T]$ if  $\tau=T$ a.s. and $X$ satisfies \eqref{DefinitionGlobalMildSolutions}  and is continuous at $t=T$.
\begin{definition}\label{Def2}
A predictable $E$-valued continuous process $X(t) , t\in [0,T]$ is called a strict solution of \eqref{E2} if
\begin{itemize}
  \item [\rm(i)] $G(\cdot,X(\cdot))\in \mathcal N^2(0,T)$,
  \item [\rm(ii)] $|\int_0^t F(s,X(s))ds| <\infty, $ \hspace{1cm} $t\in[0,T]$,
  \item [\rm(iii)] $X(t)\in D(A) $  and $\int_0^t \mathbb E|X(s)|_{D(A)}^2 ds<\infty,$ \hspace{1cm}$t\in (0,T],$
  \item [\rm(iv)]  for every $t\in(0,T]$
$$X(t)=\xi -\int_0^t AX(s)ds+\int_0^t F(s, X(s))ds+ \int_0^t  G(s,X(s))dw_s \hspace{1cm} \text{a.s.}$$
\end{itemize}
 A strict (mild) solution $\{X(t),0\leq t\leq T\}$ is said to be unique if any other strict (mild) solution $ \{\bar X(t),0\leq t\leq T\}$  is indistinguishable from it, means that $\mathbb P\{X(t)=\bar X(t) \text { for every } t\in [0,T]\}=1$.
\end{definition}
\subsection{Function spaces with values in a Banach space}
Let $I$ be an interval of the real line. By $\mathcal B(I;E)$, we denote the space of uniformly bounded $E$-valued functions on $I$. The space is a Banach space with the supremum norm
$$|G|_{\mathcal B(I;E)}=\sup_{t\in I} |G(t)|, \quad\quad G\in \mathcal B(I;E).$$
Denotes by $\mathcal C(I;E)$ the space of $E$-valued continuous functions. The following well-known result is used very often.
\begin{theorem}
Let $A$ be a closed linear operator of $E$ and $a,b\in \mathbb R, a\leq b.$
\begin{itemize}
\item [{\rm (i)}] If $G\in \mathcal C([a,b];E)$ and $AG\in \mathcal C([a,b];E)$ then
$$A\int_a^b G(t)dt=\int_a^b AG(t)dt.$$
\item [{\rm (ii)}]  If $G\in \mathcal N^2(a,b) $ and $AG\in  \mathcal N^2(a,b)$ then
$$A\int_a^b G(t)dw_t=\int_a^b AG(t)dw_t.$$
\end{itemize}
\end{theorem}

For an exponent $\sigma>0$, $\mathcal C^\sigma([a,b];E), a\leq b$ denotes the space of functions which are H\"{o}lder continuous on $[a,b]$ with exponent $\sigma$. The space is equipped with  norm
$$|G|_{\mathcal C^\sigma([a,b];E)}=\sup_{a\leq s<t\leq b} \frac{|G(t)-G(s)|}{(t-s)^\sigma}.$$
The Kolmogorov test gives a sufficient condition for a stochastic process to be H\"{o}lder continuous.
\begin{theorem}[Kolmogorov test, see e.g. Da Prato-Zabczyk \cite{prato} ]  \label{Kolmogorov test}
Let $\zeta(t), t\in [0,T]$ be  an $E$-valued stochastic process such that  for some constants $c>0, \epsilon_i>0, i=1,2$ and all $t,s\in [0,T]$
\begin{equation} \label{Kolmogorov test condition}
\mathbb E|\zeta(t)-\zeta(s)|^{\epsilon_1}\leq c |t-s|^{1+\epsilon_2}.
\end{equation}
Then $\zeta$ has  an version whose $\mathbb P$-almost all trajectories are H\"{o}lder continous functions with an  arbitrary exponent smaller than $\frac{\epsilon_2}{\epsilon_1}$.
\end{theorem}
When the process $\zeta(t)$ in Theorem \ref{Kolmogorov test} is a Gaussian process, one can weaken the condition \eqref{Kolmogorov test condition}.
\begin{theorem}  \label{Kolmogorov testGaussian}
Let $\zeta(t), t\in [0,T]$ be an $E$-valued   Gaussian process such that $\mathbb E \zeta(t)=0, t\geq 0$, and that for some constants $c>0, \epsilon\in (0,1]$ and all $t,s\in [0,T]$
$$
\mathbb E|\zeta(t)-\zeta(s)|^2\leq c |t-s|^\epsilon.
$$
Then there exists a modification of $\zeta$ with $\mathbb P$-almost all trajectories being H\"{o}lder continous functions with an  arbitrary exponent smaller than $\frac{\epsilon}{2}$.
\end{theorem}
For two exponents $0<\sigma<\beta\leq 1$ we define a function space  $\mathcal F^{\beta, \sigma}((0,T]; E)$ as follows, see Yagi \cite{yagi} ($\mathcal F^{\beta, \sigma}((a,b]; E), a<b$ is defined similarly). The space  $\mathcal F^{\beta, \sigma}((0,T]; E)$ consists of all continuous function $f(t)$ on $(0,T]$ (resp. $[0,T]$) when $0<\beta<1$ (resp. $\beta=1$)  with the following three properties:
\begin{enumerate}
  \item When $\beta<1$, $t^{1-\beta} f(t) $ has a limit as $t\to 0$.
  \item The function $f$ is H\"{o}lder continuous with  exponent $\sigma$ and with the weight $s^{1-\beta+\sigma}$, i.e.
$$\sup_{0\leq s<t\leq T} \frac{s^{1-\beta+\sigma}|f(t)-f(s)|}{(t-s)^\sigma}=\sup_{0\leq t\leq T}\sup_{0\leq s<t}\frac{s^{1-\beta+\sigma}|f(t)-f(s)|}{(t-s)^\sigma}<\infty.$$
  \item 
  \begin{equation} \label{Fbetasigma3}
  \lim_{t\to 0} \sup_{0\leq s\leq t}\frac{s^{1-\beta+\sigma}|f(t)-f(s)|}{(t-s)^\sigma}=0.
  \end{equation}
\end{enumerate}  
Then  $\mathcal F^{\beta, \sigma}((0,T]; E)$ becomes a Banach space with  norm
$$|f|_{\mathcal F^{\beta, \sigma}}=\sup_{0\leq t\leq T} t^{1-\beta} |f(t)|+ \sup_{0\leq s<t\leq T} \frac{s^{1-\beta+\sigma}|f(t)-f(s)|}{(t-s)^\sigma}.$$
 
The following useful inequality follows the definition directly. For every $ f\in \mathcal F^{\beta, \sigma}((0,T]; E), 0<s<t\leq T$ we have
\begin{equation} \label{FbetasigmaSpaceProperty}
\begin{cases}
|f(t)|\leq |f|_{\mathcal F^{\beta, \sigma}} t^{\beta-1}, \\
 |f(t)-f(s)| \leq |f|_{\mathcal F^{\beta, \sigma}} (t-s)^{\sigma} s^{\beta-\sigma-1}.
\end{cases}
\end{equation}
In addition, it is not hard to show that
\begin{equation} \label{FbetaFgammasigmaSpaceProperty}
\mathcal F^{\gamma,\sigma} ((0,T];E)\subset \mathcal F^{\beta,\sigma} ((0,T];E), \hspace{2cm} 0<\sigma<\beta<\gamma\leq 1.
\end{equation}

The space $\mathcal F^{\beta, \sigma}((0,T]; E)$ is not  a trivial space. Indeed, we have
\begin{remark}[Yagi \cite{yagi}]
When  $0<\sigma<\beta<1$,  $f(t) =t^{\beta-1} g(t) \in \mathcal F^{\beta, \sigma}((0,T]; E),$ where  $g(t)$ is any $E$-valued  function on $[0,T]$ such that
 $g\in \mathcal C^\sigma([0,T];E)$
  and $g(0)=0.$
When  $0<\sigma<\beta=1$,  the space $ \mathcal F^{1, \sigma}((0,T]; E)$ includes the space of  H\"{o}lder continuous functions with  exponent $\sigma.$ 
\end{remark}
\subsection{Integral inequality of Volterra type}
Let us introduce an useful inequality of Volterra type that will be  used in this paper. The proof of the inequality can be found, for example, in Yagi \cite{yagi}.
\begin{lemma} \label{Integral inequality of Volterra type}
Let $a\geq 0, b>0,  \mu_1>0 $ and $\mu_2>0$ be constants. 
\begin{itemize}
\item [{\rm (i)}]
Let $\Gamma$ be the gamma function. Then  the function defined by the series 
 $$E_{\mu,\nu}(t)=\sum_{n=0}^\infty \frac{t^{n\nu}}{\Gamma(\mu_1+n\mu_2)}, \hspace{1cm} 0\leq t<\infty$$
 satisfies an estimate
 $$E_{\mu,\nu}(t)\leq \frac{2}{\Gamma_0 \nu_2} (1+t)^{2-\mu_1} e^{t+1}, \hspace{1cm} 0\leq t<\infty,$$
 where $\Gamma_0=\min_{0<s<\infty} \Gamma(s)$. 
\item [{\rm (ii)}]
Let $\varphi(t,s)$ be a nonnegative  continuous  function defined for $0\leq s<t\leq T.$ If $\varphi(t,s)$ satisfies the integral inequality 
$$\varphi(t,s) \leq a (t-s)^{\mu_1-1}+b \int_s^t (t-r)^{\mu_2-1}\varphi(r,s)dr, \quad\quad 0\leq s<t\leq T,$$
then
$$\varphi(t,s)\leq a \Gamma(\mu_1) (t-s)^{\mu_1-1} E_{\mu_1,\mu_2} ([b\Gamma(\mu_2)]^{\frac{1}{\mu_2}}(t-s)),\quad\quad 0\leq s<t\leq T.$$\end{itemize}
\end{lemma}
\section{Nonlinear evolution equations with multiplicative noise}   \label{section3}
\subsection{Existence of global mild solutions under linear growth and Lipschitz conditions}
In this subsection, we shall show existence and uniqueness of global mild solutions of \eqref{E2}  
under  linear growth and Lipschitz conditions  on $F$ and $G$. The proof is similar to that in Da Prato-Zabczyk \cite{prato}.
\begin{theorem}[global existence]\label{linear growth and Lipschitz conditions}
Assume that $F$ and $G$ satisfy  two conditions:
\begin{itemize}
  \item  [{\rm (i)}] The linear growth condition
\begin{equation}  \label{The linear growth condition}
|F(t,x)|+|G(t,x)|\leq c_1(1+|x|),  \quad\quad x\in E, t\in[0,T]. 
\end{equation}
  \item [{\rm (ii)}]  The Lipschitz condition
\begin{equation} \label{The Lipschitz condition}
|F(t,x)-F(t,y)|+|G(t,x)-G(t,y)| \leq c_2 |x-y|, \quad \quad x, y \in E, t\in[0,T], 
\end{equation} 
\end{itemize}
where $c_i>0 \,(i=1,2) $ are some positive constants. 
 Suppose further that $\mathbb E|\xi|^p< \infty$ for some $p\geq 2$. Then there exists a unique  mild solution $X(t)$ to \eqref{E2} on $[0,T]$.  Furthermore, it satisfies the estimate
\begin{equation}  \label{EEE}
\sup_{0\leq t\leq T}\mathbb E |X(t)|^{p}] \leq \alpha(1+\mathbb E|\xi|^{p}),
\end{equation}
where $\alpha=\alpha(c_1,p,M_T,T)>0$ is  some constant  depending only on $c_1, p, M_T$ and $T.$
\end{theorem}
\begin{proof} We shall use the fixed point theorem  and the Gronwall  lemma. The proof is divided into three steps.

{\bf Step 1.} Let us show  existence of a mild solution.
  Put
\begin{align} 
\mathcal Q_1(Y)(t)&=\int_0^t S(t-s) F(s,Y(s)) ds,  \label{Q1Definition} \\
 \mathcal Q_2(Y)(t)&=\int_0^t S(t-s) G(s,Y(s))dw_s, \label{Q2Definition}\\
\mathcal Q(Y)(t)&=S(t) \xi + \mathcal Q_1(Y)(t)+\mathcal Q_2(Y)(t)  \notag
\end{align}
and let $\mathcal E_p(0,\bar T) (\bar T\leq T)$ be the set of all $E$-valued predictable process $Y(t)$  on $[0,\bar T]$ such that
$\sup_{[0,\bar T]} \mathbb E |Y(t)|^p<\infty. $ Then up to indistinguishability,  $\mathcal E_p(0,\bar T)$ is a Banach space with  norm
$$||Y||_{p,\bar T}=[\sup_{t\in[0,\bar T]} \mathbb E |Y(t)|^p]^\frac{1}{p}.$$

Let us  show that $\mathcal Q(\mathcal E_p(0,\bar T))\subset \mathcal E_p(0,\bar T)$. Indeed, by using H\"{o}lder inequality, we have
\begin{align}
||\mathcal Q_1(Y)||_{p,\bar T}^p &\leq \sup_{t\in[0,\bar T]} \Big[\int_0^t |S(t-s) F(s,Y(s))| ds\Big]^p\notag\\
&\leq M_{\bar T}^p \mathbb E \Big[\int_0^{\bar T} |F(s,Y(s)| ds\Big]^p\notag\\
&\leq (c_1M_{\bar T})^p \mathbb E \Big[\int_0^{\bar T} [1+|Y(s)| ]ds\Big]^p\notag\\
&\leq (c_1M_{\bar T})^p {\bar T}^{p-1} \mathbb E \int_0^{\bar T} [1+|Y(s)| ]^pds\notag\\
&\leq (c_1M_{\bar T})^p (2{\bar T})^{p-1} \mathbb E \int_0^{\bar T} [1+|Y(s)|^p ]ds  \label{Eq30}\\
&\leq (c_1{\bar T}M_{\bar T})^p 2^{p-1}  [1+||Y||_{p,\bar T}^p ]<\infty, \quad  \quad\quad Y\in \mathcal E_p(0,{\bar T}),\notag
\end{align}
here we used the inequality $(a+b)^p \leq 2^{p-1} (a^p+b^p), a>0, b>0.$
Thus, $\mathcal Q_1(\mathcal E_p(0,{\bar T}))\subset \mathcal E_p(0,{\bar T})$. Furthermore, due to  Proposition \ref{IntegralInequality}, we have
\begin{align}
||\mathcal Q_2(Y)||_{p,\bar T}^p&=\sup_{t\in[0,{\bar T}]} \mathbb E \Big|\int_0^t S(t-s) G(s,Y(s))dw_s\Big|^p\notag\\
&\leq \Big(\frac{p}{p-1}\Big)^p c_p(E) \mathbb E \Big[\int_0^t |S(t-s) G(s,Y(s))|^2 ds\Big]^{\frac{p}{2}}\notag\\
&\leq \Big(\frac{pM_{\bar T}}{p-1}\Big)^p c_p(E) \mathbb E \Big[\int_0^{\bar T} |G(s,Y(s))|^2 ds\Big]^{\frac{p}{2}}\notag\\
&\leq \Big(\frac{c_1pM_{\bar T}}{p-1}\Big)^p c_p(E) \mathbb E \Big[\int_0^{\bar T} [1+|Y(s)|]^2 ds\Big]^{\frac{p}{2}}\notag\\
&\leq \Big(\frac{c_1pM_{\bar T}}{p-1}\Big)^p c_p(E)  {\bar T}^{\frac{p-2}{2}} \mathbb E \int_0^{\bar T} [1+|Y(s)|]^p ds\notag\\
&\leq \Big(\frac{c_1pM_{\bar T}}{p-1}\Big)^p c_p(E)  {\bar T}^{\frac{p-2}{2}} 2^{p-1} \mathbb E \int_0^{\bar T} [1+|Y(s)|^p] ds\notag\\
&= \Big(\frac{c_1pM_{\bar T}}{p-1}\Big)^p c_p(E)  {\bar T}^{\frac{p-2}{2}} 2^{p-1}  \int_0^{\bar T} [1+\mathbb E|Y(s)|^p] ds\label{Eq31}\\
&\leq \Big(\frac{c_1pM_{\bar T}}{p-1}\Big)^p c_p(E)  {\bar T}^{\frac{p}{2}} 2^{p-1} [1+||Y||_{p,\bar T}^p]<\infty, \quad  \quad  Y\in \mathcal E_p(0,{\bar T}).\notag
\end{align}
Therefore, $\mathcal Q_2(\mathcal E_p(0,{\bar T}))\subset \mathcal E_p(0,{\bar T})$. We thus have shown that  $\mathcal Q(\mathcal E_p(0,{\bar T}))\subset \mathcal E_p(0,{\bar T})$. 

Let us next verify that   $\mathcal Q$ is a contraction mapping of $\mathcal E_p(0,{\bar T})$, provided $\bar T>0$ is sufficiently small. For any $Y_1, Y_2 \in \mathcal E_p(0,{\bar T})$  we have
\begin{align*}
||\mathcal Q_1(Y_1)- \mathcal Q_1(Y_2)||_{p,\bar T}^p&=\sup_{t\in[0,{\bar T}]} \mathbb E \Big|\int_0^t S(t-s) (F(s,Y_1(s))-F(s,Y_2(s))) ds\Big|^p\\
&\leq M_{\bar T}^p \sup_{t\in[0,{\bar T}]} \mathbb E\Big[\int_0^t  |F(s,Y_1(s))-F(s,Y_2(s))| ds\Big]^p\\
&\leq (c_2M_{\bar T})^p  \mathbb E \Big[\int_0^{\bar T}  |Y_1(s)-Y_2(s)| ds\Big]^p\\
&\leq (c_2M_{\bar T})^p {\bar T}^{p-1} \mathbb E \int_0^{\bar T}  |Y_1(s)-Y_2(s)|^p ds\\
&\leq (c_2{\bar T}M_{\bar T})^p  ||Y_1-Y_2||_{p,\bar T}^p 
\end{align*}
and
\begin{align*}
|&|\mathcal Q_2(Y_1)- \mathcal Q_2(Y_2)||_{p,\bar T}^p\\
&=\sup_{t\in[0,{\bar T}]} \mathbb E \Big|\int_0^t S(t-s) [G(s,Y_1(s))-G(s,Y_2(s))] dw_s\Big|^p\\
&\leq \Big(\frac{p}{p-1}\Big)^p c_p(E)\sup_{t\in[0,{\bar T}]} \mathbb E \Big[\int_0^t |S(t-s) [G(s,Y_1(s))-G(s,Y_2(s))]|^2 ds\Big]^{\frac{p}{2}}\\
&\leq \Big(\frac{pM_{\bar T}}{p-1}\Big)^p c_p(E)\sup_{t\in[0,{\bar T}]} \mathbb E \Big[\int_0^t |G(s,Y_1(s))-G(s,Y_2(s))|^2 ds\Big]^{\frac{p}{2}}\\
&\leq \Big(\frac{c_2pM_{\bar T}}{p-1}\Big)^p c_p(E) \mathbb E \Big[\int_0^{\bar T} |Y_1(s)-Y_2(s)|^2 ds\Big]^{\frac{p}{2}}\\
&\leq \Big(\frac{c_2pM_{\bar T}}{p-1}\Big)^p c_p(E)  {\bar T}^{\frac{p-2}{2}}\mathbb E \int_0^{\bar T} |Y_1(s)-Y_2(s)|^p ds\\
&\leq \Big(\frac{c_2pM_{\bar T}}{p-1}\Big)^p c_p(E)  {\bar T}^{\frac{p}{2}}||Y_1-Y_2||_{p,\bar T}^p.
\end{align*}
Hence,
\begin{align*}
||\mathcal Q(Y_1)- \mathcal Q(Y_2)||_{p,\bar T}&\leq ||\mathcal Q_1(Y_1)- \mathcal Q_1(Y_2)||_{p,\bar T}+||\mathcal Q_2(Y_1)- \mathcal Q_2(Y_2)||_{p,\bar T}\\
&\leq c_2M_{\bar T} \sqrt{{\bar T}} \Big[ \sqrt{{\bar T}}+ \frac{pC_p^{\frac{1}{p}}(E)}{p-1} \Big] ||Y_1-Y_2||_{p,\bar T}.
\end{align*}
Therefore, if 
\begin{equation}  \label{BarTsmallCondition}
c_2M_{\bar T} \sqrt{{\bar T}} \Big[ \sqrt{{\bar T}}+ \frac{pC_p^{\frac{1}{p}}(E)}{p-1} \Big]<1,
\end{equation}
  then $\mathcal Q$ is contraction in $\mathcal E_p(0,{\bar T})$.

Since  $\mathcal Q$ maps   $\mathcal E_p(0,{\bar T})$ into itself and is contraction with respect to the norm of $\mathcal E_p(0,{\bar T})$, $\mathcal Q$   has a unique fixed point $X\in \mathcal E_p(0,{\bar T}).$ This shows that $X(t)$ is a mild solution  to \eqref{E2} on $[0,{\bar T}]$. In view of \eqref{BarTsmallCondition}, this solution can be extended on $[0,T]$ by considering the equation on intervals $[0,\bar T], [\bar T, 2\bar T],\dots.$ Furthermore, the continuity of $X$ on $[0,T]$ follows from the continuity of the stochastic integral (see Proposition \ref{IntegralInequality}).

  {\bf Step 2.} Let us verify  uniqueness of the mild solution.
   Let $X_1$ and $X_2$ be  mild solutions of \eqref{E2}. Then for every $t\in[0,T]$ we have
\begin{align*}
&\mathbb E|X_1(t)-X_2(t)|^2\\
=&\mathbb E \Big|\int_0^t S(t-s)[F(s,X_1(s))-F(s,X_2(s))]ds \\
&+ \int_0^t S(t-s)[G(s,X_1(s))-G(s,X_2(s))]dw_s\Big|^2\\
\leq& 2 \mathbb E \Big|\int_0^t S(t-s)[F(s,X_1(s))-F(s,X_2(s))]ds\Big|^2\\
&+2 \mathbb E \Big|\int_0^t S(t-s)[G(s,X_1(s))-G(s,X_2(s))]dw_s\Big|^2\\
\leq& 2 M_T^2\mathbb E \Big[\int_0^t |F(s,X_1(s))-F(s,X_2(s))|ds\Big]^2\\
&+2c(E) \mathbb E \int_0^t |S(t-s)[G(s,X_1(s))-G(s,X_2(s))]|^2ds\\
\leq& 2 c_2^2M_T^2\mathbb E \Big[\int_0^t |X_1(s)-X_2(s)|ds\Big]^2\\
&+2c(E) M_T^2\mathbb E \int_0^t |G(s,X_1(s))-G(s,X_2(s))|^2ds\\
\leq &2c_2^2[t+c(E)] M_T^2\mathbb E \int_0^t |X_1(s)-X_2(s)|^2ds\\
\leq &2c_2^2[T+c(E)]M_T^2 \int_0^t \mathbb E|X_1(s)-X_2(s)|^2ds.
\end{align*}
The Gronwall  lemma then derives that 
$\mathbb E|X_1(t)-X_2(t)|^2=0 $ for every  $ t\in [0,T].$ Since $X_1$ and $X_2$ are continuous, they are indistinguishable.

  {\bf Step 3.} Let us finally verify the estimate \eqref{EEE}.
In view of \eqref{Eq30} and \eqref{Eq31}, we have 
\begin{align*}
&\sup_{t\in[0,t]} \mathbb E |X(s)|^p\\
=&||X||_{p,t}^p=||\mathcal Q(X)||_{p,t}^p\\
\leq& [||S(\cdot) \xi ||_{p,t}+ ||\mathcal Q_1(X) ||_{p,t}
+||\mathcal Q_2(X) ||_{p,t}]^p\\
\leq& 3^p[||S(\cdot) \xi ||_{p,t}^p+ ||\mathcal Q_1(X) ||_{p,t}^p
+||\mathcal Q_2(X) ||_{p,t}^p]\\
\leq& 3^p\Big[M_T^p\mathbb E|\xi |^p+  (c_1M_t)^p (2t)^{p-1} \mathbb E \int_0^{t} [1+|Y(s)|^p ]ds\\
&+ \Big(\frac{c_1pM_t}{p-1}\Big)^p c_p(E)  t^{\frac{p-2}{2}} 2^{p-1}  \int_0^t [1+\mathbb E|Y(s)|^p\Big] ds\\
\leq&3^pM_T^p\mathbb E|\xi |^p+ \Big [(3c_1M_t)^p (2t)^{p-1}+\Big(\frac{3c_1pM_t}{p-1}\Big)^p c_p(E)  t^{\frac{p-2}{2}} 2^{p-1} \Big] \\
& \times \int_0^{t} [1+\sup_{[0,s]}\mathbb E|Y(r)|^p ]ds\\
\leq&3^pM_T^p\mathbb E|\xi |^p+ \Big [(3c_1M_T)^p (2T)^{p-1}+\Big(\frac{3c_1pM_T}{p-1}\Big)^p c_p(E)  T^{\frac{p-2}{2}} 2^{p-1} \Big]  \\
& \times \Big[T+\int_0^{t} \sup_{[0,s]}\mathbb E|Y(r)|^p ds\Big], \hspace{2cm} t\in [0,T].
\end{align*}
Then  the Gronwall  lemma again provides \eqref{EEE}.
\end{proof}
\subsection{Existence and regular dependence on initial data of local mild solutions under local linear growth and local Lipschitz   conditions}
Let us  first explore existence and uniqueness of local mild solutions of \eqref{E2} under local linear growth and local Lipschitz   conditions on $F$ and $G$. Following ideas in \cite{Fried}, we shall construct two lemmas.
\begin{lemma} \label{lem1}
Let  $(\alpha_1,\alpha_2)\subset \mathbb R, \Omega_0\subset \Omega$ and $\Phi_i\in \mathcal N^2(\alpha_1,\alpha_2)$. 
If 
$$ {\bf1}_{\Omega_0}\Phi_1(t)={\bf1}_{\Omega_0}\Phi_2(t) \hspace{1cm} \text{ for all } t\in (\alpha_1,\alpha_2),$$
 then
$${\bf1}_{\Omega_0}\int_{\alpha_1}^{\alpha_2}\Phi_1(t)dw_t={\bf1}_{\Omega_0}\int_{\alpha_1}^{\alpha_2}\Phi_2(t)dw_t \quad\quad\quad\text{a.s.}$$
\end{lemma}
The proof of Lemma \ref{lem1} for the case $E=\mathbb R$ can be found in \cite[Lemma 2.11]{Fried}. In fact, the arguments are available for any $M$-type 2 separable Banach space. So we omit it.
\begin{lemma} \label{lem2}
Consider two equations of the form \eqref{E2}: 
\begin{equation} \label{E3}
\begin{cases}
dX_i+AX_idt=F_i(t,X_i)dt+ G_i(t,X_i)dw_t,\\
X_i(0)=\xi_i, \hspace{2cm} i=1,2.
\end{cases}
\end{equation}
Assume that there exists a constant $c>0$ such that
$$|F_i(t,x)-F_i(t,y)|+|G_i(t,x)-G_i(t,y)|\leq c|x-y|$$
and
$$|F_i(t,x)|^2+|G_i(t,x)|^2 \leq c^2 (1+|x|^2) $$
for $i=1,2, x,y\in E$ and $ t\in [0,T].$ 
Suppose further that 
$F_1(t,x)=F_2(t,x),$ $ G_1(t,x)=G_2(t,x)$  for $|x|\leq n, 0\leq t\leq T$  with some $n>0$,
 and that  $ \xi_1=\xi_2$ for a.s. $\omega$ for which either $\xi_1(\omega)<n$ or $\xi_2(\omega)<n$. Denote $\tau_i=\inf\{t: |X_i(t)|>n\}$ with the convention $\inf \emptyset=T.$ Then
$$\mathbb P(\tau_1=\tau_2)=1,$$
$$\mathbb P\{\sup_{0< t\leq \tau_1}|X_1(t)-X_2(t)|=0\}=1.$$
\end{lemma}
\begin{proof}
On the account of  Theorem \ref{linear growth and Lipschitz conditions}, there exists   a  mild solution $X_i(t)$ of \eqref{E3} on $[0,T]$.  Consider a function $\phi: [0,T] \to \mathbb R$  defined by  
\begin{equation*}
\phi(t)=
\begin{cases}
1 \text{   if    } |X_1(s)| \leq n \text{   for all  }  0\leq s\leq t,\\
0  \text{   in all other cases}.
\end{cases}
\end{equation*}
Then for $t\in [0,T]$ we have
\begin{equation*}
\begin{aligned}
\begin{cases}
\phi(t) (\xi_1-\xi_2)=0  &\quad\quad\text{ a.s., }\\
\phi(t)S(t)(\xi_1-\xi_2)=0 &\quad\quad\text{ a.s., }\\
\phi(t)={\bf 1}_{\{\phi(t)=1\}}\phi(t), &\\
 \phi(t)=\phi(t)^2. &
 \end{cases}
 \end{aligned}
\end{equation*}
Therefore,
\begin{equation*}
\begin{aligned}
\phi(t)[X_1(t)-X_2(t)]=&\phi(t)\int_0^t S(t-s) [F_1(s, X_1(s))-F_2(s,X_1(s))] ds \\
&+\phi(t)\int_0^t S(t-s) [F_2(s, X_1(s))-F_2(s,X_2(s))] ds \\
&+\phi(t)\int_0^t S(t-s) [G_1(s, X_1(s))-G_2(s,X_1(s))] dw_s\\
&+\phi(t)\int_0^t S(t-s) [G_2(s, X_1(s))-G_2(s,X_2(s))]dw_s\\
=&J_1+J_2+J_3+J_4.
\end{aligned}
\end{equation*}
When $\phi(t)=1$,  $F_1(s, X_1(s))=F_2(s,X_1(s))$ for every $s\in[0,t]$. Hence, $J_1=0$. In addition, we have $G_1(s, X_1(s))=G_2(s,X_1(s))$ on $\{\phi(t)=1\}$ for all $s\in[0,t]$. Lemma \ref{lem1}  then provides that 
\begin{align*}
J_3&={\bf1}_{\{\phi(t)=1\}}\phi(t)\int_0^t S(t-s) [G_1(s, X_1(s))-G_2(s,X_1(s))] dw_s\\
&={\bf1}_{\{\phi(t)=1\}}\int_0^t S(t-s) [G_1(s, X_1(s))-G_2(s,X_1(s))] dw_s\\
&={\bf1}_{\{\phi(t)=1\}}\int_0^t 0 dw_s=0.
\end{align*}
Hence,
\begin{align}
&\mathbb E\phi(t)|X_1(t)-X_2(t)|^2=\mathbb E|\phi(t)[X_1(t)-X_2(t)]|^2=\mathbb E|J_2+J_4|^2 \notag\\
&\leq 2\mathbb E|J_2|^2+2\mathbb E|J_4|^2. \label{Eq32}
\end{align}
Let us  estimate  $\mathbb E|J_2|^2$ and $\mathbb E|J_4|^2.$
Since $\phi(t)$  decreases in $t$, we have
\begin{equation*}
\begin{aligned}
2|J_2|^2&\leq 2\Big|\int_0^t \phi(s)S(t-s) [F_2(s, X_1(s))-F_2(s,X_2(s))] ds\Big|^2\\
&\leq 2t\int_0^t \phi(s)|S(t-s)[F_2(s, X_1(s))-F_2(s,X_2(s))]|^2 ds\\
&\leq 2tM_t^2\int_0^t \phi(s)| F_2(s, X_1(s))-F_2(s,X_2(s))|^2 ds\\
&\leq 2tc^2M_t^2\int_0^t \phi(s)|  X_1(s)-X_2(s)|^2 ds, \hspace{2cm} t\in [0,T].
\end{aligned}
\end{equation*}
Thus,
\begin{equation} \label{Eq33}
\mathbb E|J_2|^2\leq 2Tc^2 M_T^2 \int_0^t \mathbb E\phi(s)|  X_1(s)-X_2(s)|^2 ds, \hspace{2cm} t\in [0,T].
\end{equation}

On the other hand, we have $\phi(s)=1$ on $\{\phi(t)=1\} $ for every $s\in [0,t]$.  Lemma \ref{lem1} again provides that 
\begin{equation*}
\begin{aligned}
2|J_4|^2&=2\Big|\phi(t){\bf1}_{\{\phi(t)=1\}}\int_0^t S(t-s) [G_2(s, X_1(s))-G_2(s,X_2(s))] dw_s\Big|^2\\
&=2\Big|{\bf1}_{\{\phi(t)=1\}}\int_0^t S(t-s) [G_2(s, X_1(s))-G_2(s,X_2(s))] dw_s\Big|^2\\
&=2\Big|{\bf1}_{\{\phi(t)=1\}}\int_0^t \phi(s)S(t-s) [G_2(s, X_1(s))-G_2(s,X_2(s))] dw_s\Big|^2\\
&\leq 2\Big|\int_0^t \phi(s)S(t-s) [G_2(s, X_1(s))-G_2(s,X_2(s))] dw_s\Big|^2.
\end{aligned}
\end{equation*}
Using Proposition \ref{IntegralInequality}, we then obtain that
\begin{align} 
2\mathbb E|J_4|^2\leq &2c(E)\mathbb E \int_0^t |\phi(s)S(t-s) [G_2(s, X_1(s))-G_2(s,X_2(s))]|^2 ds\notag\\
\leq &2c(E)M_t^2\mathbb E \int_0^t \phi(s)|G_2(s, X_1(s))-G_2(s,X_2(s))|^2 ds\notag\\
\leq &2c^2 c(E) M_T^2 \int_0^t \mathbb E\phi(s)|  X_1(s)-X_2(s)|^2 ds,
\hspace{2cm} t\in [0,T]. \label{Eq34}
\end{align}
Substituting \eqref{Eq33} and \eqref{Eq34} into \eqref{Eq32}, we observe  that
\begin{align*}
\mathbb E\phi(t)&|X_1(t)-X_2(t)|^2  \\
\leq & 2c^2[T+c(E)]M_T^2\int_0^t \mathbb E\phi(s)|  X_1(s)-X_2(s)|^2 ds, \hspace{1cm} t\in [0,T].
\end{align*}
Thanks to the Gronwall  lemma, we verify that  $\mathbb E\phi(t)|X_1(t)-X_2(t)|^2=0$ for every $t\in[0,T].$ As a consequence, 
$$\phi(t)|X_1(t)-X_2(t)|^2=0 \quad\quad\quad \text{ a.s.} $$
From the definition of $\phi$, it is then clear that $X_1(t)=X_2(t)$ a.s. $t\in (0,\tau_1]$ and  $\mathbb P(\tau_2\geq \tau_1)=1.$ Similarly, we have $X_1(t)=X_2(t)$ a.s. $t\in (0,\tau_2]$ and  $\mathbb P(\tau_1\geq \tau_2)=1.$ We thus have shown that 
$$\mathbb P\{X_1(t)=X_2(t)\} =\mathbb P(\tau_1= \tau_2)=1, \hspace{2cm} t\in (0,\tau_1]. $$
In addition, by  the continuity of $X_1(t)$ and $X_2(t)$,  we  conclude that $$\sup_{0< t\leq \tau_1} |X_1(t)-X_2(t)|^2=0 \hspace{2cm} \text{ a.s. }$$
It completes the proof.
\end{proof}
\begin{theorem}[local existence] \label{local existence theorem}
Suppose that for any $n>0$ there exist $c_n>0$ and $\bar c_n>0$ such that whenever $ |x|\leq n, |y|\leq n$ and $ t\in [0,T],$ the following two conditions hold:
\begin{itemize}
\item [{\rm (i)}]  The local growth condition
$$
|F(t,x)|+|G(t,x)| \leq \bar c_n (1+|x|).
$$
\item [{\rm (ii)}]  The local Lipschitz condition
\begin{equation}  \label{The local Lipschitz condition}
 |F(t,x)-F(t,y)|+|G(t,x)-G(t,y)| \leq c_n |x-y|.
\end{equation}
\end{itemize}
 Then there exists a unique maximal local mild solution $\{X(t),t\in[0,\tau)\}$ to \eqref{E2}. Furthermore, there exists a constant $\alpha=\alpha(\bar c_n, M_T, T)>0$ depending only on $\bar c_n, M_T$ and $T$ such that 
\begin{equation} \label{XtMinTaunEstimate}
\mathbb E |X(t\wedge \tau_n)|^2\leq \alpha(1+\mathbb E|\xi|^{2}), \hspace{1cm} t\geq 0, n=0,1,\dots,
\end{equation}
where $\{\tau_n\}_{n=0}^\infty$ is a sequence of stopping times defined by
$$
\tau_n=\inf\{t\in[0,T]: |X(t)|>n\}
$$
 with the convention $\inf\emptyset=T$ and $\tau_0=0$ a.s. 
\end{theorem}
\begin{proof}
Let us first show existence of a maximal local mild solution by using the truncation method. For $n=1,2,\dots,$ we denote
\begin{equation*} 
\begin{aligned}
F_n(t,x)=&
\begin{cases}
\begin{aligned}
&F(t,x)       &      \hspace{1cm}  \text{if    }& |x|\leq n,\\
&F(t,x)(2-\frac{|x|}{n})   &    \text{if    } &n<|x|\leq 2n,\\
&0               &      \text{if    }& |x|> 2n,
\end{aligned}
\end{cases}\\
G_n(t,x)=&
\begin{cases}
\begin{aligned}
&G(t,x)                \hspace*{2cm}&  \text{     if    }& |x|\leq n,\\
&G(t,x)(2-\frac{|x|}{n})        & \text{   if    }& n<|x|\leq 2n,\\
&0                   & \text{     if    }& |x|> 2n,
\end{aligned}
\end{cases}\\
\end{aligned}
\end{equation*}
and
\begin{equation*} 
\begin{aligned}
\xi_{n_0}=&
\begin{cases}
\xi            \hspace*{1cm}\text{     if    } |\xi|\leq n,\\
0                  \hspace*{1cm}\text{     if    } |\xi|> n.
\end{cases}
\end{aligned}
\end{equation*}
It is easy to see that  the measurable functions $F_n$ and $G_n$ satisfy the linear growth and global Lipshitz conditions. On the account of  Theorem \ref{linear growth and Lipschitz conditions}, there exists a unique mild solution $X_n(t)$ on $[0,T]$ of the system
\begin{equation} \label{H2}
\begin{cases}
dX_n+AX_ndt=F_n(t,X_n)dt+G_n(t,X_n)dw_t,\\
X_n(0)=\xi_{n_0}.
\end{cases}
\end{equation}

Define the stopping times
\begin{equation}  \label{stoppingTimes}
\tau_n=\inf\{t\in[0,T]: |X_n(t)|>n\}, \hspace{2cm} n=1,2,\dots.
\end{equation}
By Lemma \ref{lem2}, we  observe that
 $$X_n(t)=X_m(t) \hspace{1cm} \text{ a.s. if } 0\leq t \leq \tau_n \text{ and }m>n.$$ 
 Hence, the sequence $\{\tau_n\}_n$ increases and has a limit $\tau=\lim_{n\to\infty} \tau_n\leq T$ a.s. We then  define $\{X(t), 0\leq t<\tau\}$ by
\begin{equation}  \label{DefintionOfX(t)}
X(t)=X_n(t), \hspace{2cm} t\in[\tau_{n-1},\tau_n], n\geq 1.
\end{equation}  
It is clear that $X(t)$ is continuous and $X\in \mathcal N^2(0,t\wedge \tau_n)$ for every $t\geq 0, n\geq 1$. In addition, on $\{\tau<T\}$ we have
$$\liminf_{t\to\tau}|X(t)|\geq \liminf_{n\to\infty}|X(\tau_n)|=\liminf_{n\to\infty} |X_n(\tau_n)|=\infty.$$

On the other hand, using Lemma \ref{lem1},  we obtain that  
\begin{align}
X(t\wedge \tau_n)=&X_n(t\wedge \tau_n)\notag\\
=&S(t\wedge \tau_n)\xi_{n_0}+\int_0^{t\wedge \tau_n} S(t\wedge \tau_n-s) F_n(s,X_n(s))ds\notag\\
&+\int_0^{t\wedge \tau_n} S(t\wedge \tau_n-s) G_n(s,X_n(s))dw_s\notag\\
=&S(t\wedge \tau_n)\xi_{n_0}+\int_0^{t\wedge \tau_n} S(t\wedge \tau_n-s) F(s,X(s))ds\label{XtMinTaunEquation}
\\
&+\int_0^{t\wedge \tau_n} S(t\wedge \tau_n-s) G(s,X(s))dw_s.   \notag\end{align}
Therefore, $\{X(t), t\in [0,\tau)\}$ is a maximal local mild solution of \eqref{E2}. 

Let us now verify the estimate \eqref{XtMinTaunEstimate}.
Using the same argument as in the proof of  the estimate  \eqref{EEE} in Theorem \ref{linear growth and Lipschitz conditions} to the stochastic integral equation    \eqref{XtMinTaunEquation},  we conclude that
$$\mathbb E |X(t\wedge \tau_n)|^2\leq \alpha(1+\mathbb E|\xi_{n_0}|^{2})\leq \alpha(1+\mathbb E|\xi|^{2}),\hspace{1cm} t\geq 0, n\geq 1,$$
where $\alpha=\alpha(\bar c_n, M_T, T)>0$ is some constant  depending only on $\bar c_n, M_T$ and $T.$ Clearly, by \eqref{stoppingTimes} and \eqref{DefintionOfX(t)},
$$
\tau_n=\inf\{t\in[0,T]: |X(t)|>n\}.
$$
The estimate \eqref{XtMinTaunEstimate} thus has been verified.

Let us finally verify  uniqueness of the solution. Let  $\{\bar X(t), t\in \mathcal [0,\bar \tau)\}$  be another maximal local mild solution of \eqref{E2}. Denote 
$$\bar \tau_n= \inf\{t\in [0, \bar\tau): |\bar X(t)|> n\} \quad \text{ and } \quad 
\tau_n^*=\tau_n\wedge \bar \tau_n.$$
 Then the sequence $\{\tau_n^*\}_n $ increases and converges to $\tau\wedge \bar \tau$ a.s. as $n\to \infty$. 
For  $t\geq 0$ and $n=1,2,\dots,$ we have 
\begin{align}
\mathbb E&|X(t\wedge \tau_n^*)-\bar X(t\wedge \tau_n^*)|^2\notag\\
\leq&2\mathbb E \Big|\int_0^{t\wedge \tau_n^*} S(t\wedge \tau_n^*-s)\{F(s,X(s))-F(s,\bar X(s))\}ds\Big|^2\notag\\
&+2\mathbb E \Big|\int_0^{t\wedge \tau_n^*} S(t\wedge \tau_n^*-s)\{G(s,X(s))-G(s,\bar X(s))\}dw_s\Big|^2\notag\\
\leq&2TM_T^2\mathbb E \int_0^{t\wedge \tau_n^*} |F(s,X(s))-F(s,\bar X(s))|^2ds\notag\\
&+2c(E)\mathbb E \int_0^{t\wedge \tau_n^*} |S(t\wedge \tau_n^*-s)\{G(s,X(s))-G(s,\bar X(s))\}|^2 ds\notag\\
\leq &2Tc_n^2M_T^2  \mathbb E\int_0^{t\wedge \tau_n^*} |X(s)-\bar X(s)|^2ds\notag\\
&+2c(E)M_T^2\mathbb E \int_0^{t\wedge \tau_n^*} |G(s,X(s))-G(s,\bar X(s))|^2 ds\notag\\
\leq &2(T+c(E)) c_n^2M_T^2\mathbb E \int_0^{t\wedge \tau_n^*}  |X(s)-\bar X(s)|^2ds\notag\\
= &2(T+c(E)) c_n^2M_T^2\mathbb E \int_0^t {\bf1}_{[0, \tau_n^*]}(s)  |X(s\wedge \tau_n^*)-\bar X(s\wedge \tau_n^*)|^2ds\notag\\
\leq &2(T+c(E)) c_n^2M_T^2\int_0^t \mathbb E   |X(s\wedge \tau_n^*)-\bar X(s\wedge \tau_n^*)|^2ds.  \label{H4}
\end{align}
The Gronwall  lemma then gives that $\mathbb E|X(t\wedge \tau_n^*)-\bar X(t\wedge \tau_n^*)|^2=0.$  Hence, $X(t\wedge \tau_n^*)=\bar X(t\wedge \tau_n^*)$ for every $n\geq1$ and  $t\geq 0.$ Letting $n\to \infty$, we conclude that  $X(t)=\bar X(t)$ for every   $t\in [ 0, \tau\wedge \bar \tau).$

On the other hand, if $\mathbb P\{\tau<\bar \tau\}<1$ then for almost sure  $\omega\in \{\tau<\bar \tau\}$, $\bar X(t,\omega)$ is continuous at $\tau(\omega)$. We then arrive at a contradiction
$$\infty>|\bar X(\tau(\omega),\omega)|=\lim_{n\to\infty} |\bar X(\tau_n(\omega),\omega)|=\lim_{n\to\infty} | X(\tau_n(\omega),\omega)|=\infty.$$
Therefore, $\mathbb P  \{\tau<\bar \tau\}=1$. Similarly, we have $\mathbb P  \{\tau>\bar \tau\}=1.$ We thus have shown that $\mathbb P  \{\tau=\bar \tau\}=1.$ The theorem is proved.
\end{proof}

Let us now study dependence of the maximal local mild solution on the initial data. It turns out that the maximal  local mild solution $\{X(t),t\in[0,\tau)\}$ of  \eqref{E2}, which is shown in Theorem \ref{local existence theorem}, depends continuously on the initial data in the sense specified  in the following theorem.
\begin{theorem} \label{dependence of local mild solution on the initial data}
Let the assumption on the functions $F$ and $G$  of Theorem \ref{local existence theorem} be satisfied. Let $\{X(t),t\in[0,\tau)\}$ and $\{\bar X(t),t\in[0,\bar\tau)\}$ be the maximal local solutions of  \eqref{E2} with the initial values $\xi$ and $\bar \xi$, respectively. Then there exist two positive constants $C_1=C_1( c_n, M_T, T)$ and $C_2=C_2( c_n, \bar c_n, M_T, T)$   satisfying the  estimates
\begin{align}
&\mathbb E|X(t\wedge \tau_n)-\bar X(t\wedge \tau_n)|^2 \leq C_1 \mathbb E|\xi-\bar\xi|^2 \label{Eq49}
\end{align}
and 
\begin{align}
&\mathbb E|X(t\wedge \tau_n)- X(s\wedge \tau_n)|^2\label{Eq50}\\
& \leq C_2 [\mathbb E|[S(t\wedge \tau_n-s\wedge \tau_n)-I]X(s\wedge \tau_n)|^2 +(1+ \mathbb E|\xi|^{2}) (t-s)]\notag
\end{align}
for every $0\leq s\leq t\leq T$,  
where $\{\tau_n\}_{n=0}^\infty$ is a sequence of stopping times defined by
$$
\tau_n=\inf\{t\in[0,T]: |X(t)|>n \text{ or } |\bar X(t)|>n\}.
$$
\end{theorem}
\begin{proof}
Let us first prove \eqref{Eq49}. The case $\mathbb E|\xi-\bar\xi|^2=\infty$ is obvious. Hence, we may assume that $\mathbb E|\xi-\bar\xi|^2<\infty.$ In view of the proof of Theorem \ref{local existence theorem} (see \eqref{XtMinTaunEquation}), we have 
\begin{align}
X(t\wedge \tau_n)=&S(t\wedge \tau_n)\xi+\int_0^{t\wedge \tau_n} S(t\wedge \tau_n-r) F(r,X(r))dr \label{Eq52}\\
&+\int_0^{t\wedge \tau_n} S(t\wedge \tau_n-r) G(r,X(r))dw_r,  \notag\end{align}
and 
\begin{align*}
\bar X(t\wedge \tau_n)=&S(t\wedge \tau_n)\bar\xi+\int_0^{t\wedge \tau_n} S(t\wedge \tau_n-r) F(r,\bar X(r))dr\\
&+\int_0^{t\wedge \tau_n} S(t\wedge \tau_n-r) G(r,\bar X(r))dw_r.  \notag\end{align*}
Similarly to \eqref{H4}, we obtain that
\begin{align*}
\mathbb E&|X(t\wedge \tau_n)-\bar X(t\wedge \tau_n)|^2\notag\\
 \leq &3\mathbb E|S(t\wedge \tau_n)(\xi-\bar\xi)|^2\notag\\
 &+3(T+c(E)) c_n^2M_T^2\int_0^t \mathbb E   |X(r\wedge \tau_n)-\bar X(r\wedge \tau_n)|^2dr  \notag\\
 \leq &3M_T^2\mathbb E|\xi-\bar\xi|^2\notag\\
 &+3(T+c(E)) c_n^2M_T^2\int_0^t \mathbb E   |X(r\wedge \tau_n)-\bar X(r\wedge \tau_n)|^2dr, 
 \hspace{1cm} t\in[0,T].
\end{align*}
The Gronwall lemma then provides \eqref{Eq49}.

Let us now verify \eqref{Eq50}. By the semigroup property, from \eqref{Eq52} we observe that
\begin{align*}
X&(t\wedge \tau_n)-X(s\wedge \tau_n)\notag\\
=&S(t\wedge \tau_n-s\wedge \tau_n)\Big[S(s\wedge \tau_n)\xi+\int_0^{s\wedge \tau_n} S(s\wedge \tau_n-r) F(r,X(r))dr\\
&+\int_0^{s\wedge \tau_n} S(s\wedge \tau_n-r) G(r,X(r))dw_r\Big]+\int_{s\wedge \tau_n}^{t\wedge \tau_n} S(t\wedge \tau_n-r) F(r,X(r))dr\\
&+\int_{s\wedge \tau_n}^{t\wedge \tau_n} S(t\wedge \tau_n-r) G(r,X(r))dw_r-X(s\wedge \tau_n)\notag\\
=&[S(t\wedge \tau_n-s\wedge \tau_n)-I]X(s\wedge \tau_n)+\int_{s\wedge \tau_n}^{t\wedge \tau_n} S(t\wedge \tau_n-r) F(r,X(r))dr\\
&+\int_{s\wedge \tau_n}^{t\wedge \tau_n} S(t\wedge \tau_n-r) G(r,X(r))dw_r.
\end{align*}
Using the local growth condition on $F$ and $G$, we then observe that 
\begin{align*}
\mathbb E&|X(t\wedge \tau_n)-X(s\wedge \tau_n)|^2\notag\\
\leq &3\mathbb E|[S(t\wedge \tau_n-s\wedge \tau_n)-I]X(s\wedge \tau_n)|^2\notag\\
&+3\mathbb E\Big|\int_{s\wedge \tau_n}^{t\wedge \tau_n} S(t\wedge \tau_n-r) F(r,X(r))dr|^2\notag\\
&+3\mathbb E\Big|\int_{s\wedge \tau_n}^{t\wedge \tau_n} S(t\wedge \tau_n-r) G(r,X(r))dw_r\Big|^2\notag\\
\leq &3\mathbb E|[S(t\wedge \tau_n-s\wedge \tau_n)-I]X(s\wedge \tau_n)|^2\notag\\
&+3\mathbb E(t\wedge \tau_n-s\wedge \tau_n)\int_{s\wedge \tau_n}^{t\wedge \tau_n} |S(t\wedge \tau_n-r)|^2 |F(r,X(r))|^2dr\notag\\
&+3c(E)\mathbb E\int_{s\wedge \tau_n}^{t\wedge \tau_n} |S(t\wedge \tau_n-r)|^2|G(r,X(r))|^2dr\notag\\
\leq &3\mathbb E|[S(t\wedge \tau_n-s\wedge \tau_n)-I]X(s\wedge \tau_n)|^2\notag\\
&+6c_n^2M_T^2T\mathbb E\int_{s\wedge \tau_n}^{t\wedge \tau_n}  [1+|X(r)|^2]dr\notag\\
&+6c_n^2M_T^2c(E)\mathbb E\int_{s\wedge \tau_n}^{t\wedge \tau_n}  [1+|X(r)|^2]dr\notag\\
= &3\mathbb E|[S(t\wedge \tau_n-s\wedge \tau_n)-I]X(s\wedge \tau_n)|^2\notag\\
&+6c_n^2M_T^2[T+c(E)]\int_s^t  [1+\mathbb E|X(r\wedge \tau_n)|^2]dr.\notag
\end{align*}
By virtue of  \eqref{XtMinTaunEstimate}, we obtain that 
\begin{align*}
\mathbb E&|X(t\wedge \tau_n)-X(s\wedge \tau_n)|^2\notag\\
\leq &3\mathbb E|[S(t\wedge \tau_n-s\wedge \tau_n)-I]X(s\wedge \tau_n)|^2\notag\\
&+6 c_n^2M_T^2[T+c(E)] (1+\alpha+\alpha \mathbb E|\xi|^{2}) (t-s),\notag
\end{align*}
 where $\alpha=\alpha(\bar c_n, M_T, T)$ is some positive constant. Thus, \eqref{Eq50} has been verified. 
\end{proof}
\subsection{Existence and regular dependence on initial data of global mild solutions under linear growth and local Lipschitz   conditions}
Let us  first show existence and uniqueness of mild solutions of \eqref{E2} under linear growth and local Lipschitz conditions on $F$ and $G$. 
\begin{theorem}[global existence]\label{thm1}
Suppose that  $F$ and $G$ satisfies the linear growth condition \eqref{The linear growth condition} and the local Lipschitz condition \eqref{The local Lipschitz condition} and that 
 $\mathbb E|\xi|^2< \infty$. Then 
\begin{itemize}
  \item [{\rm (i)}]  there exists a unique mild solution $X(t)$ of \eqref{E2} on $[0,T]$ such that
    \begin{equation}  \label{EX(t)2Estimate}
     \mathbb E |X(t)|^2 \leq \alpha_1(1+\mathbb E|\xi|^2), \hspace{1.5cm}  t\in [0,T],
     \end{equation}
     where $\alpha_1$ is some constant depending only on $c_1, M_T$ and $ T.$
\item [{\rm (ii)}]  For any $p> 2$ there exists a constant $\alpha_2>0$ depending only on $c_1, p, M_T$ and $ T$ such that 
\begin{equation}  \label{EEEglobal existence}
\mathbb E\sup_{0\leq t\leq T} |X(t)|^{p} \leq \alpha_2(1+\mathbb E|\xi|^{p}).
\end{equation}
\end{itemize}
\end{theorem}
\begin{proof}
\text{}
\begin{itemize}
\item Proof of  (i).
\end{itemize}
It is already known by Theorem \ref{local existence theorem} that  there exists a unique maximal local mild solution $\{X(t),t\in [0,\tau)\}$ to \eqref{E2} such that 
\begin{align*} 
\begin{cases}
\mathbb E |X(t\wedge \tau_k)|^2\leq \alpha_1(1+\mathbb E|\xi|^{2}), &\hspace{1cm}t\geq 0, k\geq 1, \\
X(\tau_{k})=k, &\hspace{1cm} k\geq 1,
\end{cases}
\end{align*}
where $\alpha_1>0$ is  some constant  depending only on $c_1, M_T$ and $ T.$ 

Let us first verify that $\tau=T$ a.s. Indeed, if this statement is false, then there would exist $t_0\in(0,T)$ and $\epsilon \in (0,1)$  such that $\mathbb P\{\tau<t_0\}>\epsilon.$ Hence, by denoting $\Omega_k=\{\tau_k\leq t_0\}$, there exists $k_0\in \mathbb N\setminus \{0\}$ such that 
$\mathbb P(\Omega_k)\geq \epsilon $ for all $ k\geq k_0.$ Since the sequence $\{\Omega_k\}_k$ is decreasing, we observe that  $\mathbb P (\cap_{k=k_0} \Omega_k)\geq\epsilon.$ From this, for every $k\geq k_0$  we have
\begin{align*}
\alpha_1(1+\mathbb E|\xi|^{2}) &\geq  \mathbb E |X(t_0\wedge \tau_k)|^2\\
&\geq  \mathbb E |{\bf 1}_{\cap_{k=k_0} \Omega_k}X(t_0\wedge \tau_k)|^2\\
&\geq  \mathbb E |{\bf 1}_{\cap_{k=k_0} \Omega_k}X(\tau_k)|^2\\
&=k^2  \mathbb P \{\cap_{k=k_0} \Omega_k\}=\epsilon k^2.
\end{align*}
Letting $k\to \infty$, we arrive at a contradiction: $\infty>\alpha_1(1+\mathbb E|\xi|^{2})=\infty.$
Thus, $\tau=T$ a.s. 

Let us now show that $X(t)$ is defined and is continuos at $t=T$ and that $X$ satisfies the estimate \eqref{EX(t)2Estimate}.
Since $\tau=T$ a.s., we have 
\begin{equation} \label{Eq-2}
X(t)=S(t) \xi+ \mathcal Q_1(X)(t) +\mathcal Q_2(X)(t), \hspace{1cm} t\in [0,T),
\end{equation}
where $\mathcal Q_1$ and $\mathcal Q_2$ are defined by \eqref{Q1Definition} and \eqref{Q2Definition}, respectively.
Using the same argument as in  the proof of the estimate \eqref{EEE} of Theorem \ref{linear growth and Lipschitz conditions} to the stochastic integral equation \eqref{Eq-2}, we observe that 
\begin{equation} \label{Eq-3}
\mathbb E|X(t)|^2 \leq \alpha_1 (1+\mathbb E |\xi|^2), \hspace{1cm} t\in [0,T),
\end{equation}
where $\alpha_1$ is some constant depending only on $c_1, M_T$ and $ T$. Clearly, the estimate \eqref{Eq-3} derives that $\mathcal Q_1(X)(t)$ and $\mathcal Q_2(X)(t)$  are defined at $t=T$. Furthermore,  Proposition \ref{IntegralInequality} provides the continuity of $\mathcal Q_2(X)$ at $t=T$. The continuity of $\mathcal Q_1(X)$ at $t=T$ can be seen easily from \eqref{Eq-3} and the equality
\begin{align*}
&\mathcal Q_1(X)(T)-\mathcal Q_1(X)(t)\\
=&\int_0^T S(T-t)S(t-s) F(s,X(s))ds-\int_0^tS(t-s)F(s,X(s))ds\\
=& [S(T-t)-I]\int_0^tS(t-s)F(s,X(s))ds +\int_t^TS(T-s)F(s,X(s))ds
\end{align*}
with any $ t\in [0,T)$. Therefore, by setting 
$$X(T)=S(T) \xi+ \mathcal Q_1(X)(T) +\mathcal Q_2(X)(T),$$
we conclude that  the obtained process $\{X(t), t\in [0,T]\}$  is a unique mild solution of \eqref{E2} on $[0,T]$. Furthermore, the estimate \eqref{EX(t)2Estimate} follows from  \eqref{Eq-3} and the continuity of $X(t)$ on $[0,T]$.
\begin{itemize}
\item Proof of (ii).
\end{itemize}
 The case $\mathbb E|\xi|^p=\infty$ is obvious. Therefore, we can assume that  $\mathbb E|\xi|^p<\infty.$ To simply the proof, we shall use a  notation $C$ to denote positive constants which are determined by the constants  $c_1, p, M_T$ and $T$. So, it may change from occurrence to occurrence. Since 
$$X(t)=S(t) \xi+ \mathcal Q_1(X)(t) +\mathcal Q_2(X)(t),$$
for every $s\in [0,T]$ we have 
\begin{equation*}
\begin{aligned}
|X(s)|^{p}
\leq &C\Big[|\xi|^{p}+\Big\{\int_0^s|S(s-u) F(u, X(u))|du\Big\}^{p}+|\mathcal Q_2(X)(s)|^{p}\Big]\\
\leq&C\Big[|\xi|^{p}+\Big\{\int_0^s| F(u, X(u))|du\Big\}^{p}+|\mathcal Q_2(X)(s)|^{p}\Big]\\
\leq&C\Big[|\xi|^{p}+\Big\{\int_0^s (1+| X(u)|)du\Big\}^{p}+ |\mathcal Q_2(X)(s)|^{p}\Big]\\
\leq&C\Big[|\xi|^{p}+\int_0^s (1+| X(u)|^{p})du+|\mathcal Q_2(X)(s)|^{p}\Big].
\end{aligned}
\end{equation*}
Hence,
\begin{equation*}
\begin{aligned}
\mathbb E&\sup_{0\leq s\leq t}|X(s)|^{p} \\
&\leq C\Big[\mathbb E |\xi|^{p}+\int_0^t (1+\mathbb E\sup_{0\leq u\leq s}| X(u)|^{p})ds+\mathbb E\sup_{0\leq s\leq t}|\mathcal Q_2(X)(s)|^{p}\Big].
\end{aligned}
\end{equation*}

On the other hand, applying  the Burkholder-Davis-Gundy inequality \cite[Theorem 3.28]{IS} to  the continuous  martingale $\{\mathcal Q_2(X)(s),s\in [0,T]\}$, we see that 
\begin{equation*}
\begin{aligned}
\mathbb E\sup_{0\leq s\leq t} |\mathcal Q_2(X)(s)|^{p}&\leq C \mathbb E \Big[\int_0^t|S(s-u)G(u,X(u))|^2du\Big ]^{\frac{p}{2}}\\
&\leq C \mathbb E \Big[\int_0^t|G(u,X(u))|^2du\Big ]^{\frac{p}{2}}\\
&\leq  C \mathbb E \Big[\int_0^t(1+|X(u)|^2)du\Big ]^{\frac{p}{2}}\\
&\leq  C\int_0^t[1+\mathbb E\sup_{0\leq u\leq s}| X(u)|^{p}]ds.
\end{aligned}
\end{equation*}
Therefore, we have shown that
$$\mathbb E\sup_{0\leq s\leq t}|X(s)|^{p}\leq C\Big [1+ \mathbb E |\xi|^{p}+\int_0^t \mathbb E\sup_{0\leq u\leq s}| X(u)|^{p}ds\Big].$$
By the Gronwall  lemma, we conclude that
$$\mathbb E\sup_{0\leq s\leq T}|X(s)|^{p}\leq  C(1+\mathbb E|\xi|^{p}).$$
The proof is complete.
\end{proof}
In the remain of the subsection, we will give the continuous dependence of the global mild solution on the initial data. 
\begin{theorem} \label{dependence of global mild solution on the initial data}
Let \eqref{The linear growth condition} and \eqref{The local Lipschitz condition}  be satisfied.  Let $X$ and $\bar X$ be the global solutions of  \eqref{E2} with the initial values $\xi$ and $\bar \xi$ satisfying the condition $\mathbb E|\xi|^2+\mathbb E|\bar \xi|^2< \infty$, respectively. Then there exist two positive constants $C_1=C_1( c_1, M_T, T)$ and $C_2=C_2( c_1, \bar c_n, M_T, T)$   such that 
\begin{align*}
&\mathbb E|X(t)-\bar X(t)|^2 \leq C_1 \mathbb E|\xi-\bar\xi|^2 
\end{align*}
and 
\begin{align*}
&\mathbb E|X(t)- X(s)|^2 \leq C_2 [\mathbb E|[S(t-s)-I]X(s)|^2 +(1+ \mathbb E|\xi|^{2}) (t-s)]\notag
\end{align*}
for every $0\leq s\leq t\leq T$. 
\end{theorem}
As the proof of this theorem is similar to that of Theorem \ref{dependence of local mild solution on the initial data}, we may omit it.
\section{Linear evolution equations with additive noise} \label{section4}
This section handles linear evolution equations with additive noise: the functions $F$ and $G$ are considered to be dependent only on $t$, i.e. $F(t,X)=F(t)$ and $ G(t,X)=G(t).$ 
Let us rewrite the equation  \eqref{E2}  in the form
\begin{equation} \label{autonomous linear evolution equation}
\begin{cases}
dX+AXdt=F(t)dt+ G(t)dw_t,\\
X(0)=\xi.
\end{cases}
\end{equation}
We shall explore existence of strict and mild solutions and  regularity of solutions of \eqref{autonomous linear evolution equation}. 

Throughout this section, we suppose that 
\begin{itemize}
  \item  the spectrum $\sigma(A)$ of $A$ is  contained in an open sectorial domain $\Sigma_{\varpi}$: 
\begin{equation} \label{spectrumsectorialdomain} 
\sigma(A) \subset  \Sigma_{\varpi}=\{\lambda \in \mathbb C: |\arg \lambda|<\varpi\}, \quad \quad 0<\varpi<\frac{\pi}{2},
       \end{equation}
  \item there exists a constant $M_{\varpi}>0$ such that
\begin{equation} \label{resolventnorm}
          ||(\lambda-A)^{-1}|| \leq \frac{M_{\varpi}}{|\lambda|}, \quad\quad\quad \quad   \lambda \notin \Sigma_{\varpi}.
     \end{equation}
    \item  the non-random functions $F\colon[0,T]\to E$ and $G\colon[0,T]\to E$  satisfy one of the following two conditions:
       \begin{equation}  \label{FGSpace1}
       F\in \mathcal F^{\beta, \sigma}((0,T];E), G\in \mathcal F^{\beta+\frac{1}{2}, \sigma} ((0,T];E)  
              \end{equation}
            \hspace{2.5cm}  with  $0<\sigma<\beta\leq \frac{1}{2},$
       \begin{equation}  \label{FGSpace2}
       F\in \mathcal F^{\beta, \sigma} \, \text{ with } 0<\sigma<\beta\leq 1 \text{  and } G\in \mathcal B([0,T];E). 
              \end{equation} 
 \end{itemize} 
  
The following lemma presents some useful results. The proof can be found  in the monograph \cite{yagi}.
\begin{lemma}
The linear operator $A$  satisfies the following properties.
\begin{itemize}
\item  [\rm (i)] $(-A)$ generates an analytic semigroup $S(t)=e^{-tA}$.
\item  [\rm (ii)]
\begin{equation} \label{EstimateIofAnu}
|A^\nu S(t)| \leq \iota_\nu t^{-\nu}, \hspace{2cm}  t\in (0,T],  \nu \in [0,\infty),
\end{equation}
 where $ \iota_\nu=\sup_{0<t\leq T} t^\nu |A^\nu S(t)|<\infty$. In particular, 
\begin{equation} \label{EstimateIofS(t)Maximum}
|S(t)|\leq \iota_0, \hspace{2cm}  t\in[0,T].
\end{equation}
 \item  [\rm (iii)] There exists $\nu>0$ such that 
\begin{equation} \label{EstimateIofS(t)}
|S(t)|\leq \iota_0 e^{-\nu t}, \hspace{2cm}  t\in[0,T].
\end{equation}
\item  [\rm (iv)] For every $\theta\in (0,1]$ 
\begin{equation} \label{EstimateIofS(t)-IA-theta}
|[S(t)-I]A^{-\theta}|\leq \frac{\iota_{1-\theta}}{\theta} t^\theta, \hspace{2cm}  t\in[0,T].
\end{equation}
\item  [\rm (v)]  
\begin{equation} \label{Eq53}
AS(\cdot)U \in \mathcal F^{\beta,\sigma}((0,T];E) \hspace{2cm} \text{ for every }U\in \mathcal D(A^\beta).
\end{equation}
\end{itemize}
\end{lemma}
\subsection{Existence of strict solutions}
This subsection presents existence of strict solutions of  the autonomous linear evolution equation \eqref{autonomous linear evolution equation}.
\begin{theorem} \label{StrictSolutions}
 Assume that 
 \begin{equation} \label{Eq35}
 \begin{aligned}
 &\text{there exist } \delta\in (0,\frac{1}{2}) \text{  and } c_{\delta}>0 \text{  such that  }|AS(t)| <c_{\delta} t^{-\delta}\\
&  \text{for every } t\in (0,T].
\end{aligned}
\end{equation}
 Then there exists a unique strict solution of \eqref{autonomous linear evolution equation} in the space $\mathcal C((0,T];\mathcal D(A))$.
\end{theorem}
\begin{proof}
Let us observe that  the integral $\int_0^t S(t-s)G(s)dw_s$ is well-defined and  is continuous. Indeed, it suffices to show that $\int_0^t |S(t-s)G(s)|^2ds$ is finite for $t\in (0,T]$.
If \eqref{FGSpace1}  takes place then  from  \eqref{FbetasigmaSpaceProperty} and \eqref{EstimateIofS(t)Maximum}, 
  we have
\begin{align*}
\int_0^t|S(t-s)G(s)|^2 ds& \leq  \int_0^t \iota_0^2 |G|_{\mathcal F^{\beta+\frac{1}{2}, \sigma}}^2 s^{2\beta-1}ds\\
&=\frac{\iota_0^2 |G|_{\mathcal F^{\beta+\frac{1}{2}, \sigma}}^2 t^{2\beta}}{2\beta}<\infty, \hspace{2cm}  t\in [0,T].
\end{align*}
Otherwise, if  \eqref{FGSpace2} takes place then from \eqref{EstimateIofS(t)Maximum}, we have
$$\int_0^t|S(t-s)G(s)|^2 ds \leq  \iota_0^2 t |G|_{\mathcal B((0,T]; E)}^2  <\infty, \quad \quad\quad t\in [0,T].$$

On the other hand, by \eqref{FbetasigmaSpaceProperty} and \eqref{EstimateIofS(t)}, we have
\begin{equation} \label{Eq0}
\int_0^t|S(t-s) F(s)|ds\leq \iota_0 |F|_{\mathcal F^{\beta,\sigma}} \int_0^t e^{-\nu (t-s)} s^{\beta-1}ds \leq \frac{\iota_0 |F|_{\mathcal F^{\beta,\sigma}} t^\beta }{\beta}.
\end{equation}
Hence, $\int_0^tS(t-s) F(s)ds$ is continuous on $[0,T]$.
We thus have shown that  \eqref{autonomous linear evolution equation} has a unique mild solution
$X(t)=I_1(t)+I_2(t),$ where
\begin{equation}\label{XI1I2}
\begin{aligned}
I_1(t)&=S(t)\xi +\int_0^tS(t-s) F(s)ds, \\
I_2(t)&= \int_0^t S(t-s) G(s)dw_s.
\end{aligned}
\end{equation}

We shall show that this mild solution is a strict solution in the space $\mathcal C((0,T];\mathcal D(A))$. For this purpose, we divide the proof into two steps.

{\bf Step 1}. Let us verify that $I_1\in  \mathcal C((0,T];\mathcal D(A)) $ and satisfies the integral equation
\begin{equation}  \label{I1equation}
I_1(t)+\int_0^t AI_1(s)ds=\xi+ \int_0^t F(s)ds, \hspace{1cm}   t\in (0,T].
\end{equation}
Let $A_n=A(1+\frac{A}{n})^{-1}, n\in \mathbb N\setminus \{0\}$  be the Yosida approximation of $A$. Then $A_n$ satisfies \eqref{spectrumsectorialdomain} and \eqref{resolventnorm} uniformly and generates an analytic semigroup $S_n(t)$ (see e.g. \cite{yagi}). Furthermore, for every $\nu \in [0,\infty)$  we have 
\begin{equation} \label{LimitOfAnnuSn}
\lim_{n\to\infty} A_n^\nu S_n(t)=A^\nu S(t)  \quad\quad\text{ (limit in } \mathcal L(E))
\end{equation}
 and there exists $\varsigma_\nu>0$ independent of $n$ such that for every $t\in (0,T]$
\begin{equation} \label{EstimateIofAnnu}
\begin{aligned}
\begin{cases}
|A_n^\nu S_n(t)| \leq \varsigma_\nu t^{-\nu} &\quad\quad\text { if   } \nu>0,\\
|A_n^\nu S_n(t)| \leq \varsigma_\nu e^{-\varsigma_\nu t} &\quad\quad\text { if   } \nu=0.
\end{cases}
\end{aligned}
\end{equation}
Consider a function
$$I_{1n}(t)=S_n(t)\xi +\int_0^tS_n(t-s) F(s)ds, \hspace{2cm} t\in [0,T].$$
Due to \eqref{LimitOfAnnuSn} and \eqref{EstimateIofAnnu}, $\lim_{n\to \infty} I_{1n}(t)=I_1(t)$ a.s. Since $A_n$ is bounded, we verify that
$$dI_{1n}=[-A_n I_{1n}+F(t)]dt, \hspace{2cm} t\in (0,T].$$
From this equation, for any $\epsilon \in (0,T)$ we obtain that
\begin{equation}\label{EquationOfI1n}
I_{1n}(t)=I_{1n}(\epsilon)+\int_\epsilon^t[F(s)-A_nI_{1n}(s)]ds, \quad\quad t\in [\epsilon,T].
\end{equation}

We shall  observe convergence of $A_nI_{1n}$. We have
\begin{align} 
A_n&I_{1n}(t) \notag\\
=&A_nS_n(t)\xi+\int_0^t A_nS_n(t-s)[F(s)-F(t)]ds + \int_0^t A_nS_n(t-s)ds F(t)\notag\\
=&A_nS_n(t)\xi+\int_0^t A_nS_n(t-s)[F(s)-F(t)]ds + [I-S_n(t)]F(t). \label{Eq36}
\end{align}
Using \eqref{FbetasigmaSpaceProperty} and \eqref{EstimateIofAnnu}, we observe that 
\begin{align} 
|A_n&I_{1n}(t)|\notag\\
\leq &\varsigma_1 t^{-1} |\xi|+\int_0^t  \varsigma_1  |F|_{\mathcal F^{\beta,\sigma}} (t-s)^{\sigma-1} s^{\beta-\sigma-1} ds+(1+\varsigma_0 e^{-\varsigma_0 t}) |F|_{\mathcal F^{\beta,\sigma}} t^{\beta-1}\notag\\
=&\varsigma_1 |\xi| t^{-1} + [1+\varsigma_1  {\bf B}(\beta-\sigma, \sigma)   +\varsigma_0 e^{-\varsigma_0 t}         ]  |F|_{\mathcal F^{\beta,\sigma}} t^{\beta-1}, \quad \quad t\in (0,T], \label{EstimateOfNormOfAnI1n}
\end{align}
where ${\bf B} (\cdot,\cdot)$ is the beta function. 
Furthermore, due to \eqref{LimitOfAnnuSn} and \eqref{Eq36}, 
$$
\lim_{n\to\infty}A_nI_{1n}(t)=W(t),
$$
where
$$W(t)=AS(t)\xi+\int_0^t AS(t-s)[F(s)-F(t)]ds + [I-S(t)]F(t).$$

Let us verify  that $W(t)$ is continuous on $(0,T]$. Indeed, let $t_0\in (0,T]$. Using \eqref{FbetasigmaSpaceProperty} and \eqref{EstimateIofAnu}, for every $t\geq t_0$ we have
\begin{align}
&|W(t)-W(t_0)|\notag\\
\leq & |AS(t_0)[S(t-t_0)-I]\xi| + |[I-S(t)]F(t)-[I-S(t_0)]F(t_0)|\notag \\
&+\Big| \int_{t_0}^t AS(t-s)[F(s)-F(t)]ds+\int_0^{t_0} AS(t-s)ds[F(t_0)-F(t)]\notag\\
&+\int_0^{t_0} S(t-t_0)AS(t_0-s)[F(s)-F(t_0)]ds\notag\\
&-\int_0^{t_0} AS(t_0-s)[F(s)-F(t_0)]ds\Big|\notag\\
\leq & \iota_1 t_0^{-1}|S(t-t_0)\xi-\xi| + |[I-S(t)]F(t)-[I-S(t_0)]F(t_0)|\notag \\
&+ \int_{t_0}^t |AS(t-s)| |F(s)-F(t)|ds+|[S(t-t_0)-S(t)][F(t_0)-F(t)]|\notag\\
&+\int_0^{t_0} |[S(t-t_0)-I]AS(t_0-s)[F(s)-F(t_0)]|ds\notag\\
\leq & \iota_1 t_0^{-1}|S(t-t_0)\xi-\xi| + |[I-S(t)]F(t)-[I-S(t_0)]F(t_0)| \notag\\
&+ \int_{t_0}^t \iota_1  |F|_{\mathcal F^{\beta, \sigma}} (t-s)^{\sigma-1} s^{\beta-\sigma-1}ds+|S(t-t_0)-S(t)| |F(t_0)-F(t)|\notag\\
&+|S(t-t_0)-I|\int_0^{t_0}  \iota_1|F|_{\mathcal F^{\beta, \sigma}} (t_0-s)^{\sigma-1} s^{\beta-\sigma-1}ds\notag\\
\leq  & \iota_1 t_0^{-1}|S(t-t_0)\xi-\xi| + |[I-S(t)]F(t)-[I-S(t_0)]F(t_0)|\notag \\
&+ \varsigma_1  |F|_{\mathcal F^{\beta, \sigma}}  t_0^{\beta-\sigma-1} \int_{t_0}^t(t-s)^{\sigma-1} ds+|S(t-t_0)-S(t)| |F(t_0)-F(t)|\notag\\
&+  \iota_1|F|_{\mathcal F^{\beta, \sigma}} {\bf B}(\beta-\sigma,\sigma) t_0^{\beta-1} |S(t-t_0)-I|\notag\\
\leq  & \iota_1 t_0^{-1}|S(t-t_0)\xi-\xi| + |[I-S(t)]F(t)-[I-S(t_0)]F(t_0)| \label{Eq26}\\
&+ \frac{\iota_1  |F|_{\mathcal F^{\beta, \sigma}} }{\sigma}  t_0^{\beta-\sigma-1} (t-t_0)^\sigma+|S(t-t_0)-S(t)| |F(t_0)-F(t)|\notag\\
&+  \iota_1|F|_{\mathcal F^{\beta, \sigma}} {\bf B}(\beta-\sigma,\sigma) t_0^{\beta-1} |S(t-t_0)-I|.\notag
\end{align}
Thus, $\lim_{t\searrow  t_0}W(t)=W(t_0).$ Similarly, we obtain that $\lim_{t\nearrow  t_0}W(t)=W(t_0).$ Hence, $W(t)$ is continuous at $t=t_0$ and then at every point in $(0,T]$.

By the above aguments, we have
$$I_1(t)=\lim_{n\to\infty} I_{1n}(t)=\lim_{n\to\infty} A_n^{-1} A_n I_{1n}(t)=A^{-1}W(t).$$
This shows that $I_1(t) \in \mathcal D(A)$ and $AI_1(t)=W(t)\in \mathcal C((0,T];E).$ 

Let us prove that $\int_0^t W(s)ds$ exists. Indeed, by virtue of \eqref{FbetasigmaSpaceProperty}, \eqref{EstimateIofAnu}  and \eqref{EstimateIofS(t)Maximum}, we have
\begin{align*}
\int_0^t |[I-S(r)]F(r)|dr &\leq (1+\iota_0) |F|_{\mathcal F^{\beta,\sigma}}\int_0^t r^{\beta-1}dr\\
&=\frac{(1+\iota_0) |F|_{\mathcal F^{\beta,\sigma}} r^{\beta}}{\beta}, \hspace{2cm} t\in [0,T],
\end{align*}
and
\begin{align*}
&\int_0^t \Big|\int_0^s AS(s-r)[F(r)-F(s)]dr\Big| ds\\
&\leq \int_0^t \int_0^s |AS(s-r)| |F(r)-F(s)|dr ds\\
&\leq \iota_1 |F|_{\mathcal F^{\beta, \sigma}} \int_0^t \int_0^s   (s-r)^{\sigma-1} r^{\beta-\sigma-1}dr ds\\
&=\frac{\iota_1 |F|_{\mathcal F^{\beta, \sigma}} {\bf B}(\beta-\sigma, \sigma)  t^\beta}{\beta}, \hspace{2cm} t\in [0,T].
\end{align*}
These  estimates show that $\int_0^t [I-S(r)]F(r)dr$ and $\int_0^t \int_0^s AS(s-r)[F(r)-F(s)]dr ds $ exist for $t\in [0,T].$ Therefore, the existence of $\int_0^t W(s)ds$ for $t\in [0,T]$ follows from the equality
\begin{align*}
\int_0^t W(s)ds=&\int_0^t AS(r)\xi dr+\int_0^t \int_0^s AS(s-r)[F(r)-F(s)]dr ds\\
& +\int_0^t [I-S(r)]F(r)dr\\
=&[I-S(t)]\xi+\int_0^t \int_0^s AS(s-r)[F(r)-F(s)]dr ds\\
& +\int_0^t [I-S(r)]F(r)dr.
\end{align*}

In view of \eqref{EstimateOfNormOfAnI1n}, by applying  the Lebesgue dominate convergence theorem to \eqref{EquationOfI1n}, for any $\epsilon \in (0,T)$ we obtain  that
 
\begin{equation} \label{Eq27}
I_1(t)=I_1(\epsilon)+\int_\epsilon^t[F(s)-AI_1(s)]ds, \quad\quad t\in [\epsilon,T].
\end{equation}
From \eqref{Eq0} and \eqref{XI1I2}, $\lim_{\epsilon\to 0} I_1(\epsilon) =\xi$. Letting $\epsilon \to 0$ in \eqref{Eq27},  we then observe the equation \eqref{I1equation}. 

{\bf Step 2}. Let us observe that $I_2\in \mathcal C([0,T]; \mathcal D(A)) $ and satisfies the equation
\begin{equation} \label{I2equation}
I_2(t)+ \int_0^t AI_2(s)ds=\int_0^t G(u)dw_u, \quad\quad t\in (0,T].
\end{equation}
If \eqref{FGSpace1}  takes place, then by using \eqref{FbetasigmaSpaceProperty} and \eqref{Eq35}, we have
\begin{align*}
\int_0^t&|AS(t-s)G(s)|^2 ds \leq   \int_0^t  c_\delta^2(t-s)^{-2\delta}|G|_{\mathcal F^{\beta+\frac{1}{2}, \sigma}}^2 s^{2\beta-1}ds\\
&=c_\delta^2 |G|_{\mathcal F^{\beta+\frac{1}{2}, \sigma}}^2 t^{2(\beta-\delta)} {\bf B}(2\beta,1-2\delta)<\infty, \hspace{1cm} t\in (0,T].
\end{align*}
Otherwise, if  \eqref{FGSpace2}   takes place then 
\begin{align*}
\int_0^t|AS(t-s)G(s)|^2 ds &\leq  c_\delta^2 \int_0^t (t-s)^{-2\delta} |G|_{\mathcal B((0,T]; E)}^2ds\\
&=\frac{ c_\delta^2 t^{1-2\delta}|G|_{\mathcal B((0,T]; E)}^2}{1-2\delta}  <\infty, \hspace{1cm}t\in (0,T].
\end{align*}
Hence, the integral $\int_0^t AS(t-s)G(s)dw_s, t\in (0,T]$ is well-defined and then is continuous. 
 Since $A$ is closed, we obtain that
$$AI_2(t)=\int_0^t AS(t-s)G(s)dw_s, \hspace{2cm} t\in (0,T].$$
Using the Fubini formula, we have
\begin{equation*}
\begin{aligned}
A \int_0^t I_2(s)ds&=\int_0^t \int_0^s AS(s-u) G(u)dw_uds\\
&=\int_0^t \int_u^t AS(s-u) G(u)dsdw_u\\
&=\int_0^t [G(u)-S(t-u)G(u)]dw_u\\
&=\int_0^t G(u)dw_u-\int_0^t S(t-u)G(u)dw_u\\
&=\int_0^t G(u)dw_u-I_2(t),  \hspace{2cm}  t\in (0,T].
\end{aligned}
\end{equation*}
This means that  $I_2$ satisfies \eqref{I2equation}.

From these two steps, we conclude that  $X(t)$ is a strict solution in the space $\mathcal C((0,T];\mathcal D(A))$.
\end{proof}
\begin{remark} \label{remark1}
In {\bf Step 1} of the proof of Theorem \ref{StrictSolutions}, the assumption \eqref{Eq35}  is not used. Using this assumption, we can reduce the proof of the statement of this step.  However, we do not present such a proof, because it is not useful for the study of regularity of mild solutions in the next theorems.  In fact, this assumption is only to guarantee  the existence of the integral $\int_0^t AS(t-s) G(s)dw_s, t\in (0,T].$
\end{remark}
\subsection{Regularity of mild solutions}
In this subsection, we will  study   regularity of mild solutions of \eqref{autonomous linear evolution equation} without the condition \eqref{Eq35} in  Theorem \ref{StrictSolutions}. The initial value is considered in the domain $\mathcal D(A^\beta).$ 
\begin{theorem} \label{regularity theorem autonomous linear evolution equation}
Suppose that   $\xi \in \mathcal D(A^\beta)$ a.s. Then  there exists a mild solution $X$ of \eqref{autonomous linear evolution equation} in the space
\begin{equation*} \label{regularity theorem autonomous linear evolution equationXSpace}
 X \in \mathcal C([0,T];\mathcal D(A^\beta)) \cap \mathcal C^\beta([0,T]; E).
 \end{equation*}
Furthermore,  $X$ satisfies the estimates
\begin{align} 
&\mathbb E |X(t)|^2 \leq \rho_1[\mathbb E|\xi|^2 + |F|_{\mathcal F^{\beta,\sigma}}^2+  |G|_{\mathcal F^{\beta+\frac{1}{2},\sigma}}^2], \label{estimateofstrictsolutionCase2}
\end{align}
when  \eqref{FGSpace1}  takes place, and 
\begin{align} 
&\mathbb E |X(t)|^2 \leq \rho_2[\mathbb E|\xi|^2+ |F|_{\mathcal F^{\beta,\sigma}}^2+ (1-e^{-\rho_2t}) |G|_{\mathcal B([0,t]; E)}^2 ],  \label{estimateofstrictsolution}
\end{align}
when \eqref{FGSpace2} takes place. Here, $\rho_1$ and $\rho_2$ are two positive constants depending only on $A, \beta$ and $\sigma$.
\end{theorem}
\begin{proof}
The existence of the mild solution $X(t)=I_1(t)+I_2(t)$  of \eqref{autonomous linear evolution equation} is already shown in Theorem \ref{StrictSolutions},  where
$I_1(t)$ and $ I_2(t)$ are defined by \eqref{XI1I2}. First, let us assume that the condition \eqref{FGSpace1}  takes place. The proof for this case is divided into several steps.

{\bf Step 1}. Let us show that $I_1(t)\in \mathcal D(A^\beta) $ for $t\in [0,T]$ and that
$A^\beta I_1\in \mathcal C([0,T]; E).
$
 Since $S(t)$ is strongly continuous,  we have
$$\lim_{t\to s}|A^\beta S(t)\xi-A^\beta S(s)\xi|=\lim_{t\to s}|[S(t)-S(t_0)]A^\beta \xi|=0, \hspace{1cm} s\in [0,T].$$
Therefore, $A^\beta S(t)\xi$ is continuous on $[0,T]$. Because of   
$$A^\beta I_1(t)=A^\beta S(t)\xi+A^\beta\int_0^t S(t-s)F(s)ds,$$
it suffices to show that $A^\beta \int_0^t  S(t-s)F(s)|ds$ is well-defined and is continuos on $[0,T]$.

Let us verify the first assertion.  Using the inequalities \eqref{FbetasigmaSpaceProperty} and \eqref{EstimateIofAnu}, we have
\begin{align}
\int_0^t|A^\beta S(t-s)F(s)|ds&\leq \int_0^t|A^\beta S(t-s)| |F(s)|ds \notag\\
&\leq |F|_{\mathcal F^{\beta,\sigma}} \iota_\beta\int_0^t(t-s)^{-\beta} s^{\beta-1}ds\notag\\
&= |F|_{\mathcal F^{\beta,\sigma}} \iota_\beta\int_0^1u^{\beta-1}(1-u)^{-\beta}du\notag\\
&=\iota_\beta |F|_{\mathcal F^{\beta,\sigma}}  {\bf B}(\beta,1-\beta), \hspace{2cm} t\in[0,T].  \label{Eq4}
\end{align}
 Hence, $\int_0^tA^\beta S(t-s)F(s)ds$ is well-defined. Since $A^\beta$ is closed, we obtain that
$$A^\beta\int_0^t S(t-s)F(s)ds=\int_0^tA^\beta S(t-s)F(s)ds.$$

Let us now verify the second assertion, i.e. to verify the continuity of  $A^\beta\int_0^t S(t-s)F(s)ds$ on $[0,T]$. Fix $t_0\in[0,T]$. We consider two cases. 

Case 1: $t_0>0$. For every $t\geq t_0$ we have
\begin{align*}
&\left|\int_0^tA^\beta S(t-s)F(s)ds-\int_0^{t_0}A^\beta S(t_0-s)F(s)ds\right|\\
&\leq \left|\int_0^{t_0}A^\beta [S(t-s)-S(t_0-s)]F(s)ds\right|+\left|\int_{t_0}^tA^\beta S(t-s)F(s)ds\right|\\
&\leq \int_0^{t_0}|A^\beta S(t_0-s)| |F(s)|ds|S(t-t_0)-I| +\int_{t_0}^t|A^\beta S(t-s)| |F(s)|ds.
\end{align*}
Using \eqref{FbetasigmaSpaceProperty} and \eqref{EstimateIofAnu}, we observe that 
\begin{align*}
&\left|\int_0^tA^\beta S(t-s)F(s)ds-\int_0^{t_0}A^\beta S(t_0-s)F(s)ds\right|\\
&\leq  \int_0^{t_0} \iota_\beta (t_0-s)^{-\beta} |F|_{\mathcal F^{\beta,\sigma}}s^{\beta-1}ds|S(t-t_0)-I|+\int_{t_0}^t\iota_\beta (t-s)^{-\beta} \sup_{s\in[t_0,t]}|F(s)|ds\\
&=  |F|_{\mathcal F^{\beta,\sigma}} \iota_\beta {\bf B}(\beta,1-\beta) |S(t-t_0)-I| +   \frac{\iota_\beta\sup_{s\in[t_0,t]}|F(s)|(t-t_0)^{1-\beta}}{{1-\beta}} \\
&\to 0 \text { as } t\searrow t_0.
\end{align*}
This means that  $\int_0^tA^\beta S(t-s)F(s)ds$ is right-continuous at $t=t_0$. Similarly, we can show that it is left-continuous at $t=t_0$. Therefore, $A^\beta\int_0^t S(t-s)F(s)ds$ is continuous at $t=t_0$.

Case 2: $t_0=0$. By the property of the space $\mathcal F^{\beta,\sigma}((0,T];E),$  we may put $z=\lim_{t\searrow 0} t^{1-\beta}F(t)$.  We have
\begin{align*}
\Big|A^\beta & \int_0^t S(t-s)F(s)ds\Big|\notag\\
=&\Big|\int_0^t A^\beta S(t-s)[F(s)-F(t)]ds\Big|+\Big|\int_0^t A^\beta S(t-s)F(t)ds\Big|\notag\\
=&\Big|\int_0^t A^\beta S(t-s)[F(s)-F(t)]ds\Big|+\Big|[I-S(t)]A^{\beta-1}F(t)\Big|\notag\\
\leq &\int_0^t |A^\beta S(t-s)| |F(t)-F(s)|ds
+| t^{\beta-1}[I-S(t)]A^{\beta-1} [t^{1-\beta}F(t)-z]|\notag\\
&+| t^{\beta-1}[I-S(t)]A^{\beta-1} z|.\notag
\end{align*}
Using \eqref{Fbetasigma3}, \eqref{EstimateIofAnu} and \eqref{EstimateIofS(t)-IA-theta},  we obtain that 
\begin{align*}
\limsup_{t\searrow 0}&\left|A^\beta\int_0^t S(t-s)F(s)ds\right|\notag\\
\leq & \iota_\beta \limsup_{t\searrow 0}\int_0^t \iota_\beta (t-s)^{-\beta}|F(t)-F(s)|ds\notag\\
&+\frac{\iota_\beta}{1-\beta}\limsup_{t\searrow 0}|t^{1-\beta}F(t)-z|\notag\\
&+\limsup_{t\searrow 0}| t^{\beta-1}[I-S(t)]A^{\beta-1} z|\notag\\
= & \iota_\beta \limsup_{t\searrow 0}\int_0^t (t-s)^{\sigma-\beta}s^{-1+\beta-\sigma}  \frac{ s^{1-\beta+\sigma}|F(t)-F(s)|}{(t-s)^\sigma}ds\notag\\
&+\frac{\iota_\beta}{1-\beta}\limsup_{t\searrow 0}|t^{1-\beta}F(t)-z|\notag\\
&+\limsup_{t\searrow 0}| t^{\beta-1}[I-S(t)]A^{\beta-1} z|\notag\\
\leq &\iota_\beta {\bf B}(\beta-\sigma, 1-\beta+\sigma)\limsup_{t\searrow 0} \sup_{s\in[0,t)}\frac{ s^{1-\beta+\sigma}|F(t)-F(s)|}{(t-s)^\sigma}\notag\\
&+\limsup_{t\searrow 0}| t^{\beta-1}[I-S(t)]A^{\beta-1} z|\notag\\
=&\limsup_{t\searrow 0}| t^{\beta-1}[I-S(t)]A^{\beta-1} z|. \notag
\end{align*}
Since $\mathcal D(A^\beta)$ is dense in $E$, there exists a sequence $\{z_n\}_n$  in $\mathcal D(A^\beta)$ that converges to $z$ as $n\to \infty.$ Using \eqref{EstimateIofS(t)-IA-theta}, we observe that
\begin{align*}
&\limsup_{t\searrow 0}\left|A^\beta\int_0^t S(t-s)F(s)ds\right|\\
&\leq\limsup_{t\searrow 0}  |t^{\beta-1}[I-S(t)]A^{\beta-1} (z-z_n)|+\limsup_{t\searrow 0}  |t^{\beta-1}[I-S(t)]A^{-1}A^{\beta} z_n|\\
&\leq \frac{\iota_\beta }{1-\beta}|z-z_n|+\iota_0 \limsup_{t\searrow 0}  t^\beta |A^{\beta} z_n|\\
&=  \frac{\iota_\beta }{1-\beta}|z-z_n|, \hspace{2cm} n=1,2,\dots.
\end{align*}
Letting $n$ to $\infty$, we obtain that $\lim_{t\searrow 0}A^\beta\int_0^t S(t-s)F(s)ds=0.$ This means that  $A^\beta\int_0^t S(t-s)F(s)ds$ is right-continuous at $t=t_0$.

From these two cases, we conclude that $A^\beta\int_0^t S(t-s)F(s)ds$ is continuous at $t=t_0$ and then at every point in $[0,T]$. The second assertion  has been shown.

{\bf Step 2}. Let us verify  that
$
A^\beta X\in \mathcal C([0,T]; E).
$
In view of  \eqref {FbetasigmaSpaceProperty} and \eqref{EstimateIofAnu}, we have
\begin{align}
\int_0^t |A^\beta  S(t-s) G(s) |^2 ds &\leq \iota_\beta^2  |G|_{\mathcal F^{\beta+\frac{1}{2},\sigma}}^2 \int_0^t (t-s)^{-2\beta} s^{2\beta-1}ds \notag \\
&=\iota_\beta^2  |G|_{\mathcal F^{\beta+\frac{1}{2},\sigma}}^2 {\bf B}(2\beta, 1-2\beta)< \infty.  \label{Eq5}
\end{align}
Thus, $\int_0^t A^\beta  S(t-s) G(s)dw_s $ is well-defined. Since $A^\beta$ is closed, we obtain that
$$ A^\beta I_2(t) =\int_0^t A^\beta  S(t-s) G(s)dw_s. $$ 
Thanks to Proposition \ref{IntegralInequality}, $ A^\beta I_2$ is a continuous square integrable martingale on $[0,T]$. On the account of  {\bf Step 1}, we conclude that
$$A^\beta X=A^\beta I_1+A^\beta I_2 \in \mathcal C([0,T]; E).$$

{\bf Step 3}. Let us show that $I_1$   is $\beta$\,{-}\,H\"older continuous on $[0,T]$. 
From \eqref{EstimateIofAnu}, \eqref{EstimateIofS(t)Maximum}, \eqref{XI1I2} and \eqref{I1equation}, for every $0\leq s<t\leq T$  we observe that
\begin{align}
|I_1(t)-I_1(s)|=&\Big|\int_s^t F(u)du-\int_s^t AI_1(u) du\Big|  \notag \\ 
=&\Big|\int_s^t F(u)du-\int_s^t AS(u)\xi  du-\int_s^t \int_0^uAS(u-r)F(r)drdu\Big|  \notag \\ 
\leq& \Big|\int_s^t [F(u)-AS(u)\xi]  du\Big|+\int_s^t  \Big|\int_0^uAS(u-r)F(u)dr\Big|du  \notag  \\
&+\int_s^t  \Big|\int_0^uAS(u-r)[F(u)-F(r)]dr\Big|du   \notag\\ 
\leq & \Big|\int_s^t [F(u)-AS(u)\xi]  du\Big|+\int_s^t  |[I-S(u)]F(u)|du  \label{Eq2} \\
&+\int_s^t  \int_0^u|AS(u-r)| |F(u)-F(r)|drdu  \notag \\ 
\leq & \int_s^t [|F(u)-AS(u)\xi  |+(1+\iota_0)|F(u)|]du  \notag \\
&+\iota_1\int_s^t  \int_0^u(u-r)^{-1} |F(u)-F(r)|drdu  \notag \\ 
=&I_{11}(s,t)+I_{12}(s,t).  \label{Eq28} 
\end{align}
We shall give estimates for $I_{11}$ and $I_{12}$. 
Since $\xi\in \mathcal D(A^\beta)$ a.s., we have 
$AS(\cdot)\xi \in \mathcal F^{\beta,\sigma}((0,T];E)$ a.s. (see \eqref{Eq53}). Therefore,  
$$F(\cdot)-AS(\cdot)\xi \in \mathcal F^{\beta,\sigma}((0,T];E) \hspace{2cm} \text{ a.s.}$$ 
In view of \eqref {FbetasigmaSpaceProperty}, we see that 
\begin{align*}
I_{11}(s,t)&\leq \int_s^t [|F(\cdot)-AS(\cdot)\xi|_ {\mathcal F^{\beta,\sigma}}   u^{\beta-1}+|F|_ {\mathcal F^{\beta,\sigma}}(1+\iota_0) u^{\beta-1}]du\\
&=\frac{|F(\cdot)-AS(\cdot)\xi|_ {\mathcal F^{\beta,\sigma}} + |F|_ {\mathcal F^{\beta,\sigma}}(1+\iota_0)}{\beta}    (t^\beta-s^\beta)\\
&\leq \frac{|F(\cdot)-AS(\cdot)\xi|_ {\mathcal F^{\beta,\sigma}} + |F|_ {\mathcal F^{\beta,\sigma}}(1+\iota_0)}{\beta} (t-s)^\beta.
\end{align*}

Meanwhile, using \eqref {FbetasigmaSpaceProperty}, we have
\begin{align}
I_{12}(s,t)=&\iota_1\int_s^t  \int_0^u(u-r)^{\sigma-1}  r^{\beta-1-\sigma}\frac{r^{1-\beta+\sigma}|F(u)-F(r)|}{(u-r)^\sigma}drdu \notag\\
\leq & \iota_1|F|_{\mathcal F^{\beta,\sigma}}\int_s^t  \int_0^u(u-r)^{\sigma-1}  r^{\beta-\sigma-1}drdu\notag\\
= & \iota_1|F|_{\mathcal F^{\beta,\sigma}}\int_s^t  u^{\beta-1}\int_0^1(1-v)^{\sigma-1}  v^{\beta-\sigma-1}dvdu\notag\\
= & \iota_1|F|_{\mathcal F^{\beta,\sigma}} {\bf B}(\beta-\sigma,\sigma)\int_s^t  u^{\beta-1}du\notag\\
= & \frac{\iota_1|F|_{\mathcal F^{\beta,\sigma}} {\bf B}(\beta-\sigma,\sigma)}{\beta} (t^\beta-s^\beta)\notag\\
 \leq& \frac{\iota_1|F|_{\mathcal F^{\beta,\sigma}} {\bf B}(\beta-\sigma,\sigma)}{\beta} (t-s)^\beta. \label{Eq29}
\end{align}
Thus, $I_1(\cdot)$ is $\beta$\,{-}\,H\"older continuous on $[0,T]$. 

{\bf Step 4}. Let us show that $X $ is $\beta$-H\"older continuous on $[0,T]$. 
On the account of  {\bf Step 3},  it suffices to show that 
$I_2\in \mathcal C^\beta([0,T];E)$. Let $0\leq s<t\leq T$, then 
$$I_2(t)-I_2(s)=\int_s^t S(t-r)G(r)dw_r +\int_0^s [S(t-r)-S(s-r)]G(r)dw_r.$$
Since the integrals in the right hand side are independent and have  zero expectation,  we have
\begin{align*}
\mathbb E& |I_2(t)-I_2(s)|^2  \notag\\
=&\mathbb E \Big|\int_s^t S(t-r)G(r)dw_r\Big|^2 +\mathbb E\Big|\int_0^s [S(t-r)-S(s-r)]G(r)dw_r\Big|^2  \notag\\
\leq & c(E)\Big[\int_s^t  |S(t-r)|^2  | G(r)|^2 dr+ \int_0^s  |S(t-r)-S(s-r)|^2 |G(r)|^2dr\Big].  \notag
\end{align*}
Using \eqref {FbetasigmaSpaceProperty}, \eqref{EstimateIofAnu} and \eqref{EstimateIofS(t)Maximum}, we then observe that 
\begin{align*}
\mathbb E& |I_2(t)-I_2(s)|^2  \notag\\
\leq & c(E)|G|_{\mathcal F^{\beta+\frac{1}{2},\sigma}}^2 \Big [\iota_0^2 \int_s^t r^{2\beta-1}dr+ \int_0^s\Big |\int_{s-r}^{t-r} AS(u) du\Big|^2 r^{2\beta-1}dr \Big ]  \notag\\
\leq &c(E)|G|_{\mathcal F^{\beta+\frac{1}{2},\sigma}}^2 \Big [\frac {\iota_0^2 (t^{2\beta}-s^{2\beta})}{2\beta}+\iota_1^2 \int_0^s \Big(\int_{s-r}^{t-r} u^{-1}du\Big)^2 r^{2\beta-1}dr \Big ].  \notag
\end{align*}
Dividing $-1$ as $-1=-\beta+\beta-1$, we have
\begin{align}
\Big(\int_{s-r}^{t-r} u^{-1}du\Big)^2&\leq    \Big[\int_{s-r}^{t-r} (s-r)^{-\beta}u^{\beta-1}du\Big]^2\notag\\
&= (s-r)^{-2\beta}\frac{[(t-r)^\beta-(s-r)^\beta]^2}{\beta^2}\notag\\
&\leq (s-r)^{-2\beta}\frac{(t-s)^{2\beta}}{\beta^2}. \label{Eq37}
\end{align}
Hence, 
\begin{align}
\mathbb E& |I_2(t)-I_2(s)|^2  \notag\\
\leq &c(E)|G|_{\mathcal F^{\beta+\frac{1}{2},\sigma}}^2 \Big [\frac {\iota_0^2 (t^{2\beta}-s^{2\beta})}{2\beta}+\frac{\iota_1^2(t-s)^{2\beta}}{\beta^2}  \int_0^s (s-r)^{-2\beta} r^{2\beta-1}dr  \Big ]  \notag\\
\leq  &c(E)|G|_{\mathcal F^{\beta+\frac{1}{2},\sigma}}^2 \Big [\frac {\iota_0^2}{2\beta}+\frac{\iota_1^2 {\bf B}(2\beta, 1-2\beta) }{\beta^2} \Big ]   (t-s)^{2\beta}.\label{ExpextationOfI2tI2sSquare}
\end{align}
In addition, by the definition of the stochastic integral, $  I_2(t)$ is a Gaussian process. Theorem \ref{Kolmogorov testGaussian} then provides  that $I_2\in \mathcal C^\beta([0,T];E)$. 

{\bf Step 5}. Let us prove \eqref{estimateofstrictsolutionCase2}. 
Taking $s=0, t\in (0,T]$ in \eqref{Eq2},  it yields that
\begin{align*}
|I_1(t)&-I_1(0)|\\
\leq&  \Big|\int_0^t [F(u)-AS(u)\xi]  du\Big|+\int_0^t  |[I-S(u)]F(u)|du\\
&+\int_0^t  \int_0^u|AS(u-r)| |F(u)-F(r)|drdu\\
\leq&  \int_0^t  [1+|I-S(u)|] |F(u)|du+\Big|\int_0^t AS(u)\xi  du\\
&+c_1\int_0^t  \int_0^u(u-r)^{-1} |F(u)-F(r)|drdu.
\end{align*}
In view of  \eqref{FbetasigmaSpaceProperty},  \eqref{EstimateIofS(t)Maximum}, \eqref{Eq28} and \eqref{Eq29}, we  have
\begin{align*}
|I_1(t)&-I_1(0)|\\
\leq&  \int_0^t  (2+\iota_0) |F(u)|du+|[I-S(t)]\xi||+I_{12}(0,t)\\
\leq&  (2+\iota_0) \int_0^t   |F|_{\mathcal F^{\beta,\sigma}} u^{\beta-1}du+(1+\iota_0)|\xi|+\frac{\iota_1|F|_{\mathcal F^{\beta,\sigma}} {\bf B}(\beta-\sigma,\sigma)}{\beta} t^\beta\\
=&  \frac{[2+\iota_0+\iota_1 {\bf B}(\beta-\sigma,\sigma)]  t^\beta |F|_{\mathcal F^{\beta,\sigma}}}{\beta}   +(1+\iota_0)|\xi|.
\end{align*}
Then 
$$|I_1(t)| \leq (2+\iota_0)|\xi|+\frac{[2+\iota_0+\iota_1 {\bf B}(\beta-\sigma,\sigma) ] t^\beta  |F|_{\mathcal F^{\beta,\sigma}}}{\beta},  \quad\quad t\in (0,T].$$
Obviously, this inequality also holds true for $t=0$. As a consequence, there exists $c_1>0 $ depending only on $A, \beta$ and $\sigma$ such that
\begin{equation}\label{estimateofstrictsolutionI1}
\mathbb E|I_1(t)|^2\leq c_1 [\mathbb E|\xi|^2 + |F|_{\mathcal F^{\beta,\sigma}}^2], \hspace{1cm} t\in [0,T].
\end{equation}

On the other hand, using  \eqref{FbetasigmaSpaceProperty} and \eqref{EstimateIofS(t)}, we have
\begin{align}
\mathbb E|I_2(t)|^2  &\leq c(E) \int_0^t  |S(t-s)G(s)|^2 ds  \notag\\
&\leq c(E) \iota_0^2 |G|_{\mathcal F^{\beta+\frac{1}{2},\sigma}}^2 \int_0^t  e^{-2\nu(t-s)} s^{2\beta-1}ds  \notag\\
&\leq c(E)  c_{\nu,\beta} \iota_0^2 |G|_{\mathcal F^{\beta+\frac{1}{2},\sigma}}^2,  \hspace{2cm}  t\in [0,T], \label{estimateofstrictsolutionI2Case2}
\end{align}
where $c_{\nu,\beta} =\sup_{t\in [0,\infty)} \int_0^t  e^{-2\nu(t-s)} s^{2\beta-1}ds <\infty.$

Combining \eqref{estimateofstrictsolutionI1} and \eqref{estimateofstrictsolutionI2Case2}, we conclude that
\begin{align*}
\mathbb E|X(t)|^2 \leq & 2\mathbb E[ I_1(t)^2 + I_2(t)^2]\\
\leq &2 c_1[\mathbb E|\xi|^2 + |F|_{\mathcal F^{\beta,\sigma}}^2]+2c(E)  c_{\nu,\beta} \iota_0^2  |G|_{\mathcal F^{\beta+\frac{1}{2},\sigma}}^2,
\end{align*}
from which it follows  \eqref{estimateofstrictsolutionCase2}.

From these steps, the assertion of the theorem under the condition \eqref{FGSpace1} is proved.
 
Let us now assume that the condition \eqref{FGSpace2} takes place.  Similarly to the above case, we will verify that the assertions in {\bf Steps 1-4} still hold true. Indeed, the assertions in  {\bf Step 1} and {\bf Step 3} are obviously true, since they are independent of  the conditions  \eqref{FGSpace1} and \eqref{FGSpace2}.

To show the assertion in {\bf Step 2}, it suffices to verify that $\int_0^t A^\beta S(t-s) G(s)ds$ is well-defined. This follows from a fact that
\begin{align*}
\int_0^t |A^\beta  S(t-s) G(s) |^2 ds &\leq \iota_\beta^2|G|_{\mathcal B([0,t];E)}^2 \int_0^t (t-s)^{-2\beta}ds\\
&=\frac{\iota_\beta^2}{1-2\beta}|G|_{\mathcal B([0,t];E)}^2 t^{1-2\beta}<\infty,
\end{align*}
here we used  the property  \eqref{EstimateIofAnu}.

To obtain the assertion in {\bf Step 4}, it suffices to prove that $I_2$ is $\beta$\,{-}\,H\"older continuous on $[0,T]$. Indeed,  let  $0\leq s<t\leq T$.   Using \eqref{EstimateIofAnu} and \eqref{EstimateIofS(t)Maximum}, 
we have 
 \begin{align*}
\mathbb E& |I_2(t)-I_2(s)|^2  \notag\\
=&\mathbb E \Big|\int_s^t S(t-r)G(r)dw_r\Big|^2 +\mathbb E\Big|\int_0^s [S(t-r)-S(s-r)]G(r)dw_r\Big|^2  \notag\\
\leq & c(E)\Big [\int_s^t  |S(t-r)|^2  | G(r)|^2 dr+ \int_0^s  |S(t-r)-S(s-r)|^2 |G(r)|^2dr  \Big] \notag\\
\leq & c(E)|G|_{\mathcal B([0,T];E)}^2 \Big [\iota_0^2(t-s) + \int_0^s\Big |\int_{s-r}^{t-r} AS(u) du\Big|^2 dr \Big ]  \notag\\
\leq & c(E)|G|_{\mathcal B([0,T];E)}^2 \Big [\iota_0^2(t-s) +\iota_1^2 \int_0^s \Big(\int_{s-r}^{t-r} u^{-1}du\Big)^2 dr \Big ].
\end{align*}
Fix $\alpha \in (0,\frac{1}{2}). $ Dividing $-1$ as $-1=-\alpha+\alpha-1$, we have
$$\Big(\int_{s-r}^{t-r} u^{-1}du\Big)^2\leq (s-r)^{-2\alpha} \frac{(t-s)^{2\alpha}}{\alpha^2}, \hspace{1cm} (\text{see } \eqref{Eq37}).$$
Therefore, 
\begin{align*}
\mathbb E& |I_2(t)-I_2(s)|^2  \notag\\
\leq &c(E)|G|_{\mathcal B([0,T];E)}^2 \Big [\iota_0^2(t-s) +\frac{\iota_1^2}{\alpha^2}  (t-s)^{2\alpha} \int_0^s (s-r)^{-2\alpha} dr  \Big ]  \notag\\
\leq &c(E)|G|_{\mathcal B([0,T];E)}^2 \Big [\iota_0^2(t-s) +\frac{\iota_1^2 s^{1-2\alpha}}{\alpha^2(1-2\alpha)} (t-s)^{2\alpha} \Big]\notag\\
= &c(E)|G|_{\mathcal B([0,T];E)}^2 \Big [\iota_0^2(t-s)^{1-2\alpha}+\frac{\iota_1^2 s^{1-2\alpha}}{\alpha^2(1-2\alpha)}  \Big ](t-s)^{2\alpha}   \notag\\
\leq &c(E)|G|_{\mathcal B([0,T];E)}^2 \Big [\iota_0^2+\frac{\iota_1^2 }{\alpha^2(1-2\alpha)}  \Big ]T^{1-2\alpha} (t-s)^{2\alpha}. 
\end{align*}
Applying Theorem \ref{Kolmogorov testGaussian} for the Gaussian process $I_2(t)$, we observe that $I_2\in \mathcal C^\alpha([0,T];E)$ for any $\alpha \in (0,\frac{1}{2}).$ In particular, $I_2\in \mathcal C^\beta([0,T];E).$

Let us finally observe  \eqref{estimateofstrictsolution}. 
Obviously, the estimate \eqref{estimateofstrictsolutionI1} still holds true, since it  depends on  neither  \eqref{FGSpace1} nor  \eqref{FGSpace2}.

Meanwhile,  using \eqref{EstimateIofS(t)}, we have
\begin{align}
 |I_2(t)|^2  &\leq c(E) \int_0^t  |S(t-s)G(s)|^2 ds  \notag\\
&\leq c(E) \iota_0^2 |G|_{\mathcal B([0,t];E)}^2 \int_0^t  e^{-2\nu(t-s)}ds  \notag\\
&=\frac{c(E) \iota_0^2 |G|_{\mathcal B([0,t];E)}^2}{2\nu}  (1-e^{-2\nu t}), \hspace{1cm} t\in [0,T]. \label{estimateofstrictsolutionI2}
\end{align}
Combining \eqref{estimateofstrictsolutionI1} and \eqref{estimateofstrictsolutionI2},  we obtain that 
$$\mathbb E|X(t)|^2\leq 2 c_1[\mathbb E|\xi|^2 + |F|_{\mathcal F^{\beta,\sigma}}^2]+\frac{c(E) \iota_0^2 |G|_{\mathcal B([0,t];E)}^2}{\nu}  (1-e^{-2\nu t}), \quad t\in [0,T].$$
Thus, \eqref{estimateofstrictsolution} has been verifed.  
It completes the proof.
\end{proof}
The following corollary is a direct consequence of Theorems \ref{StrictSolutions}-\ref{regularity theorem autonomous linear evolution equation}.
\begin{corollary}
Assume that \eqref{Eq35} holds true and that $\xi \in \mathcal D(A^\beta)$ a.s.  Then  there exists a strict solution $X$ of \eqref{autonomous linear evolution equation} in the space: 
$$ X \in \mathcal C((0,T];\mathcal D(A)) \cap  \mathcal C([0,T];\mathcal D(A^\beta))\cap \mathcal C^\beta([0,T];E).$$
Furthermore,  $X$ satisfies the estimate \eqref{estimateofstrictsolutionCase2} when \eqref{FGSpace1} takes place, and satisfies  the estimate \eqref{estimateofstrictsolution} when  \eqref{FGSpace2} takes place.
\end{corollary}
\section{Semilinear evolution equations with additive noise} \label{section5}
In this section, we will study semilinear evolution equations with additive noise:  the function $F(t,x)$ is divided into two parts: one depends only on $x$ and the other depends only on $t$, i.e.  $F(t,x)=F_1(X)+F_2(t),$ and the function $G(t,x)=G(t)$ depends only on $t$.  Let us rewrite \eqref{E2} into the form of   semilinear evolution equations with additive noise.
\begin{equation} \label{semilinear evolution equation}
\begin{cases}
dX+AXdt=[F_1(X)+F_2(t)]dt+ G(t)dw_t, \hspace{1cm} t\in(0,T],\\
X(0)=\xi,
\end{cases}
\end{equation}
where $F_1$ is measurable from $(\Omega_T\times E, \mathcal P_T\times \mathcal B(E))$ to $(E,\mathcal B(E)),$ $F_2$ and $G$ are non-random measurable functions from $[0,T]$ to $E$.

Let fix constants $\eta, \beta, \sigma $ such that
\begin{equation*}
\begin{cases}
0<\eta<\frac{1}{2}, \\
 0\lor (2\eta-\frac{1}{2})<\beta<\eta, \\
 0<\sigma <\beta.
\end{cases}
\end{equation*}
We suppose further that
\begin{itemize}
\item [(H1)] $F_1$ defines on the domain $ \mathcal D(A^\eta)$ and satisfies a Lipschitz condition of the form
     \begin{equation} \label{AbetaLipschitzcondition}
        |F_1(x)-F_1(y)|\leq c_{F_1}  [ |A^\eta(x-y)|+|A^{\tilde\beta} (x-y)|], \quad\quad x,y\in \mathcal D(A^\eta)
      \end{equation}
with $\tilde\beta=\beta$, where $c_{F_1}$ is some positive constant.
\item [(H2)] $F_2 \in \mathcal F^{\beta,\sigma} ((0,T];E). $
\item [(H3)] $ G\in \mathcal F^{\beta+\frac{1}{2},\sigma} ((0,T];E).$
\end{itemize}
\subsection{Existence of local mild solutions}
\begin{theorem}\label{semilinear evolution equationTheorem1}
Let \eqref{spectrumsectorialdomain}, \eqref{resolventnorm}, {\rm (H1)}, {\rm (H2)} and {\rm (H3)}   be satisfied.
Let $\xi\in \mathcal D(A^\beta)$ such that $\mathbb E|A^\beta\xi|^2<\infty$.  Then \eqref{semilinear evolution equation} possesses a unique local mild solution $X$ in the function space:
\begin{equation}\label{semilinear evolution equationTheorem1Regularity}
X\in  \mathcal C((0,T_{F_1,F_2,G,\xi}];\mathcal D(A^\eta))\cap \mathcal C([0,T_{F_1,F_2,G,\xi}];\mathcal D(A^\beta)),
\end{equation}
where $T_{F_1,F_2,G,\xi}$ depends only on the squared norms $|F_2|_{\mathcal F^{\beta,\sigma}}^2, |G|_{\mathcal F^{\beta+\frac{1}{2},\sigma}}^2$ and  $\mathbb E |F_1(0)|^2$ and $ \mathbb E |A^\beta\xi|^2.$ Furthermore, $X$ satisfies the estimates
\begin{equation}\label{semilinear evolution equationExpectationAbetaXSquare}
\mathbb E|X(t)|^2+\mathbb E |A^\beta X(t)|^2 \leq C_{F_1,F_2,G,\xi},    \hspace{2cm} t\in [0,T_{F_1,F_2,G,\xi}],
\end{equation}
and 
\begin{equation}\label{semilinear evolution equationExpectationAetaXSquare}
\mathbb E |A^\eta X(t)|^2 \leq C_{F_1,F_2,G,\xi} \Big[t^{-2(\eta-\beta)}+t^{2(1+\beta-2\eta)} +t^{2(1-\eta)}\Big], \hspace{0.5cm} t\in (0,T_{F_1,F_2,G,\xi}]
\end{equation}
with some constant $C_{F_1,F_2,G,\xi}$ depending only on $|F_2|_{\mathcal F^{\beta,\sigma}}^2, |G|_{\mathcal F^{\beta+\frac{1}{2},\sigma}}^2, \mathbb E |F_1(0)|^2$ and $ \mathbb E |A^\beta\xi|^2.$ 
\end{theorem}
\begin{proof}
We shall use the fixed point theorem for contractions to prove existence and uniqueness of a local solution.
For each $S\in (0,T]$, we set the underlying space:
\begin{align*}
\Xi (S)=&\{Y\in \mathcal C((0,S];\mathcal D(A^\eta)) \cap  \mathcal C([0,S];\mathcal D(A^\beta)) \text{  such that  }\\
& \sup_{0<t\leq S} t^{2(\eta-\beta)} \mathbb E|A^\eta Y(t)|^2+ \sup_{0\leq t\leq S}\mathbb E|A^\beta Y(t)|^2 <\infty \}.
\end{align*}
Then  up to indistinguishability, $\Xi (S)$  is a Banach space with norm
\begin{equation} \label{norminXi(S)}
|Y|_{\Xi (S)}=\Big[\sup_{0<t\leq S} t^{2(\eta-\beta)} \mathbb E|A^\eta Y(t)|^2+ \sup_{0\leq t\leq S}\mathbb E|A^\beta Y(t)|^2\Big]^{\frac{1}{2}}. 
\end{equation}

Let fix a constant $\kappa>0$  such that 
\begin{equation}  \label{Eq45}
\frac{\kappa^2}{2}> C_1\vee C_2,
\end{equation}
 where two constants $C_1$ and $C_2$ will be fixed below.
Consider a subset $\Upsilon(S) $ of $\Xi (S)$ which consists of all  function $Y\in \Xi (S)$  such that
\begin{equation}  \label{Upsilon(S)Definition}
\max\left\{\sup_{0<t\leq S} t^{2(\eta-\beta)} \mathbb E|A^\eta Y(t)|^2,
 \sup_{0\leq t\leq S}\mathbb E|A^\beta Y(t)|^2\right\} \leq \kappa^2.
\end{equation}
Obviously, $\Upsilon(S) $  is a nonempty closed subset of $\Xi (S)$.

For $Y\in \Upsilon(S)$, we define a function on $[0,S]$
\begin{equation}  \label{DefinitionOfFunctionPhi}
\Phi Y(t)=S(t) \xi+\int_0^t S(t-s)[F_1(Y(s))+F_2(s)]ds + \int_0^t S(t-s)G(s) dw_s.
\end{equation}
Our goal is then to verify that $\Phi$ is a contraction mapping from $\Upsilon(S)$ into itself, provided that $S$ is sufficiently small, and that the fixed point of $\Phi$ is the desired solution of \eqref{semilinear evolution equation}. For this purpose, we divide the proof into some steps.

{\bf Step 1}. Let us show that $ \Phi Y \in \Upsilon(S)$ for every $Y \in \Upsilon(S). $ Let  $Y\in \Xi (S)$. 
Due to \eqref{AbetaLipschitzcondition} and \eqref{Upsilon(S)Definition}, we observe that
\begin{align}
\mathbb E|F_1(Y(t))|^2 &\leq \mathbb E[c_{F_1}|A^\eta Y(t)|+c_{F_1}|A^\beta Y(t)|+|F_1(0)|]^2\notag \\
&\leq 3[c_{F_1}^2 \mathbb E |A^\eta Y(t)|^2+c_{F_1}^2 \mathbb E|A^\beta Y(t)|^2 +\mathbb E|F_1(0)|^2]\label{Eq10}\\
&\leq 3[c_{F_1}^2 \kappa^2 t^{2(\beta-\eta)} + c_{F_1}^2 \kappa^2  +\mathbb E|F_1(0)|^2], \hspace{1cm} t\in (0,S].  \label{Eq3}
\end{align}

First, we shall show that  $ \Phi Y$ satisfies \eqref{Upsilon(S)Definition}.  For $\theta \in [\beta, \frac{1}{2})$, from  \eqref{Upsilon(S)Definition} and \eqref{DefinitionOfFunctionPhi}, we have
\begin{align*}
&t^{2(\theta-\beta)}\mathbb E|A^\theta\{\Phi Y\}(t)|^2   \notag\\
 \leq & 3t^{2(\theta-\beta)}\mathbb E \Big[ |A^\theta S(t) \xi|^2+\Big|\int_0^tA^\theta S(t-s)[F_1(Y(s))+F_2(s)]ds\Big|^2 \notag\\
&+\Big |\int_0^tA^\theta S(t-s) G(s) dw_s\Big|^2 \Big]  \notag\\
\leq & 3t^{2(\theta-\beta)}  |A^{\theta-\beta} S(t)|^2\mathbb E|A^\beta \xi|^2\notag\\
&+6t^{2(\theta-\beta)}\mathbb E\Big|\int_0^tA^\theta S(t-s)F_1(Y(s))ds\Big|^2 \notag\\
&+6t^{2(\theta-\beta)}\mathbb E\Big|\int_0^t A^\theta S(t-s)F_2(s)ds\Big|^2 \notag\\
& +3c(E)t^{2(\theta-\beta)} \int_0^t|A^\theta S(t-s) G(s)|^2 ds. \notag
\end{align*}
On the account of \eqref{FbetasigmaSpaceProperty} and \eqref{EstimateIofAnu}, we observe that
\begin{align*}
&t^{2(\theta-\beta)}\mathbb E|A^\theta\{\Phi Y\}(t)|^2   \notag\\
\leq  &  3\iota_{\theta-\beta}^2 \mathbb E|A^\beta \xi|^2+6t^{2(\theta-\beta)} \iota_\theta^2 \mathbb E \Big[\int_0^t (t-s)^{-\theta}|F_1(Y(s))|ds\Big]^2\notag\\
&+6t^{2(\theta-\beta)} \iota_\theta^2 |F_2|_{\mathcal F^{\beta,\sigma} }^2 \Big[\int_0^t(t-s)^{-\theta}s^{\beta-1}ds\Big]^2 \notag\\
&+3c(E) \iota_\theta^2  |G|_{\mathcal F^{\beta+\frac{1}{2},\sigma}}^2  t^{2(\theta-\beta)}\int_0^t(t-s)^{-2\theta} s^{2\beta-1} ds  \notag\\
\leq &  3\iota_{\theta-\beta}^2 \mathbb E|A^\beta \xi|^2+6t^{1+2(\theta-\beta)} \iota_\theta^2  \int_0^t (t-s)^{-2\theta}\mathbb E |F_1(Y(s))|^2ds\notag\\
&+6 \iota_\theta^2 |F_2|_{\mathcal F^{\beta,\sigma} }^2 {\bf B}(\beta,1-\theta)^2 +3c(E) \iota_\theta^2  |G|_{\mathcal F^{\beta+\frac{1}{2},\sigma}}^2 {\bf B}(2\beta,1-2\theta).  \notag
\end{align*}
In view of \eqref{Eq3}, we obtain that
\begin{align}
&t^{2(\theta-\beta)}\mathbb E|A^\theta\{\Phi Y\}(t)|^2   \notag\\
\leq &  3\iota_{\theta-\beta}^2 \mathbb E|A^\beta \xi|^2\notag\\
&+18t^{1+2(\theta-\beta)} \iota_\theta^2  \int_0^t (t-s)^{-2\theta}[c_{F_1}^2 \kappa^2 s^{2(\beta-\eta)} + c_{F_1}^2 \kappa^2  +\mathbb E|F_1(0)|^2]ds\notag\\
&+6 \iota_\theta^2 |F_2|_{\mathcal F^{\beta,\sigma} }^2 {\bf B}(\beta, 1-\theta)^2+3c(E) \iota_\theta^2  |G|_{\mathcal F^{\beta+\frac{1}{2},\sigma}}^2 {\bf B}(2\beta, 1-2\theta)  \notag\\
= &  3\iota_{\theta-\beta}^2 \mathbb E|A^\beta \xi|^2+18 \iota_\theta^2 c_{F_1}^2 \kappa^2 t^{1+2(\theta-\eta)}\int_0^t (t-s)^{-2\theta} s^{2(\beta-\eta)} ds\notag\\
&+\frac{18 \iota_\theta^2 [ c_{F_1}^2 \kappa^2  +\mathbb E|F_1(0)|^2] }{1-2\theta} t^{2(1-\beta)}\notag\\
&+6 \iota_\theta^2 |F_2|_{\mathcal F^{\beta,\sigma} }^2 {\bf B}(\beta, 1-\theta)^2+3c(E) \iota_\theta^2  |G|_{\mathcal F^{\beta+\frac{1}{2},\sigma}}^2 {\bf B}(2\beta, 1-2\theta)  \notag\\
= &  3\iota_{\theta-\beta}^2 \mathbb E|A^\beta \xi|^2+6 \iota_\theta^2 |F_2|_{\mathcal F^{\beta,\sigma} }^2 {\bf B}(\beta, 1-\theta)^2\label{tthetabetaAthetaPhiYestimate}\\
&+3c(E) \iota_\theta^2  |G|_{\mathcal F^{\beta+\frac{1}{2},\sigma}}^2 {\bf B}(2\beta, 1-2\theta)\notag\\
&+ 18\iota_\theta^2 c_{F_1}^2 \kappa^2 {\bf B}( 1+2\beta-2\eta, 1-2\theta) t^{2(1+\beta-2\eta)}\notag\\
&+\frac{18 \iota_\theta^2 [ c_{F_1}^2 \kappa^2  +\mathbb E|F_1(0)|^2] }{1-2\theta} t^{2(1-\beta)}. \notag
\end{align}
We apply these estimates with  $\theta=\eta$ and $\theta=\beta$. Then it is observed that if $C_1$ and $C_2$ are fixed in such a way that
\begin{equation} \label{Eq42}
\begin{aligned}
C_1>&3\iota_{\eta-\beta}^2 \mathbb E|A^\beta \xi|^2+6 \iota_\eta^2 |F_2|_{\mathcal F^{\beta,\sigma} }^2 {\bf B}(\beta, 1-\eta)^2\\
&+3c(E) \iota_\eta^2 |G|_{\mathcal F^{\beta+\frac{1}{2},\sigma}}^2 {\bf B}(2\beta, 1-2\eta),\\
C_2>&3\iota_0^2 \mathbb E|A^\beta \xi|^2+6 \iota_\beta^2 |F_2|_{\mathcal F^{\beta,\sigma} }^2 {\bf B}(\beta, 1-\beta)^2\\
&+3c(E) \iota_\beta^2  |G|_{\mathcal F^{\beta+\frac{1}{2},\sigma}}^2 {\bf B}(2\beta, 1-2\beta),
\end{aligned}
\end{equation}
and if   $S$ is sufficiently small,    we have
\begin{align}
t^{2(\eta-\beta)}\mathbb E|A^\eta\{\Phi Y\}(t)|^2   \leq & C_1+ 18\iota_\eta^2 c_{F_1}^2 \kappa^2 {\bf B}( 1+2\beta-2\eta, 1-2\eta) t^{2(1+\beta-2\eta)}\notag\\
&+\frac{18 \iota_\eta^2 [ c_{F_1}^2 \kappa^2  +\mathbb E|F_1(0)|^2] }{1-2\eta} t^{2(1-\beta)}\notag\\
 \leq& \frac{\kappa^2}{2}+ 18\iota_\eta^2 c_{F_1}^2 \kappa^2 {\bf B}( 1+2\beta-2\eta, 1-2\eta) t^{2(1+\beta-2\eta)}\notag\\
 &+\frac{18 \iota_\eta^2 [ c_{F_1}^2 \kappa^2  +\mathbb E|F_1(0)|^2] }{1-2\eta} t^{2(1-\beta)}\notag\\
<&\kappa^2, \hspace{3cm} t\in (0,S],\label{Eq43}
\end{align}
and 
\begin{align}
\mathbb E|A^\beta\{\Phi Y\}(t)|^2  
  \leq& C_2+ 18\iota_\beta^2 c_{F_1}^2 \kappa^2 {\bf B}( 1+2\beta-2\eta, 1-2\beta) t^{2(1+\beta-2\eta)}\notag\\
  &+\frac{18 \iota_\beta^2 [ c_{F_1}^2 \kappa^2  +\mathbb E|F_1(0)|^2] }{1-2\beta} t^{2(1-\beta)}\notag\\
 \leq&\frac{\kappa^2}{2}+18\iota_\beta^2 c_{F_1}^2 \kappa^2 {\bf B}( 1+2\beta-2\eta, 1-2\beta) t^{2(1+\beta-2\eta)}\notag\\
  &+\frac{18 \iota_\beta^2 [ c_{F_1}^2 \kappa^2  +\mathbb E|F_1(0)|^2] }{1-2\beta} t^{2(1-\beta)}\notag\\
<&\kappa^2, \hspace{3cm} t\in (0,S]. \label{Eq44}
\end{align}
 We thus have shown that
 \begin{equation*} 
\max\left\{\sup_{0<t\leq S} t^{2(\eta-\beta)} \mathbb E|A^\eta \Phi Y(t)|^2,
 \sup_{0\leq t\leq S}\mathbb E|A^\beta \Phi Y(t)|^2\right\} \leq \kappa^2.
\end{equation*}
 This means that $ \Phi Y$ satisfies \eqref{Upsilon(S)Definition}.  
 
 Now, we  shall show that 
 $$\Phi (Y)\in  \mathcal C((0,S];\mathcal D(A^\eta))\cap \mathcal C([0,S];\mathcal D(A^\beta)).$$
We divide $\Phi Y$ into two parts: $\Phi Y(t)=\Psi Y(t)+ I_2(t),$
where 
$${\Psi Y}(t)=S(t) \xi+\int_0^t S(t-s)[F_1(Y(s))+F_2(s)]ds,$$
and $I_2(t)$ is defined by \eqref{XI1I2}.

As seen in {\bf Step 2} of Theorem \ref{regularity theorem autonomous linear evolution equation}, 
$$
\begin{cases}
A^\beta I_2(t)=A^\beta \int_0^t S(t-s) G(s) dw_s =\int_0^t A^\beta S(t-s)G(s) dw_s,\\
A^\beta I_2 \in \mathcal C([0,S];E).
\end{cases}
$$
Furthermore, by using \eqref {FbetasigmaSpaceProperty} and \eqref{EstimateIofAnu}, we have
\begin{align*}
&\int_0^t |A^\eta  S(t-s) G(s) |^2 ds \leq \iota_\eta^2  |G|_{\mathcal F^{\beta+\frac{1}{2},\sigma}}^2 \int_0^t (t-s)^{-2\eta} s^{2\beta-1}ds\\
&=\iota_\eta^2  |G|_{\mathcal F^{\beta+\frac{1}{2},\sigma}}^2 {\bf B}(2\beta, 1-2\eta) t^{2(\beta-\eta)}< \infty, \hspace{1cm} t\in (0,S].
\end{align*}
By the definition of stochastic integrals, $\int_0^t A^\eta  S(t-s) G(s)dw_s, \,t \in (0,S] $ is well-defined and belongs to the space  $ \mathcal C((0,S];E).$ We thus have verified that  
$$I_2\in  \mathcal C((0,S];\mathcal D(A^\eta))\cap \mathcal C([0,S];\mathcal D(A^\beta)).$$
Therefore, to finish {\bf Step 1}, it suffices to show that 
$$\Psi Y\in  \mathcal C((0,S];\mathcal D(A^\eta))\cap \mathcal C([0,S];\mathcal D(A^\beta)).$$
For $0<s<t\leq S$, by the semigroup property we observe that
\begin{align*}
\Psi Y(t)-\Psi Y(s)=&S(t-s)S(s)\xi+ S(t-s)\int_0^s S(s-r)[F_1(Y(r))+F_2(r)]dr\\
& -\Psi Y(s)+\int_s^t S(t-r)[F_1(Y(r))+F_2(r)]dr\\
=&[S(t-s)-I]\Psi Y(s)+\int_s^t S(t-r)[F_1(Y(r))+F_2(r)]dr.
\end{align*}
Let $\rho\in (\frac{1}{2}, 1-\eta)$. In view of \eqref{EstimateIofAnu} and \eqref{EstimateIofS(t)-IA-theta}, we have
\begin{align*}
&|A^\eta[\Psi Y(t)-\Psi Y(s)]|  \notag\\
\leq &|[S(t-s)-I]A^{-\rho}|  |A^{\eta+\rho}\Psi Y(s)|+\int_s^t |A^\eta S(t-r)| [|F_1(Y(r))|+|F_2(r)|]dr \notag\\
\leq &\frac{\iota_{1-\rho}(t-s)^\rho }{\rho} \Big |A^{\eta+\rho}\Big[S(s) \xi+\int_0^s S(s-r)[F_1(Y(r))+F_2(r)]dr\Big]\Big| \notag \\
&+\iota_\eta\int_s^t (t-r)^{-\eta} [|F_1(Y(r))|+|F_2(r)|]dr \notag\\
\leq &\frac{\iota_{1-\rho}(t-s)^\rho }{\rho}  |A^{\eta+\rho-\beta}S(s)| |A^\beta  \xi|\notag\\
&+\frac{\iota_{1-\rho}(t-s)^\rho }{\rho} \int_0^s |A^{\eta+\rho} S(s-r)| |F_1(Y(r))|dr \notag \\
&+\frac{\iota_{1-\rho}(t-s)^\rho }{\rho} \int_0^s |A^{\eta+\rho} S(s-r)| |F_2(r)|dr\notag\\
&+\iota_\eta\int_s^t (t-r)^{-\eta} |F_1(Y(r))|dr+\iota_\eta\int_s^t (t-r)^{-\eta} |F_2(r)|dr \notag\\
\leq &\frac{\iota_{1-\rho}\iota_{\eta+\rho-\beta}(t-s)^\rho }{\rho}  s^{-\eta-\rho+\beta} |A^\beta  \xi|\notag\\
&+\frac{\iota_{1-\rho}\iota_{\eta+\rho} (t-s)^\rho }{\rho}\int_0^s  (s-r)^{-\eta-\rho} |F_1(Y(r))|dr \notag\\
&+\frac{\iota_{1-\rho}\iota_{\eta+\rho} |F_2|_{\mathcal F^{\beta,\sigma}}(t-s)^\rho }{\rho} \int_0^s (s-r)^{-\eta-\rho}   r^{\beta-1} dr \notag \\
& +\iota_\eta\int_s^t (t-r)^{-\eta} |F_1(Y(r))|dr\notag \\
&+\iota_\eta|F_2|_{\mathcal F^{\beta,\sigma}}\int_s^t (t-r)^{-\eta} r^{\beta-1} dr \notag\\
= &\frac{\iota_{1-\rho} \iota_{\eta+\rho-\beta} }{\rho}|A^\beta  \xi| s^{\beta-\eta-\rho}(t-s)^\rho \notag\\
&+\frac{\iota_{1-\rho} \iota_{\eta+\rho}|F_2|_{\mathcal F^{\beta,\sigma}}{\bf B}(\beta,1-\eta-\rho)}{\rho}  s^{\beta-\eta-\rho}(t-s)^\rho \notag \\
&+\iota_\eta|F_2|_{\mathcal F^{\beta,\sigma}}\int_s^t (t-r)^{-\eta}r^{\beta-1}dr \notag\\
&+\frac{\iota_{1-\rho}\iota_{\eta+\rho} (t-s)^\rho }{\rho}\int_0^s  (s-r)^{-\eta-\rho} |F_1(Y(r))|dr \notag\\
& +\iota_\eta\int_s^t (t-r)^{-\eta} |F_1(Y(r))|dr.\notag 
\end{align*}
Dividing $\beta-1$ as $\beta-1=(\eta+\rho-1)+(\beta-\eta-\rho),$ we have
\begin{align*}
\int_s^t (t-r)^{-\eta}  r^{\beta-1} dr &\leq \int_s^t (t-r)^{-\eta}  (r-s)^{\eta+\rho-1} s^{\beta-\eta-\rho} dr\\
&= {\bf B}(\eta+\rho,1-\eta) s^{\beta-\eta-\rho}(t-s)^\rho.
\end{align*}
Hence, 
\begin{align*}
&|A^\eta[\Psi Y(t)-\Psi Y(s)]|  \notag\\
\leq  &\frac{\iota_{1-\rho}\iota_{\eta+\rho-\beta} }{\rho} |A^\beta  \xi| s^{\beta-\eta-\rho}(t-s)^\rho \notag\\
&+\Big[\frac{\iota_{1-\rho} \iota_{\eta+\rho} {\bf B}(\beta,1-\eta-\rho) }{\rho} +\iota_\eta{\bf B}(\eta+\rho,1-\eta)\Big]|F_2|_{\mathcal F^{\beta,\sigma}}s^{\beta-\eta-\rho}(t-s)^\rho \notag\\
&+\frac{\iota_{1-\rho}\iota_{\eta+\rho}}{\rho} (t-s)^\rho \int_0^s  (s-r)^{-\eta-\rho} |F_1(Y(r))|dr  \notag\\
&+\iota_\eta\int_s^t (t-r)^{-\eta} |F_1(Y(r))|dr.\notag 
\end{align*}
Taking expectation of the square of the both hand sides of the above inequality, we see that 
\begin{align*}
\mathbb E&|A^\eta[\Psi Y(t)-\Psi Y(s)]|^2 \notag\\
\leq &\frac{4\iota_{1-\rho}^2\iota_{\eta+\rho-\beta}^2}{\rho^2}  \mathbb E|A^\beta  \xi|^2 s^{2(\beta-\eta-\rho)}(t-s)^{2\rho} \notag\\
& +4\Big[\frac{\iota_{1-\rho}\iota_{\eta+\rho} {\bf B}(\beta,1-\eta-\rho) }{\rho} +\iota_\eta{\bf B}(\eta+\rho,1-\eta)\Big]^2\notag\\
&\times |F_2|_{\mathcal F^{\beta,\sigma}}^2s^{2(\beta-\eta-\rho)}(t-s)^{2\rho} \notag\\
&+\frac{4\iota_{1-\rho}^2\iota_{\eta+\rho}^2}{\rho^2} (t-s)^{2\rho}\mathbb E\Big[ \int_0^s  (s-r)^{-\eta-\rho} |F_1(Y(r))|dr\Big]^2  \notag\\
&+4\iota_\eta^2\mathbb E\Big[\int_s^t (t-r)^{-\eta} |F_1(Y(r))|dr\Big]^2.\notag
\end{align*}
Since 
\begin{align*}
&\Big[ \int_0^s  (s-r)^{-\eta-\rho} |F_1(Y(r))|dr\Big]^2  \\
&=\Big[ \int_0^s  (s-r)^{\frac{-\eta-\rho}{2}} (s-r)^{\frac{-\eta-\rho}{2}}|F_1(Y(r))|dr\Big]^2  \\
&\leq \int_0^s  (s-r)^{-\eta-\rho}dr \int_0^s (s-r)^{-\eta-\rho}|F_1(Y(r))|^2dr\\
&=\frac{s^{1-\eta-\rho}}{1-\eta-\rho} \int_0^s (s-r)^{-\eta-\rho}|F_1(Y(r))|^2dr,
\end{align*}
we have
\begin{align}
\mathbb E&|A^\eta[\Psi Y(t)-\Psi Y(s)]|^2 \label{Eq25}\\
\leq &\frac{4\iota_{1-\rho}^2  \iota_{\eta+\rho-\beta}^2}{\rho^2} \mathbb E|A^\beta  \xi|^2 s^{2(\beta-\eta-\rho)}(t-s)^{2\rho} \notag\\
& +4\Big[\frac{\iota_{1-\rho}\iota_{\eta+\rho} {\bf B}(\beta,1-\eta-\rho)}{\rho}  +\iota_\eta{\bf B}(\eta+\rho,1-\eta)\Big]^2\notag\\
&\times |F_2|_{\mathcal F^{\beta,\sigma}}^2s^{2(\beta-\eta-\rho)}(t-s)^{2\rho} \notag\\
&+\frac{4\iota_{1-\rho}^2\iota_{\eta+\rho}^2}{\rho^2(1-\eta-\rho)} (t-s)^{2\rho} s^{1-\eta-\rho}     \int_0^s  (s-r)^{-\eta-\rho} \mathbb E|F_1(Y(r))|^2dr  \notag\\
&+4\iota_\eta^2(t-s) \int_s^t (t-r)^{-2\eta}\mathbb E |F_1(Y(r))|^2dr.\notag
\end{align}
Using \eqref{Eq3}, we have 
\begin{align}
& \int_0^s  (s-r)^{-\eta-\rho} \mathbb E|F_1(Y(r))|^2dr \notag\\
 \leq &
3c_{F_1}^2 \kappa^2\int_0^s  (s-r)^{-\eta-\rho}  r^{2(\beta-\eta)} dr  +3[ c_{F_1}^2 \kappa^2  +\mathbb E|F_1(0)|^2]\int_0^s  (s-r)^{-\eta-\rho} dr  \notag\\
 = &
3c_{F_1}^2 \kappa^2 {\bf B}(1+2\beta-2\eta,1-\eta-\rho) s^{1+2\beta-3\eta-\rho}  \label{Eq38}\\
&+\frac{3[ c_{F_1}^2 \kappa^2  +\mathbb E|F_1(0)|^2] s^{1-\eta-\rho}}{1-\eta-\rho}, \notag
\end{align}
and
\begin{align}
&\int_s^t (t-r)^{-2\eta}\mathbb E |F_1(Y(r))|^2dr\notag\\
&\leq 3\int_s^t (t-r)^{-2\eta}[c_{F_1}^2 \kappa^2 r^{2(\beta-\eta)} + c_{F_1}^2 \kappa^2  +\mathbb E|F_1(0)|^2]dr\notag\\
&=3 c_{F_1}^2 \kappa^2 \int_s^t (t-r)^{-2\eta}r^{2(\beta-\eta)} dr+\frac{3[ c_{F_1}^2 \kappa^2  +\mathbb E|F_1(0)|^2]}{1-2\eta} (t-s)^{1-2\eta}.\label{Eq39}
\end{align}
Dividing $2(\beta-\eta)$ as $2(\beta-\eta)=(\beta-\frac{1}{2})+(\frac{1}{2}+\beta-2\eta),$ we have
\begin{align}
\int_s^t (t-r)^{-2\eta} r^{2(\beta-\eta)} dr&\leq \int_s^t (t-r)^{-2\eta} (r-s)^{\beta-\frac{1}{2}} t^{\frac{1}{2}+\beta-2\eta}dr\notag\\
&={\bf B}(\frac{1}{2}+\beta,1-2\eta) t^{\frac{1}{2}+\beta-2\eta} (t-s)^{\frac{1}{2}+\beta-2\eta}.  \label{Eq40}
\end{align}
By virtue of  \eqref{Eq38},  \eqref{Eq39} and  \eqref{Eq40}, it follows from  \eqref{Eq25} that 
\begin{align*}
\mathbb E&|A^\eta[\Psi Y(t)-\Psi Y(s)]|^2 \notag\\
\leq  &\frac{4\iota_{1-\rho}^2  \iota_{\eta+\rho-\beta}^2}{\rho^2} \mathbb E|A^\beta  \xi|^2 s^{2(\beta-\eta-\rho)}(t-s)^{2\rho} \notag\\
& +4\Big[\frac{\iota_{1-\rho}\iota_{\eta+\rho} {\bf B}(\beta,1-\eta-\rho)}{\rho}  +\iota_\eta{\bf B}(\eta+\rho,1-\eta)\Big]^2\notag\\
&\times |F_2|_{\mathcal F^{\beta,\sigma}}^2s^{2(\beta-\eta-\rho)}(t-s)^{2\rho}  \notag\\
&+\frac{12\iota_{1-\rho}^2\iota_{\eta+\rho}^2c_{F_1}^2 \kappa^2 {\bf B}(1+2\beta-2\eta,1-\eta-\rho) }{\rho^2(1-\eta-\rho)} (t-s)^{2\rho} s^{2(1+\beta-2\eta-\rho)}  \notag\\
&+\frac{12\iota_{1-\rho}^2\iota_{\eta+\rho}^2[ c_{F_1}^2 \kappa^2  +\mathbb E|F_1(0)|^2] }{\rho^2(1-\eta-\rho)^2} (t-s)^{2\rho} s^{2(1-\eta-\rho)}  \notag\\
&+12 \iota_\eta^2c_{F_1}^2 \kappa^2 {\bf B}(\frac{1}{2}+\beta,1-2\eta) t^{\frac{1}{2}+\beta-2\eta} (t-s)^{\frac{3}{2}+\beta-2\eta}\notag\\
&+\frac{12\iota_\eta^2[ c_{F_1}^2 \kappa^2  +\mathbb E|F_1(0)|^2]}{1-2\eta} (t-s)^{2(1-\eta)}.
\end{align*}
Since this estimate holds for any $\rho\in (\frac{1}{2}, 1-\eta),$ and since  $1<\frac{3}{2}+\beta-2\eta<2(1-\eta)$, the Kolmogorov test then provides that $A^\eta \Psi Y $ is H\"older continuous on $(0,S]$ with an arbitrary exponent smaller than $\frac{1+2\beta}{4}-\eta.$ As a consequence,  $\Psi Y\in \mathcal C((0,S];\mathcal D(A^\eta)).$

Similarly, we also have
\begin{align*}
\mathbb E&|A^\beta[\Psi Y(t)-\Psi Y(s)]|^2 \notag\\
\leq  &\frac{4\iota_{1-\rho}^2  \iota_\rho^2}{\rho^2} \mathbb E|A^\beta  \xi|^2 s^{-2\rho}(t-s)^{2\rho} \notag\\
& +4\Big[\frac{\iota_{1-\rho}\iota_{\beta+\rho} {\bf B}(\beta,1-\beta-\rho)}{\rho}  +\iota_\beta{\bf B}(\beta+\rho,1-\beta)\Big]^2\notag\\
&\times |F_2|_{\mathcal F^{\beta,\sigma}}^2s^{-2\rho}(t-s)^{2\rho}  \notag\\
&+\frac{12\iota_{1-\rho}^2\iota_{\beta+\rho}^2c_{F_1}^2 \kappa^2 {\bf B}(1,1-\beta-\rho) }{\rho^2(1-\beta-\rho)} (t-s)^{2\rho} s^{2(1-\eta-\rho)}  \notag\\
&+\frac{12\iota_{1-\rho}^2\iota_{\beta+\rho}^2[ c_{F_1}^2 \kappa^2  +\mathbb E|F_1(0)|^2] }{\rho^2(1-\beta-\rho)^2} (t-s)^{2\rho} s^{2(1-\beta-\rho)}  \notag\\
&+12 \iota_\beta^2c_{F_1}^2 \kappa^2 {\bf B}(\frac{1}{2}+\beta,1-2\beta) t^{\frac{1}{2}-\beta} (t-s)^{\frac{3}{2}-\beta}\notag\\
&+\frac{12\iota_\beta^2[ c_{F_1}^2 \kappa^2  +\mathbb E|F_1(0)|^2]}{1-2\beta} (t-s)^{2(1-\beta)}.
\end{align*}
The Kolmogorov test again provides that  $\Psi Y\in \mathcal C((0,S];\mathcal D(A^\beta)).$

It remains to show that $A^\beta \Psi Y$ is continuous at $t=0$. Indeed, we already know by Theorem \ref{regularity theorem autonomous linear evolution equation} that 
$$A^\beta\Big[S(t) \xi+\int_0^t S(t-s)F_2(s)ds + \int_0^t S(t-s)G(s) dw_s\Big]$$
is continuous at $t=0.$  Meanwhile,  using \eqref{EstimateIofAnu} and \eqref{Eq3}, we have
\begin{align} 
\mathbb E&\Big|A^\beta \int_0^t S(t-s)F_1(Y(s))ds\Big|^2 \notag\\
 & \leq\mathbb E\Big[\int_0^t |A^\beta S(t-s)| |F_1(Y(s))|ds\Big]^2\notag\\
 &\leq \iota_\beta^2 \mathbb E  \Big[\int_0^t (t-s)^{-\beta}|F_1(Y(s))| ds\Big]^2\notag\\
 &\leq \iota_\beta^2 t   \int_0^t (t-s)^{-2\beta}\mathbb E|F_1(Y(s))|^2 ds\notag\\
 &\leq 3\iota_\beta^2 t   \int_0^t (t-s)^{-2\beta}[c_{F_1}^2 \kappa^2 s^{2(\beta-\eta)} + c_{F_1}^2 \kappa^2  +\mathbb E|F_1(0)|^2] ds\notag\\
 &= 3\iota_\beta^2 \Big[c_{F_1}^2 \kappa^2 {\bf B}(1+2\beta-2\eta,1-2\beta)t^{2(1-\eta)} +\frac{c_{F_1}^2 \kappa^2  +\mathbb E|F_1(0)|^2 }{1-2\beta} t^{2(1-\beta)}\Big]  \label{ExpectationOfAbetaSF1YSquare}\\ 
 & \to 0 \hspace{2cm} \text{ as } t\searrow 0. \notag
\end{align}
Therefore, there exists a decreasing sequence $\{t_n\}$ converging to $0$ such that 
$$\lim_{n\to\infty} A^\beta \int_0^{t_n} S(t_n-s)F_1(Y(s))ds=0.$$
By the continuity of $A^\beta \int_0^{t_n} S(t_n-s)F_1(Y(s))ds$ on $(0,S]$, we conclude that 
$$\lim_{t\to 0} A^\beta \int_0^t S(t-s)F_1(Y(s))ds=0,$$
i.e. $A^\beta \int_0^t S(t-s)F_1(Y(s))ds$ is continuous at $t=0$.
We thus have shown that  $A^\beta\Psi Y$
is continuous at $t=0.$ 

{\bf Step 2}. Let us show that $\Phi$ is a contraction mapping of $\Xi (S)$, provided $S>0$ is sufficiently small. Let $Y_1,Y_2\in \Xi (S)$ and $\theta\in [0, \frac{1}{2}).$ From \eqref{DefinitionOfFunctionPhi}, we see that
\begin{align*}
&t^{2(\theta-\beta)}\mathbb E|A^\theta[\Phi Y_1(t)-\Phi Y_2(t)]|^2  \notag \\
=&t^{2(\theta-\beta)} \mathbb E\Big|\int_0^t A^\theta S(t-s)[F_1(Y_1(s))-F_1(Y_2(s))]ds\Big|^2 \notag\\
\leq&t^{2(\theta-\beta)} \mathbb E\Big[\int_0^t |A^\theta S(t-s)| |F_1(Y_1(s))-F_1(Y_2(s))|ds\Big]^2. \notag
\end{align*}
By virtue  of  \eqref{EstimateIofAnu},  \eqref{AbetaLipschitzcondition} and  \eqref{norminXi(S)}, we have
\begin{align*}
&t^{2(\theta-\beta)}\mathbb E|A^\theta[\Phi Y_1(t)-\Phi Y_2(t)]|^2  \notag \\
\leq& c_{F_1}^2\iota_\theta^2 t^{2(\theta-\beta)}\notag \\
&\times \mathbb E\Big[\int_0^t (t-s)^{-\theta}  \{ |A^\eta(Y_1(s)-Y_2(s))|+  |A^\beta (Y_1(s)-Y_2(s))|\}ds\Big]^2 \notag\\
\leq& c_{F_1}^2\iota_\theta^2 t^{1+2(\theta-\beta)}\notag \\
&\times \mathbb E\int_0^t (t-s)^{-2\theta}  \{ |A^\eta(Y_1(s)-Y_2(s))|+  |A^\beta (Y_1(s)-Y_2(s))|\}^2ds \notag\\
\leq& 2c_{F_1}^2\iota_\theta^2 t^{1+2(\theta-\beta)} \int_0^t (t-s)^{-2\theta}  \mathbb E |A^\eta(Y_1(s)-Y_2(s))|^2ds \notag\\
&+2c_{F_1}^2\iota_\theta^2 t^{1+2(\theta-\beta)} \int_0^t (t-s)^{-2\theta}  \mathbb E |A^\beta (Y_1(s)-Y_2(s))|^2ds \notag\\
\leq&2c_{F_1}^2\iota_\theta^2 t^{1+2(\theta-\beta)} \int_0^t  (t-s)^{-2\theta} s^{2(\beta-\eta)} |Y_1-Y_2|_{{\Xi (S)}}^2ds \notag\\
&+2c_{F_1}^2\iota_\theta^2 t^{1+2(\theta-\beta)} \int_0^t  (t-s)^{-2\theta}  |Y_1-Y_2|_{{\Xi (S)}}^2ds \notag\\
=&2c_{F_1}^2\iota_\theta^2  \Big[{\bf B}(1+2\beta-2\eta,1-2\theta) t^{2(1-\eta)}+\frac{t^{2(1-\beta)}}{1-2\theta}\Big]   |Y_1-Y_2|_{{\Xi (S)}}^2 \notag\\
\leq &2c_{F_1}^2  \Big[\iota_\theta^2{\bf B}(1+2\beta-2\eta,1-2\theta)+\frac{\iota_\theta^2 S^{2(\eta-\beta)}}{1-2\theta}\Big]   S^{2(1-\eta)} |Y_1-Y_2|_{{\Xi (S)}}^2. \notag
\end{align*}
Applying these estimates with $\theta=\eta$ and $\theta=\beta$, we conclude that
\begin{align}
&|\Phi Y_1-\Phi Y_2|_{{\Xi (S)}}^2 \notag\\
=&\Big[\sup_{0<t\leq S} t^{2(\eta-\beta)} \mathbb E|A^\eta [\Phi Y_1(t)-\Phi Y_2(t)]|^2+ \sup_{0\leq t\leq S}\mathbb E|A^\beta  [\Phi Y_1(t)-\Phi Y_2(t)]|^2\Big]^{\frac{1}{2}}\notag\\
\leq &2c_{F_1}^2  \Big[\iota_\eta^2{\bf B}(1+2\beta-2\eta,1-2\eta)+\iota_\beta^2{\bf B}(1+2\beta-2\eta,1-2\beta)\label{Eq46}\\
&\hspace*{0.5cm}+\frac{\iota_\eta^2S^{2(\eta-\beta)}}{1-2\eta}+\frac{\iota_\beta^2S^{2(\eta-\beta)}}{1-2\beta}\Big]   S^{2(1-\eta)} |Y_1-Y_2|_{{\Xi (S)}}^2. \notag 
\end{align}
This shows that $\Phi$ is contraction in $\Xi (S),$ provided $S>0$ is sufficiently small.

{\bf Step 3}.  Let us show existence of a local mild solution.
Let $S>0$ be sufficiently small in such a way that $\Phi$ maps $\Upsilon(S)$ into itself and is contraction with respect to the norm of $\Xi (S).$ Due to {\bf Step 1} and {\bf Step 2}, $S=T_{F_1,F_2,G,\xi}$ can be determined by $\mathbb E|F_1(0)|^2, |F_2|_{\mathcal F^{\beta,\sigma}}^2, |G|_{\mathcal F^{\beta+\frac{1}{2},\sigma}}^2$ and $\mathbb E |A^\beta\xi|^2.$ Thanks to the fixed point theorem, there exists a unique function $X\in \Upsilon(T_{F_1,F_2,G,\xi})$ such that $X=\Phi X$. 
This means that $X$ is a mild solution of  \eqref{semilinear evolution equation} in the function space:
$$
X\in  \mathcal C((0,T_{F_1,F_2,G,\xi}];\mathcal D(A^\eta))\cap \mathcal C([0,T_{F_1,F_2,G,\xi}];\mathcal D(A^\beta)).
$$

{\bf Step 4}. Let us verify the estimate \eqref{semilinear evolution equationExpectationAbetaXSquare}.
We have
\begin{align}
X(t)=&\Big[S(t) \xi+\int_0^t S(t-s)F_2(s)ds + \int_0^t S(t-s)G(s) dw_s \Big] \label{Eq8}\\
&+\int_0^t S(t-s)F_1(X(s))ds \notag \\
=&X_1(t)+X_2(t), \hspace{2cm} t\in [0,T_{F_1,F_2,G,\xi}]. \notag
\end{align}
On the account of  Theorem \ref{regularity theorem autonomous linear evolution equation}, we have an estimate for the first term
\begin{align*} 
\mathbb E |X_1(t)|^2 \leq  \rho_1[\mathbb E|\xi|^2 + |F|_{\mathcal F^{\beta,\sigma}}^2+  |G|_{\mathcal F^{\beta+\frac{1}{2},\sigma}}^2],
\end{align*}
where $\rho_1$ is a positive constant depending only on $A, \beta$ and $\sigma$.

For the second term, in view of  \eqref{EstimateIofS(t)Maximum} and  \eqref{Eq3}, we observe that 
\begin{align*} 
\mathbb E| X_2(t)|^2 
&\leq\mathbb E\Big[\int_0^t | S(t-s)| |F_1(X(s))|ds\Big]^2\notag\\
 &\leq t  \int_0^t \iota_0^2\mathbb E|F_1(X(s))|^2 ds\notag\\
 &\leq  3t   \int_0^t \iota_0^2[c_{F_1}^2 \kappa^2 s^{2(\beta-\eta)} + c_{F_1}^2 \kappa^2  +\mathbb E|F_1(0)|^2] ds\notag\\
&=\frac{3c_{F_1}^2 \kappa^2\iota_0^2}{1+2(\beta-\eta)} t^{2(1+\beta-\eta)}+3[c_{F_1}^2 \kappa^2  +\mathbb E|F_1(0)|^2]\iota_0^2 t^2 
 \end{align*}
for every $t\in [0,T_{F_1,F_2,G,\xi}]$. Hence, 
\begin{align} 
\mathbb E| X(t)|^2 \leq &2 \mathbb E| X_1(t)|^2 +2 \mathbb E| X_2(t)|^2  \notag\\
\leq &2\rho_1[\mathbb E|\xi|^2 +\mathbb E |F|_{\mathcal F^{\beta,\sigma}}^2+  |G|_{\mathcal F^{\beta+\frac{1}{2},\sigma}}^2]\notag\\
&+\frac{6c_{F_1}^2 \kappa^2\iota_0^2}{1+2(\beta-\eta)} t^{2(1+\beta-\eta)}+6[c_{F_1}^2 \kappa^2  +\mathbb E|F_1(0)|^2]\iota_0^2 t^2\notag\\
\leq &2\rho_1[\mathbb E|\xi|^2 +\mathbb E |F|_{\mathcal F^{\beta,\sigma}}^2+  |G|_{\mathcal F^{\beta+\frac{1}{2},\sigma}}^2]\label{Eq6}\\
&+\frac{6c_{F_1}^2 \kappa^2\iota_0^2}{1+2(\beta-\eta)} T^{2(1+\beta-\eta)}+6[c_{F_1}^2 \kappa^2  +\mathbb E|F_1(0)|^2]\iota_0^2 T^2  \notag
 \end{align}
for every $t\in [0,T_{F_1,F_2,G,\xi}]$. 

On the other hand, in virtue of  \eqref{EstimateIofS(t)Maximum}, \eqref{Eq4}, \eqref{Eq5}, \eqref{ExpectationOfAbetaSF1YSquare} and \eqref{Eq8}, for every $t\in [0,T_{F_1,F_2,G,\xi}]$ we observe  that
\begin{align}
\mathbb E&| A^\beta X(t)|^2\notag\\
\leq & 4\mathbb E |S(t)A^\beta \xi|^2+4\mathbb E\Big|\int_0^t A^\beta S(t-s)F_2(s)ds\Big|^2 \notag\\
&+4\mathbb E\Big|\int_0^t A^\beta S(t-s)F_1(X(s))ds\Big|^2+ 4\mathbb E\Big|\int_0^t A^\beta S(t-s)G(s) dw_s\Big|^2\notag\\
\leq & 4\iota_0^2 \mathbb E |A^\beta \xi|^2+4\iota_\beta^2 |F_2|_{\mathcal F^{\beta,\sigma}}^2 {\bf B}(\beta,1-\beta)^2\notag\\
&+ 12\iota_\beta^2 \Big[c_{F_1}^2 \kappa^2 {\bf B}(1+2\beta-2\eta,1-2\beta)t^{2(1-\eta)} +\frac{c_{F_1}^2 \kappa^2  +\mathbb E|F_1(0)|^2 }{1-2\beta} t^{2(1-\beta)}\Big]\notag \\
&+ 4c(E)\int_0^t |A^\beta  S(t-s) G(s) |^2 ds\notag\\
\leq & 4\iota_0^2 \mathbb E |A^\beta \xi|^2+4\iota_\beta^2 |F_2|_{\mathcal F^{\beta,\sigma}}^2 {\bf B}(\beta,1-\beta)^2\notag\\
&+ 12\iota_\beta^2 \Big[c_{F_1}^2 \kappa^2 {\bf B}(1+2\beta-2\eta,1-2\beta)t^{2(1-\eta)} +\frac{c_{F_1}^2 \kappa^2  +\mathbb E|F_1(0)|^2 }{1-2\beta} t^{2(1-\beta)}\Big]\notag \\
&+ 4c(E)\iota_\beta^2  |G|_{\mathcal F^{\beta+\frac{1}{2},\sigma}}^2 {\bf B}(2\beta, 1-2\beta)\notag\\
\leq & 4\iota_0^2 \mathbb E |A^\beta \xi|^2+4\iota_\beta^2 |F_2|_{\mathcal F^{\beta,\sigma}}^2 {\bf B}(\beta,1-\beta)^2+ 4c(E)\iota_\beta^2  |G|_{\mathcal F^{\beta+\frac{1}{2},\sigma}}^2 {\bf B}(2\beta, 1-2\beta)\label{Eq7}\\
&+ 12\iota_\beta^2 \Big[c_{F_1}^2 \kappa^2 {\bf B}(1+2\beta-2\eta,1-2\beta)T^{2(1-\eta)} +\frac{c_{F_1}^2 \kappa^2  +\mathbb E|F_1(0)|^2 }{1-2\beta} T^{2(1-\beta)}\Big]. \notag
\end{align}
Combining  \eqref{Eq6} and \eqref{Eq7}, we obtain  \eqref{semilinear evolution equationExpectationAbetaXSquare}.

{\bf Step 5}. Let us verify the estimate  \eqref{semilinear evolution equationExpectationAetaXSquare}. From \eqref{EstimateIofAnu} and  \eqref{Eq8}, 
 we observe that
\begin{align*}
\mathbb E&| A^\eta X(t)|^2\notag\\
\leq & 4\mathbb E |A^\eta S(t) \xi|^2+4\Big[\int_0^t |A^\eta S(t-s)| |F_2(s)|ds\Big]^2 \notag\\
&+4\mathbb E\Big[\int_0^t |A^\eta S(t-s)| |F_1(X(s))|ds\Big]^2+ 4c(E) \int_0^t |A^\eta S(t-s)G(s)|^2 ds\notag\\
=& 4\mathbb E |A^{\eta-\beta} S(t) A^\beta \xi|^2+4J_1+4J_2+4J_3\notag\\
\leq & 4\iota_{\eta-\beta}^2   \mathbb E |A^\beta \xi|^2 t^{-2(\eta-\beta)}+4J_1+4J_2+4J_3, \hspace{2cm} t\in (0,T_{F_1,F_2,G,\xi}].\notag
\end{align*}

We shall give estimates for $J_1, J_2$ and $J_3$.
For $J_1$, similarly to \eqref{Eq4}, we have
$$J_1\leq \iota_\eta^2 |F|_{\mathcal F^{\beta,\sigma}}^2  {\bf B}(\beta,1-\eta)^2 t^{2(\beta-\eta)}.$$
For $J_2$, similarly to \eqref{ExpectationOfAbetaSF1YSquare}, we conclude that
$$J_2\leq 3\iota_\eta^2 \Big[c_{F_1}^2 \kappa^2 {\bf B}(1+2\beta-2\eta,1-2\eta)t^{2(1+\beta-2\eta)} +\frac{c_{F_1}^2 \kappa^2  +\mathbb E|F_1(0)|^2 }{1-2\eta} t^{2(1-\eta)}\Big].$$
For $J_3$, similarly to \eqref{Eq5}, we obtain that
$$J_3\leq c(E) \iota_\eta^2  |G|_{\mathcal F^{\beta+\frac{1}{2},\sigma}}^2 {\bf B}(2\beta, 1-2\eta)  t^{2(\beta-\eta)}.$$
Thus, \eqref{semilinear evolution equationExpectationAetaXSquare} is verified.

{\bf Step 6}.  Let us finally show uniqueness of the local mild solution. Let $\bar X$ be any other local mild solution to  \eqref{semilinear evolution equation} on the interval $[0,T_{F_1,F_2,G,\xi}]$ which belongs to the space $\mathcal C((0,T_{F_1,F_2,G,\xi}];\mathcal D(A^\eta))\cap \mathcal C([0,T_{F_1,F_2,G,\xi}];\mathcal D(A^\beta))$.

The formula
$$\bar X(t)=S(t) \xi+\int_0^t S(t-s)[F_1(\bar X(s))+F_2(s)]ds + \int_0^t S(t-s)G(s) dw_s $$
jointed with \eqref{Eq8} yields that
$$X(t)-\bar X(t)=\int_0^t S(t-s)[F_1(X(s))-F_1(\bar X(s))] ds, \quad\quad t\in [0,T_{F_1,F_2,G,\xi}].$$
We can then repeat the same arguments as in {\bf Step 2} to deduce that
\begin{align}
&|X-\bar X|_{\Xi (\bar T)}^2\label{Eq9}\\
&\leq 2c_{F_1}^2  \Big[\iota_\eta^2{\bf B}(1+2\beta-2\eta,1-2\eta)+\iota_\beta^2{\bf B}(1+2\beta-2\eta,1-2\beta)+\frac{\iota_\eta^2{\bar T}^{2(\eta-\beta)}}{1-2\eta}\notag\\
&\hspace{0.8cm}+\frac{\iota_\beta^2{\bar T}^{2(\eta-\beta)}}{1-2\beta}\Big]   {\bar T}^{2(1-\eta)} |X-\bar X|_{\Xi (\bar T)}^2 \hspace{1cm} \text{ for any } \bar T\in(0,T_{F_1,F_2,G,\xi}].\notag
\end{align}
Let $\bar T$ be a positive constant such that
\begin{align*}
 &2c_{F_1}^2  \Big[\iota_\eta^2{\bf B}(1+2\beta-2\eta,1-2\eta)+\iota_\beta^2{\bf B}(1+2\beta-2\eta,1-2\beta)\notag\\
&+\frac{\iota_\eta^2\bar T^{2(\eta-\beta)}}{1-2\eta}+\frac{\iota_\beta^2\bar T^{2(\eta-\beta)}}{1-2\beta}\Big]   \bar T^{2(1-\eta)}<1.
\end{align*}
It then follows from  \eqref{Eq9} that $X(t)=\bar X(t)$ a.s. for every $t\in [0,\bar T].$
 We repeat the same procedure with  initial time $\bar T$ and  initial value $X(\bar T)=\bar X(\bar T)$ to derive that $X(\bar T +t)=\bar X(\bar T +t)$ a.s. for every $t\in [0,\bar T].$ This means that $X(t)=\bar X(t)$ a.s. on a larger interval $[0,2\bar T].$ We continue this procedure by finite times, the extended interval can cover the given interval $[0,T_{F_1,F_2,G,\xi}].$ Therefore,  for every $t\in [0,T_{F_1,F_2,G,\xi}],$ $X(t)=\bar X(t)$ a.s.
\end{proof}
\begin{corollary}[global existence]  \label{semilinear evolution equationTheorem1corollary}
Assume that in Theorem \ref{semilinear evolution equationTheorem1}  the constant $ C_{F_1,F_2,G,\xi}$ is independent of  $T_{F_1,F_2,G,\xi}$ for every $\xi\in \mathcal D(A^\beta)$ such that $\mathbb E|A^\beta \xi|^2<\infty$. Then \eqref{semilinear evolution equation}  possesses a unique mild solution on the interval $[0,T].$
\end{corollary}
\begin{proof}
Let extend the functions $F_2$ and $G$ to functions $\bar F_2$ and $\bar G$ defined on the whole interval $[0,\infty)$ by putting $\bar F_2(t)\equiv F_2(T)$ and $\bar G(t)\equiv G(T)$ for $T<t<\infty$. It is obvious that
$$|\bar F_2|_{\mathcal F^{\beta,\sigma}((a,b];E)}\leq |F_2|_{\mathcal F^{\beta,\sigma}((a,b];E)}\, \text{ and }\, |\bar G|_{\mathcal F^{\beta+\frac{1}{2},\sigma}((a,b];E)}\leq |G|_{\mathcal F^{\beta+\frac{1}{2},\sigma}((a,b];E)}$$
for any interval $(a,b]\subset [0,\infty).$ 

Let $\bar \xi=X(\frac{T_{F_1,F_2,G,\xi}}{2}).$ We consider the problem
\begin{equation}  \label{Eq24}
\begin{cases}
dY+AYdt=[F_1(Y)+F_2(t)]dt+G(t)dw_t, \quad\quad \frac{T_{F_1,F_2,G,\xi}}{2}<t<\infty,\\
Y(\frac{T_{F_1,F_2,G,\xi}}{2})=\bar \xi.
\end{cases}
\end{equation}
Thanks to Theorem \ref{semilinear evolution equationTheorem1}, \eqref{Eq24} has  a local mild solution $Y(t)$ for every 
$T_{F_1,F_2,G,\xi}\leq t\leq \frac{3T_{F_1,F_2,G,\xi}}{2}. $
By the uniqueness of solution,    $X(t)=Y(t)$  a.s. for $t\in  [\frac{T_{F_1,F_2,G,\xi}}{2},T_{F_1,F_2,G,\xi}].$ This means that we have constructed a local mild solution to \eqref{semilinear evolution equation} on the interval $[0,\frac{3T_{F_1,F_2,G,\xi}}{2}]. $ The independence  of  $T_{F_1,F_2,G,\xi}$ with respect to  $ C_{F_1,F_2,G,\xi}$  allows us to continue this procedure unlimitedly. Each time the local solution is extended over the fixed length $\frac{T_{F_1,F_2,G,\xi}}{2}$ of interval.  So, by finite times, the extended interval can cover the interval $[0,T]$.
\end{proof}
\subsection{Regularity for more regular initial data}
This subsection shows regularity of local mild solutions for more regular initial values.  For any $F_2\in \mathcal F^{\gamma, \sigma}((0,T];E),$ $ \max\{\beta,\frac{1}{2}-\eta\}<\gamma< \frac{1}{2},$ and any initial value $\xi\in \mathcal D(A^\gamma)$ satisfying   $\mathbb E|A^\gamma\xi|^2<\infty$, we can verify a stronger regularity than \eqref{semilinear evolution equationTheorem1Regularity} for the local mild solution of  \eqref{semilinear evolution equation}.
\begin{theorem}\label{semilinear evolution equationMoreRegular}
Let \eqref{spectrumsectorialdomain}, \eqref{resolventnorm}, {\rm (H1)} and {\rm (H3)}   be satisfied.
 Let $F_2\in \mathcal F^{\gamma, \sigma}((0,T];E)$,  $ \max\{\beta,\frac{1}{2}-\eta\}<\gamma< \frac{1}{2}$, and  $\xi\in \mathcal D(A^\gamma)$ such that $\mathbb E|A^\gamma\xi|^2<\infty$.  Then \eqref{semilinear evolution equation} possesses a unique mild solution $X$ in the function space:
\begin{equation*}
X\in  \mathcal C((0,T_{F_1,F_2,G,\xi}];\mathcal D(A^\eta))\cap \mathcal C([0,T_{F_1,F_2,G,\xi}];\mathcal D(A^\gamma)),
\end{equation*}
where $T_{F_1,F_2,G,\xi}$ depends only on the squared norms $|F_2|_{\mathcal F^{\beta,\sigma}}^2, |G|_{\mathcal F^{\beta+\frac{1}{2},\sigma}}^2$ and $\mathbb E |F_1(0)|^2$ and $ \mathbb E |A^\gamma\xi|^2.$ In addition, $X$ satisfies the estimate
\begin{align}
\mathbb E|X(t)|^2+t^{2(\gamma-\beta)}\mathbb E |A^\gamma X(t)|^2 \leq C_{F_1,F_2,G,\xi},    \quad\quad t\in [0,T_{F_1,F_2,G,\xi}]\label{semilinear evolution equationExpectationAbetaXSquareMoreRegular}
\end{align}
with some constant $C_{F_1,F_2,G,\xi}$ depending only on $|F_2|_{\mathcal F^{\beta,\sigma}}^2, |G|_{\mathcal F^{\beta+\frac{1}{2},\sigma}}^2, \mathbb E |F_1(0)|^2$ and $ \mathbb E |A^\gamma\xi|^2.$ 
\end{theorem}
\begin{proof}
Since the embedding of $\mathcal D(A^\gamma)$  in $\mathcal D(A^\beta)$ is continuous, we have  $\xi\in \mathcal D(A^\beta)$ and $\mathbb E|A^\beta\xi|^2<\infty.$ In addition, due to \eqref{FbetaFgammasigmaSpaceProperty}, 
 we also have 
 $F_2\in \mathcal F^{\beta, \sigma}((0,T];E).$ Therefore, by Theorem \ref{semilinear evolution equationTheorem1},  \eqref{semilinear evolution equation} possesses a unique mild solution $X$ in the function space \eqref{semilinear evolution equationTheorem1Regularity}
which satisfies the estimate
\eqref{semilinear evolution equationExpectationAbetaXSquare}. 

It now remains only to show that $X\in  \mathcal C([0,T_{F_1,F_2,G,\xi}];\mathcal D(A^\gamma))$ and that $X$ satisfies \eqref{semilinear evolution equationExpectationAbetaXSquareMoreRegular}. For this purpose, we shall divide the proof into three steps. Throughout the proof, $C_{F_1,F_2,G,\xi}$ denotes a universal constant which is determined in each occurrence by $F_1,F_2,G$ and $\xi$.

{\bf Step 1}. Let us verify that 
\begin{equation}  \label{Eq11}
\mathbb E|A^\eta X(t)|^2\leq C_{F_1,F_2,G,\xi} t^{-2\varrho}, \hspace{1cm} t\in (0,T_{F_1,F_2,G,\xi}],
\end{equation}
where $\varrho=\max\{1-\eta-\gamma,\eta-\beta\}\in (0,\frac{1}{2})$.
From \eqref{Eq8}, we have
\begin{align*}
A^\eta X(t)=&A^{\eta-\gamma} S(t) A^\gamma \xi+\int_0^t A^\eta S(t-s)[F_1(X(s))+F_2(s)]ds  \notag\\
& + \int_0^t A^\eta S(t-s)G(s) dw_s. \notag
\end{align*}
Using \eqref{FbetasigmaSpaceProperty} and  \eqref{EstimateIofAnu}, we then obtain that 
\begin{align*}
\mathbb E &|A^\eta X(t)|^2 \notag\\
\leq & 4 |A^{\eta-\gamma} S(t)|^2\mathbb E |A^\gamma \xi|^2
+4\mathbb E\Big|\int_0^t |A^\eta S(t-s)| |F_1(X(s))|ds\Big|^2\notag\\
&+4\mathbb E\Big[\int_0^t |A^\eta S(t-s)||F_2(s)|ds\Big]^2
 + 4 \mathbb E \Big|\int_0^t A^\eta S(t-s)G(s) dw_s\Big|^2 \notag\\
\leq & C_\xi t^{\min\{-2(\eta-\gamma),0\}} 
+4\iota_\eta^2 \mathbb E\Big[\int_0^t (t-s)^{-\eta} |F_1(X(s))|ds\Big]^2\notag\\
& +4\iota_\eta^2 |F_2|_{\mathcal F^{\gamma,\sigma}}^2\Big[\int_0^t (t-s)^{\eta-1} s^{\gamma-1}ds\Big]^2\notag\\
&+ 4c(E) \int_0^t |A^\eta S(t-s)|^2|G(s)|^2 ds \notag\\
\leq & C_\xi t^{\min\{-2(\eta-\gamma),0\}} +4\iota_\eta^2 t \int_0^t (t-s)^{-2\eta} \mathbb E|F_1(X(s))|^2ds\notag\\
& +4\iota_\eta^2 |F_2|_{\mathcal F^{\gamma,\sigma}}^2 {\bf B}(\gamma,\eta)^2 t^{2(\eta+\gamma-1)} + 4c(E) \iota_\eta^2 |G|_{\mathcal F^{\beta+\frac{1}{2},\sigma}}^2 \int_0^t (t-s)^{-2\eta} s^{2\beta-1} ds. \notag
\end{align*}
By \eqref{semilinear evolution equationExpectationAbetaXSquare} and \eqref{Eq10},  we have
\begin{align*}
4\iota_\eta^2 &t \int_0^t (t-s)^{-2\eta} \mathbb E|F_1(X(s))|^2ds\notag\\
\leq &12\iota_\eta^2 t \int_0^t (t-s)^{-2\eta} [c_{F_1}^2 \mathbb E |A^\eta X(s)|^2+c_{F_1}^2 \mathbb E|A^\beta X(s)|^2 +\mathbb E|F_1(0)|^2]ds\notag\\
\leq &12\iota_\eta^2c_{F_1}^2 t \int_0^t (t-s)^{-2\eta} \mathbb E |A^\eta X(s)|^2ds\notag\\
&+\frac{12\iota_\eta^2 [c_{F_1}^2 C_{F_1,F_2,G,\xi}+\mathbb E|F_1(0)|^2] t^{2(1-\eta)}}{1-2\eta}.\notag
\end{align*}
Therefore, 
\begin{align}
\mathbb E &|A^\eta X(t)|^2 \notag\\
\leq & C_\xi t^{\min\{-2(\eta-\gamma),0\}} +12\iota_\eta^2c_{F_1}^2 t \int_0^t (t-s)^{-2\eta} \mathbb E |A^\eta X(s)|^2ds\notag\\
&+\frac{12\iota_\eta^2 [c_{F_1}^2 C_{F_1,F_2,G,\xi}+\mathbb E|F_1(0)|^2] t^{2(1-\eta)}}{1-2\eta}\notag\\
& +4\iota_\eta^2 |F_2|_{\mathcal F^{\gamma,\sigma}}^2 {\bf B}(\gamma,\eta)^2 t^{2(\eta+\gamma-1)} + 4c(E) \iota_\eta^2 |G|_{\mathcal F^{\beta+\frac{1}{2},\sigma}}^2 \int_0^t (t-s)^{-2\eta} s^{2\beta-1} ds\notag\\
\leq  & C_\xi t^{\min\{-2(\eta-\gamma),0\}} +C_{F_1,F_2,G,\xi} t^{-2\varrho}\notag\\
&+12\iota_\eta^2c_{F_1}^2 t\int_0^t (t-s)^{-2\eta} \mathbb E |A^\eta X(s)|^2ds,  \notag\\
\leq  & C_{F_1,F_2,G,\xi} t^{-2\varrho} 
+12\iota_\eta^2c_{F_1}^2 t\int_0^t (t-s)^{-2\eta} \mathbb E |A^\eta X(s)|^2ds, \hspace{0.5cm} t\in (0,T_{F_1,F_2,G,\xi}],\label{Eq12}
\end{align}
here we  used the estimates
\begin{equation*}
\begin{cases}
\begin{aligned}
&t^{2(\eta+\gamma-1)}\leq Ct^{-2\varrho},  \hspace{2cm} &t\in (0,T_{F_1,F_2,G,\xi}],\\
&t^{2(\beta-\eta)} \leq Ct^{-2\varrho},  \hspace{2cm} &t\in (0,T_{F_1,F_2,G,\xi}],\\
&t^{\min\{-2(\eta-\gamma),0\}} \leq Ct^{-2\varrho}, \hspace{2cm} &t\in (0,T_{F_1,F_2,G,\xi}]\,\,
\end{aligned}
\end{cases}
\end{equation*}
with some constant $C$ depending only on $T_{F_1,F_2,G,\xi}.$ 
Then the function $q(t)= t^{2\varrho} \mathbb E |A^\eta X(t)|^2$ satisfies
\begin{equation} \label{Eq48}
q(t)\leq  C_{F_1,F_2,G,\xi}+12\iota_\eta^2c_{F_1}^2 t^{1+2\varrho}\int_0^t (t-s)^{-2\eta} s^{-2\varrho}q(s)ds.
\end{equation}

Let us solve  the integral inequality as follows. Let $\epsilon>0$ denote a small parameter. For $0\leq t\leq \epsilon$ we have
\begin{align*}
q(t)\leq&  C_{F_1,F_2,G,\xi}+12\iota_\eta^2c_{F_1}^2  t^{1+2\varrho}\int_0^t (t-s)^{-2\eta} s^{-2\varrho}ds \sup_{s\in[0,\epsilon]} q(s)\\
=&C_{F_1,F_2,G,\xi}+12\iota_\eta^2c_{F_1}^2  t^{2(1-\eta)}{\bf B}(1-2\varrho,1-2\eta) \sup_{s\in[0,\epsilon]} q(s)\\
\leq&C_{F_1,F_2,G,\xi}+12\iota_\eta^2c_{F_1}^2  \epsilon^{2(1-\eta)}{\bf B}(1-2\varrho,1-2\eta) \sup_{s\in[0,\epsilon]} q(s).
\end{align*}
Hence,
$$[1-12\iota_\eta^2c_{F_1}^2  \epsilon^{2(1-\eta)}{\bf B}(1-2\varrho,1-2\eta)]\sup_{s\in[0,\epsilon]} q(s) \leq C_{F_1,F_2,G,\xi}.$$
If $\epsilon$ is taken sufficiently small so that $12\iota_\eta^2c_{F_1}^2  \epsilon^{2(1-\eta)}{\bf B}(1-2\varrho,1-2\eta)\leq \frac{1}{2}$, then we obtain that
\begin{equation}  \label{Eq13}
\sup_{s\in[0,\epsilon]} q(s) \leq C_{F_1,F_2,G,\xi}.
\end{equation}
Meanwhile, for $\epsilon<t\leq T_{F_1,F_2,G,\xi}$ we have
\begin{align*}
q(t)\leq&  C_{F_1,F_2,G,\xi}+12\iota_\eta^2c_{F_1}^2  t^{1+2\varrho}\int_0^\epsilon (t-s)^{-2\eta} s^{-2\varrho}ds \sup_{s\in[0,\epsilon]} q(s)\\
&+12\iota_\eta^2c_{F_1}^2  t^{1+2\varrho}\int_\epsilon^t (t-s)^{-2\eta} \max\{ \epsilon^{-2\varrho},  T^{-2\varrho}\} q(s)ds \\
\leq&  C_{F_1,F_2,G,\xi}+12\iota_\eta^2c_{F_1}^2\max\{ \epsilon^{-2\varrho},  T^{-2\varrho}\}T^{1+2\varrho}\int_\epsilon^t (t-s)^{-2\eta}  q(s)ds.
\end{align*}
Lemma \ref{Integral inequality of Volterra type} then provides that 
\begin{equation}  \label{Eq14}
\sup_{t\in [\epsilon, T_{F_1,F_2,G,\xi}]} q(t)\leq C_{F_1,F_2,G,\xi}.
\end{equation}

Thus, \eqref{Eq11} follows from \eqref{Eq13} and \eqref{Eq14}.

{\bf Step 2}. Let us verify that $X(t)\in \mathcal D(A^\gamma)$ for every $t\in [0,T_{F_1,F_2,G,\xi}]$. 
In virtue of    \eqref{Eq8}, it suffices to show that the integrals 
$\int_0^t A^\gamma S(t-s)F_2(s)ds,$ $ \int_0^t A^\gamma  S(t-s)G(s) dw_s$ and $ \int_0^t A^\gamma S(t-s)F_1(X(s))ds  $ are well-defined.

The first integral exists. Indeed,  using \eqref{FbetasigmaSpaceProperty} and \eqref{EstimateIofAnu}, we have
\begin{align}
\int_0^t |A^\gamma S(t-s)F_2(s)|ds\leq & \iota_\gamma |F_2|_{\mathcal F^{\gamma, \sigma}}\int_0^t (t-s)^{-\gamma}  s^{\gamma-1} ds \notag\\
=&\iota_\gamma |F_2|_{\mathcal F^{\gamma, \sigma}} {\bf B}(\gamma, 1-\gamma) 
<\infty, \hspace{0.5cm} t\in [0,T_{F_1,F_2,G,\xi}].  \label{Eq15}
\end{align}

To show existence of the second integral, we have to verify that $\int_0^t |A^\gamma  S(t-s)G(s)|^2 ds<\infty.$  Indeed, by   \eqref{FbetasigmaSpaceProperty} and \eqref{EstimateIofAnu}, we see that
\begin{align}
\int_0^t |A^\gamma  &S(t-s)G(s)|^2 ds \leq  \int_0^t |A^\gamma  S(t-s)|^2|G(s)|^2 ds  \notag \\
\leq &\iota_\gamma^2 |G|_{\mathcal F^{\beta+\frac{1}{2}, \sigma}}^2\int_0^t (t-s)^{-2\gamma}  s^{2\beta-1} ds \notag\\
=&\iota_\gamma^2 |G|_{\mathcal F^{\beta+\frac{1}{2}, \sigma}}^2 {\bf B}(2\beta,1-2\gamma) t^{2(\beta-\gamma)} 
<\infty, \quad \quad t\in (0,T_{F_1,F_2,G,\xi}].  \label{Eq16}
\end{align}

We shall finish this step by showing  that the last integral is well-defined.
From  \eqref{EstimateIofAnu}, \eqref{semilinear evolution equationExpectationAbetaXSquare}, \eqref{Eq10} and \eqref{Eq11}, we observe that
\begin{align}
\mathbb E\int_0^t &|A^\gamma S(t-s)F_1(X(s))|^2ds  \notag\\
\leq & \int_0^t |A^\gamma S(t-s)|^2 \mathbb E|F_1(X(s))|^2ds \notag \\
\leq & 3\iota_\gamma^2 \int_0^t (t-s)^{-2\gamma} [c_{F_1}^2 \mathbb E |A^\eta X(s)|^2+c_{F_1}^2 \mathbb E|A^\beta X(s)|^2 +\mathbb E|F_1(0)|^2]ds\notag\\
\leq & 3\iota_\gamma^2C_{F_1,F_2,G,\xi} \int_0^t (t-s)^{-2\gamma}  [s^{-2\varrho}+1]ds\notag\\
=&3\iota_\gamma^2C_{F_1,F_2,G,\xi} \Big[{\bf B}(1-2\varrho, 1-2\gamma) t^{1-2\varrho-2\gamma}+\frac{t^{1-2\gamma}}{1-2\gamma}\Big]\label{Eq17}\\
<&\infty, \hspace{3cm} t\in (0,T_{F_1,F_2,G,\xi}].\notag
\end{align}
Consequently,  $\int_0^t |A^\gamma S(t-s)F_1(X(s))|^2ds<\infty$ a.s. Therefore, $$\int_0^t |A^\gamma S(t-s)F_1(X(s))|ds<\infty \quad\text{ a.s.} $$

{\bf Step 3}. Let us show the estimate  \eqref{semilinear evolution equationExpectationAbetaXSquareMoreRegular}. From \eqref{Eq8},   we observe that 
\begin{align*}
\mathbb E &|A^\gamma X(t)|^2 \\
\leq& 4  \mathbb E|A^\gamma S(t)\xi|^2 + 4\mathbb E\Big[\int_0^t |A^\gamma S(t-s)F_1(X(s))|ds\Big]^2\notag\\
&+ 4\Big[\int_0^t |A^\gamma S(t-s)F_2(s)|ds\Big]^2+4\mathbb E \Big|\int_0^t A^\gamma  S(t-s)G(s) dw_s\Big|^2\notag\\
\leq& 4  \mathbb E|A^\gamma S(t)\xi|^2 + 4t\mathbb E\int_0^t |A^\gamma S(t-s)F_1(X(s))|^2ds\notag\\
&+ 4\Big[\int_0^t |A^\gamma S(t-s)F_2(s)|ds\Big]^2+4c(E) \int_0^t |A^\gamma  S(t-s)G(s) |^2ds.\notag
\end{align*}
Using   \eqref{EstimateIofS(t)},  \eqref{Eq15}, \eqref{Eq16} and \eqref{Eq17}, we have
\begin{align*}
\mathbb E &|A^\gamma X(t)|^2 \\
\leq& 4\iota_0^2  e^{-2\nu t} \mathbb E|A^\gamma\xi|^2
+ C_{F_1,F_2,G,\xi} \Big[{\bf B}(1-2\varrho, 1-2\gamma) t^{2(1-\varrho-\gamma)}+\frac{t^{2(1-\gamma)}}{1-2\gamma}\Big]\notag\\
&+ 4\iota_\gamma^2 |F_2|_{\mathcal F^{\gamma, \sigma}}^2 {\bf B}(\gamma, 1-\gamma)^2+4c(E) \iota_\gamma^2 |G|_{\mathcal F^{\beta+\frac{1}{2}, \sigma}}^2 {\bf B}(2\beta,1-2\gamma) t^{2(\beta-\gamma)}\notag\\
\leq &C_{F_1,F_2,G,\xi}t^{-2(\gamma-\beta)},  \hspace{3cm}t\in (0,T_{F_1,F_2,G,\xi}],
\end{align*}
here we used the estimates
\begin{equation*}
\begin{cases}
\begin{aligned}
&4\iota_0^2  \mathbb E|A^\gamma\xi|^2  e^{-2\nu t}\leq Ct^{-2(\gamma-\beta)},  &\hspace{2cm} t\in (0,T_{F_1,F_2,G,\xi}],\\
&t^{2(1-\varrho-\gamma)}\leq Ct^{-2(\gamma-\beta)},  &\hspace{2cm} t\in (0,T_{F_1,F_2,G,\xi}],\\
&t^{2(1-\gamma)} <Ct^{-2(\gamma-\beta)}, & \hspace{2cm} t\in (0,T_{F_1,F_2,G,\xi}],\\
&4\iota_\gamma^2 |F_2|_{\mathcal F^{\gamma, \sigma}}^2 {\bf B}(\gamma, 1-\gamma)^2<Ct^{-2(\gamma-\beta)}, & \hspace{2cm} t\in (0,T_{F_1,F_2,G,\xi}]\,\,
\end{aligned}
\end{cases}
\end{equation*}
with some constant $C$ depending only on $T_{F_1,F_2,G,\xi}.$ 
Thus,
$$t^{2(\gamma-\beta)}\mathbb E |A^\gamma X(t)|^2 \leq C_{F_1,F_2,G,\xi},\hspace{2cm}t\in [0,T_{F_1,F_2,G,\xi}].$$
This, together with \eqref{semilinear evolution equationExpectationAbetaXSquare}, then derives \eqref{semilinear evolution equationExpectationAbetaXSquareMoreRegular}.

{\bf Step 4}. Let us verify that $A^\gamma X\in \mathcal C([0,T_{F_1,F_2,G,\xi}];E)$.
Similarly to \eqref{Eq10}, for $t\in (0,T_{F_1,F_2,G,\xi}]$ we have 
$$\mathbb E|F_1(X(t))|^2 \leq 3[c_{F_1}^2 \mathbb E |A^\eta X(t)|^2+c_{F_1}^2 \mathbb E|A^\beta X(t)|^2 +\mathbb E|F_1(0)|^2].$$
Due to  \eqref{semilinear evolution equationExpectationAbetaXSquare} and \eqref{Eq11}, we observe that
\begin{align}
\mathbb E|F_1(X(t))|^2& \leq C_{F_1,F_2,G,\xi} ( t^{-2\varrho}+1)\notag\\
&\leq C_{F_1,F_2,G,\xi} t^{-2\varrho}, \hspace{2cm} t\in (0,T_{F_1,F_2,G,\xi}]. \label{Eq18}
\end{align}

First, we shall show that  $A^\gamma X\in \mathcal C((0,T_{F_1,F_2,G,\xi}];E)$.
By repeating the same argument as in verifying \eqref{Eq25} in {\bf Step 1} of the proof of Theorem \ref{semilinear evolution equationTheorem1}, for every $\rho\in (\frac{1}{2},1-\gamma)$ we see that 
\begin{align*}
\mathbb E&|A^\gamma[X(t)-X(s)]|^2 \notag\\
\leq &\frac{4\iota_{1-\rho}^2\iota_{\gamma+\rho-\beta}^2}{\rho^2}  \mathbb E|A^\beta  \xi|^2 s^{2(\beta-\gamma-\rho)}(t-s)^{2\rho} \notag\\
&+4\Big[\frac{\iota_{1-\rho} \iota_{\gamma+\rho} {\bf B}(\beta,1-\gamma-\rho) }{\rho} +\iota_\gamma{\bf B}(\gamma+\rho,1-\gamma)\Big]^2\notag\\
&\times |F_2|_{\mathcal F^{\beta,\sigma}}^2s^{2(\beta-\gamma-\rho)}(t-s)^{2\rho} \notag\\
&+\frac{4\iota_{1-\rho}^2\iota_{\gamma+\rho}^2(1-\gamma-\rho)}{\rho^2} (t-s)^{2\rho}  s^{1-\gamma-\rho}               \int_0^s  (s-r)^{-\gamma-\rho} \mathbb E|F_1(X(r))|^2dr  \notag\\
&+4\iota_\gamma^2(t-s) \int_s^t (t-r)^{-2\gamma}\mathbb E |F_1(X(r))|^2dr.\notag 
\end{align*}
Using \eqref{Eq18}, we obtain that 
\begin{align}
\mathbb E&|A^\gamma[X(t)-X(s)]|^2 \notag\\
\leq &\frac{4\iota_{1-\rho}^2\iota_{\gamma+\rho-\beta}^2}{\rho^2}  \mathbb E|A^\beta  \xi|^2 s^{2(\beta-\gamma-\rho)}(t-s)^{2\rho} \notag\\
&+4\Big[\frac{\iota_{1-\rho} \iota_{\gamma+\rho} {\bf B}(\beta,1-\gamma-\rho) }{\rho} +\iota_\gamma{\bf B}(\gamma+\rho,1-\gamma)\Big]^2\notag\\
&\times |F_2|_{\mathcal F^{\beta,\sigma}}^2s^{2(\beta-\gamma-\rho)}(t-s)^{2\rho} \notag\\
&+C_{F_1,F_2,G,\xi} (t-s)^{2\rho}  s^{1-\gamma-\rho}          \int_0^s  (s-r)^{-\gamma-\rho}  r^{-2\varrho} dr  \notag\\
&+C_{F_1,F_2,G,\xi}(t-s) \int_s^t (t-r)^{-2\gamma} r^{-2\varrho} dr\notag \\
\leq &C_{F_1,F_2,G,\xi}s^{2(\beta-\gamma-\rho)}(t-s)^{2\rho} \notag\\
&+C_{F_1,F_2,G,\xi}  s^{2(1-\gamma-\rho-\varrho)}  (t-s)^{2\rho} \notag\\
&+C_{F_1,F_2,G,\xi} (t-s) \int_s^t (t-r)^{-2\gamma} r^{-2\varrho} dr.\notag 
\end{align}
Let us estimate the latter integral. Fix $\epsilon \in (0, \min\{1-2\gamma, 2\varrho\})$. Since
$$r^{-2\varrho} =r^{-\epsilon} r^{\epsilon-2\varrho}<(r-s)^{-\epsilon} s^{\epsilon-2\varrho}, \hspace{2cm} r\in (s,t),$$
we have
\begin{align*}
\int_s^t (t-r)^{-2\gamma} r^{-2\varrho} dr\leq &s^{\epsilon-2\varrho}\int_s^t (t-r)^{-2\gamma}(r-s)^{-\epsilon}  dr\\
=&{\bf B}(1-\epsilon, 1-2\gamma) s^{\epsilon-2\varrho}(t-s)^{1-2\gamma-\epsilon}.
\end{align*}
Hence,
\begin{align*}
\mathbb E&|A^\gamma[X(t)-X(s)]|^2 \notag\\
\leq &C_{F_1,F_2,G,\xi}[s^{2(\beta-\gamma-\rho)}(t-s)^{2\rho}+  s^{2(1-\gamma-\rho-\varrho)}  (t-s)^{2\rho} + s^{\epsilon-2\varrho}(t-s)^{2-2\gamma-\epsilon}].\notag 
\end{align*}
Since $2\rho> 1$ and $2-2\gamma-\epsilon>1$, the Kolmogorov test then provides that $A^\gamma X\in \mathcal C((0,T_{F_1,F_2,G,\xi}];E)$.

Now, we shall verify that $A^\gamma X$ is continuos at $t=0$. By using \eqref{Eq18} (instead of \eqref{Eq3}), we will repeat the same argument as in showing the continuity of $A^\beta \Psi Y$ at $t=0$ in {\bf Step 1} of the proof of Theorem \ref{semilinear evolution equationTheorem1}. We then obtain the continuity of $A^\gamma X$ at $t=0$. Thus, it is concluded that $A^\gamma X\in \mathcal C([0,T_{F_1,F_2,G,\xi}];E)$.
\end{proof}
\subsection{Regular dependence of solutions on initial data}
Let $\mathcal B_1$ and $\mathcal B_2$ be bounded balls
\begin{align}
\mathcal B_1=\{f\in \mathcal F^{\beta,\sigma}((0,T];E):  |f|_{\mathcal F^{\beta,\sigma}}\leq R_1\}, \quad 0<R_1<\infty,  \label{mathcal B1}\\
\mathcal B_2=\{f\in \mathcal F^{\beta+\frac{1}{2},\sigma}((0,T];E):  |f|_{\mathcal F^{\beta+\frac{1}{2},\sigma}}\leq R_2\}, \quad 0<R_2<\infty \label{mathcal B2}
\end{align}
of the spaces $\mathcal F^{\beta,\sigma}((0,T];E)$ and $\mathcal F^{\beta+\frac{1}{2},\sigma}((0,T];E)$, respectively.
And let $B_A$ be a set of random variable 
\begin{equation}  \label{BABall}
B_A=\{\xi: \xi\in \mathcal D(A^\beta) \text{ a.s. and }    \mathbb E |A^\beta \xi|^2\leq R_3^2\}, \quad 0<R_3<\infty.
\end{equation}
 According to Theorem \ref{semilinear evolution equationTheorem1}, for every $F_2\in \mathcal B_1, G\in \mathcal B_2$ and $\xi\in B_A$, there exists a local solution of \eqref{semilinear evolution equation} on some interval $[0,T_{F_1,F_2,G,\xi}]$. Furthermore, by virtue of  {\bf Step 1} and {\bf Step 2} of the proof of Theorem 
\ref{semilinear evolution equationTheorem1}, 
\begin{equation} \label{Eq41}
\begin{aligned}
&\text{ there is a time } T_{\mathcal B_1, \mathcal B_2, B_A}>0 \text{  such that } \\
&[0,T_{\mathcal B_1, \mathcal B_2, B_A}]\subset [0,T_{F_1,F_2,G,\xi}]    \text{ for all  } (F_2,G,\xi)\in \mathcal B_1\times \mathcal B_2\times B_A. 
\end{aligned}
\end{equation}
Indeed, in view of \eqref{Eq43}, \eqref{Eq44} and \eqref{Eq46}, $T_{F_1,F_2,G,\xi}$ can be chosen to be any time $S$ satisfying the conditions
\begin{align*}
18\iota_\beta^2 &c_{F_1}^2 \kappa^2 {\bf B}( 1+2\beta-2\eta, 1-2\beta) S^{2(1+\beta-2\eta)}\notag\\
&+\frac{18 \iota_\beta^2 [ c_{F_1}^2 \kappa^2  +\mathbb E|F_1(0)|^2] }{1-2\beta} S^{2(1-\beta)} \leq \frac{\kappa^2}{2},\\
18\iota_\beta^2 &c_{F_1}^2 \kappa^2 {\bf B}( 1+2\beta-2\eta, 1-2\beta) S^{2(1+\beta-2\eta)}\notag\\
  &+\frac{18 \iota_\beta^2 [ c_{F_1}^2 \kappa^2  +\mathbb E|F_1(0)|^2] }{1-2\beta} S^{2(1-\beta)}\leq \frac{\kappa^2}{2},
\end{align*}
and
\begin{align*}
&2c_{F_1}^2  \Big[\iota_\eta^2{\bf B}(1+2\beta-2\eta,1-2\eta)+\iota_\beta^2{\bf B}(1+2\beta-2\eta,1-2\beta)\notag\\
&\hspace*{0.7cm}+\frac{\iota_\eta^2S^{2(\eta-\beta)}}{1-2\eta}+\frac{\iota_\beta^2S^{2(\eta-\beta)}}{1-2\beta}\Big]   S^{2(1-\eta)} <1,
\end{align*}
where $\kappa$ is defined by \eqref{Eq45} and \eqref{Eq42}.  Consequently, we can choose  $T_{F_1,F_2,G,\xi}$ such that it depends continuously on the norms $|F_2|_{\mathcal F^{\beta,\sigma}}, |G|_{\mathcal F^{\beta+\frac{1}{2},\sigma}}$ and $\mathbb E |A^\beta \xi|^2$.  This implies \eqref{Eq41}.
 
 We shall show continuous dependence of solutions on $(F_2,G,\xi)$ in the sense specified in the following theorem.
\begin{theorem}\label{continuityofsolutionsininitialdata}
Let \eqref{spectrumsectorialdomain}, \eqref{resolventnorm}, {\rm (H1)}, {\rm (H2)} and {\rm (H3)}   be satisfied.
Let $X$ and $\bar X$ be the solutions of  \eqref{semilinear evolution equation} for the data $(F_2,G,\xi)$ and $(\bar F_2,\bar G,\bar \xi)$ in $\mathcal B_1\times \mathcal B_2\times B_A$, respectively. Then there exists a constant $C_{\mathcal B_1, \mathcal B_2, B_A}$ depending only on $\mathcal B_1, \mathcal B_2$ and $ B_A$ such that
\begin{align}
&t^{2\eta}\mathbb E  |A^\eta[X(s)-\bar X(s)]|^2+t^{2\eta}\mathbb E|A^{\beta} [X(s)-\bar X(s)]|^2+ \mathbb E|X(t)-\bar X(t)|^2 \label{Eq22}\\
\leq &C_{\mathcal B_1, \mathcal B_2, B_A}[\mathbb E |\xi-\bar \xi|^2+ t^{2\beta}   |F_2-\bar F_2|_{\mathcal F^{\beta,\sigma}}^2 + t^{2\beta} |G-\bar G|_{\mathcal F^{\beta+\frac{1}{2},\sigma}}^2],\notag
\end{align}
and 
\begin{align}
&t^{2(\eta-\beta)}[\mathbb E  |A^\eta[X(s)-\bar X(s)]|^2+\mathbb E|A^{\beta} [X(s)-\bar X(s)]|^2] \label{Eq23}\\
& \leq C_{\mathcal B_1, \mathcal B_2, B_A}[\mathbb E |A^\beta(\xi-\bar \xi)|^2+    |F_2-\bar F_2|_{\mathcal F^{\beta,\sigma}}^2+ |G-\bar G|_{\mathcal F^{\beta+\frac{1}{2},\sigma}}^2] \notag
\end{align}
for every $ t\in (0, T_{\mathcal B_1, \mathcal B_2, B_A}].$
\end{theorem}
\begin{proof}
This theorem is proved by analogous arguments as in the proof of Theorem \ref{semilinear evolution equationTheorem1}. 

Indeed, let us first verify  \eqref{Eq22}. If $\mathbb E |\xi-\bar \xi|^2=\infty$, then the conclusion is obvious. Therefore, we may assume that 
$\mathbb E |\xi-\bar \xi|^2<\infty.$ 

First, we shall  give an estimate for 
$$t^{2\eta}\mathbb E [ |A^\eta[X(s)-\bar X(s)]|^2+|A^{\beta} [X(s)-\bar X(s)]|^2].$$
For $\theta\in [0,\frac{1}{2})$ and $0<t\leq T_{\mathcal B_1, \mathcal B_2, B_A}$, by using \eqref{FbetasigmaSpaceProperty}, \eqref{EstimateIofAnu} and \eqref{AbetaLipschitzcondition}, we observe that
\begin{align*}
&t^\theta |A^\theta[X(t)-\bar X(t)]|\\
= & \Big|t^\theta A^\theta S(t)(\xi-\bar \xi)+\int_0^t t^\theta A^\theta S(t-s)[F_1(X(s))-F_1(\bar X(s))]ds\notag\\
&+\int_0^t t^\theta A^\theta S(t-s) [F_2(s)-\bar F_2(s)]ds+\int_0^t t^\theta A^\theta S(t-s) [G(s)-\bar G(s)]dw_s\Big| \notag\\
\leq & \iota_\theta |\xi-\bar \xi|\notag\\
&+\iota_\theta c_{F_1}\int_0^t  t^\theta (t-s)^{-\theta}  [ |A^\eta[X(s)-\bar X(s)]|+|A^{\beta} [X(s)-\bar X(s)]|] ds\notag\\
&+\iota_\theta |F_2-\bar F_2|_{\mathcal F^{\beta,\sigma}} \int_0^t t^\theta (t-s)^{-\theta}s^{\beta-1}ds\notag\\
&+\Big|\int_0^t t^\theta A^\theta S(t-s) [G(s)-\bar G(s)]dw_s\Big| \notag\\
= & \iota_\theta |\xi-\bar \xi|+\iota_\theta |F_2-\bar F_2|_{\mathcal F^{\beta,\sigma}}  {\bf B}(\beta,1-\theta)t^\beta\notag\\
& +\iota_\theta c_{F_1}\int_0^t  t^\theta (t-s)^{-\theta}  [ |A^\eta[X(s)-\bar X(s)]|+|A^{\beta} [X(s)-\bar X(s)]|] ds\notag\\
&+\Big|\int_0^t t^\theta A^\theta S(t-s) [G(s)-\bar G(s)]dw_s\Big|. \notag
\end{align*}
Thus,
\begin{align}
&\mathbb E|t^\theta A^\theta[X(t)-\bar X(t)]|^2\notag\\
\leq & 4\iota_\theta^2 \mathbb E |\xi-\bar \xi|^2+4\iota_\theta^2 |F_2-\bar F_2|_{\mathcal F^{\beta,\sigma}}^2   {\bf B}(\beta,1-\theta)^2 t^{2\beta}\notag\\
& +4\iota_\theta^2 c_{F_1}^2 t^{2\theta} \mathbb E \Big[ \int_0^t  (t-s)^{-\theta}  [ |A^\eta[X(s)-\bar X(s)]|+|A^{\beta} [X(s)-\bar X(s)]|] ds\Big]^2\notag\\
&+4\mathbb E \Big|\int_0^t t^\theta A^\theta S(t-s) [G(s)-\bar G(s)]dw_s\Big|^2 \notag\\
\leq & 4\iota_\theta^2 \mathbb E |\xi-\bar \xi|^2+4\iota_\theta^2 |F_2-\bar F_2|_{\mathcal F^{\beta,\sigma}}^2   {\bf B}(\beta,1-\theta)^2 t^{2\beta} \notag\\
&+4\iota_\theta^2 c_{F_1}^2 t^{2\theta+1} \int_0^t   (t-s)^{-2\theta}  \mathbb E [ |A^\eta[X(s)-\bar X(s)]|+|A^{\beta} [X(s)-\bar X(s)]|]^2 ds\notag\\
&+4c(E)\iota_\theta^2 |G-\bar G|_{\mathcal F^{\beta+\frac{1}{2},\sigma}}^2\int_0^t t^{2\theta} (t-s)^{-2\theta} s^{2\beta -1}ds \notag\\
\leq  & 4\iota_\theta^2 \mathbb E |\xi-\bar \xi|^2+4\iota_\theta^2   {\bf B}(\beta,1-\theta)^2 t^{2\beta} |F_2-\bar F_2|_{\mathcal F^{\beta,\sigma}}^2\label{Eq19}\\
&+4c(E)\iota_\theta^2 {\bf B}(2\beta, 1-2\theta)t^{2\beta} |G-\bar G|_{\mathcal F^{\beta+\frac{1}{2},\sigma}}^2\notag\\
&+8\iota_\theta^2 c_{F_1}^2 t^{2\theta+1}  \int_0^t   (t-s)^{-2\theta}  \mathbb E [ |A^\eta[X(s)-\bar X(s)]|^2+|A^{\beta} [X(s)-\bar X(s)]|^2] ds.\notag
\end{align}
Applying these estimates  with $\theta=\beta$ and $\theta=\eta$, we have
\begin{align*}
&\mathbb E| A^\beta[X(t)-\bar X(t)]|^2\notag\\
\leq  & 4\iota_\beta^2 \mathbb E |\xi-\bar \xi|^2t^{-2\beta}+4\iota_\beta^2   {\bf B}(\beta,1-\beta)^2  |F_2-\bar F_2|_{\mathcal F^{\beta,\sigma}}^2\notag\\
&+4c(E)\iota_\beta^2 {\bf B}(2\beta, 1-2\beta) |G-\bar G|_{\mathcal F^{\beta+\frac{1}{2},\sigma}}^2\notag\\
&+8\iota_\beta^2 c_{F_1}^2 t  \int_0^t   (t-s)^{-2\beta}  \mathbb E [ |A^\eta[X(s)-\bar X(s)]|^2+|A^{\beta} [X(s)-\bar X(s)]|^2] ds,\notag
\end{align*}
and
\begin{align*}
&t^{2\eta} \mathbb E|A^\eta[X(t)-\bar X(t)]|^2\notag\\
\leq  & 4\iota_\eta^2 \mathbb E |\xi-\bar \xi|^2+4\iota_\eta^2   {\bf B}(\beta,1-\eta)^2 t^{2\beta} |F_2-\bar F_2|_{\mathcal F^{\beta,\sigma}}^2\notag\\
&+4c(E)\iota_\eta^2 {\bf B}(2\beta, 1-2\eta)t^{2\beta} |G-\bar G|_{\mathcal F^{\beta+\frac{1}{2},\sigma}}^2\notag\\
&+8\iota_\eta^2 c_{F_1}^2 t^{2\eta+1}  \int_0^t   (t-s)^{-2\eta}  \mathbb E [ |A^\eta[X(s)-\bar X(s)]|^2+|A^{\beta} [X(s)-\bar X(s)]|^2] ds.\notag
\end{align*}
By putting 
$$q(t)=t^{2\eta}\mathbb E [ |A^\eta[X(s)-\bar X(s)]|^2+|A^{\beta} [X(s)-\bar X(s)]|^2], $$
we then obtain an integral inequality 
\begin{align}
q(t)\leq  & 4(\iota_\beta^2 t^{2(\eta-\beta)}+\iota_\eta^2 )\mathbb E |\xi-\bar \xi|^2\notag\\
&+4[\iota_\beta^2   {\bf B}(\beta,1-\beta)^2t^{2\eta}+\iota_\eta^2   {\bf B}(\beta,1-\eta)^2 t^{2\beta} ]  |F_2-\bar F_2|_{\mathcal F^{\beta,\sigma}}^2\notag\\
&+4c(E)[\iota_\beta^2 {\bf B}(2\beta, 1-2\beta)t^{2\eta}+\iota_\eta^2 {\bf B}(2\beta, 1-2\eta)t^{2\beta}] |G-\bar G|_{\mathcal F^{\beta+\frac{1}{2},\sigma}}^2 \notag\\
&+8 c_{F_1}^2 t^{2\eta+1}\int_0^t [ \iota_\beta^2 (t-s)^{-2\beta}+\iota_\eta^2  (t-s)^{-2\eta}]  s^{-2\eta}q(s) ds\notag\\
\leq  & C_{\mathcal B_1, \mathcal B_2, B_A}[\mathbb E |\xi-\bar \xi|^2+ t^{2\beta}   |F_2-\bar F_2|_{\mathcal F^{\beta,\sigma}}^2+ t^{2\beta} |G-\bar G|_{\mathcal F^{\beta+\frac{1}{2},\sigma}}^2] \label{Eq47}\\
&+8 c_{F_1}^2 t^{2\eta+1}\int_0^t [ \iota_\beta^2 (t-s)^{-2\beta}+\iota_\eta^2  (t-s)^{-2\eta}]  s^{-2\eta}q(s) ds\notag
\end{align}
for $0<t\leq T_{\mathcal B_1, \mathcal B_2, B_A}.$ 
We use the same techniques as in solving the integral inequality \eqref{Eq48} to solve  \eqref{Eq47}. Arguing first in a small interval $[0,\epsilon]$ and then in the other interval $[\epsilon, T_{\mathcal B_1, \mathcal B_2, B_A}],$ we obtain that
\begin{align}
t^{2\eta}&\mathbb E [ |A^\eta[X(s)-\bar X(s)]|^2+|A^{\beta} [X(s)-\bar X(s)]|^2]
=q(t)\notag\\
&\leq C_{\mathcal B_1, \mathcal B_2, B_A}[\mathbb E |\xi-\bar \xi|^2+ t^{2\beta}   |F_2-\bar F_2|_{\mathcal F^{\beta,\sigma}}^2+ t^{2\beta} |G-\bar G|_{\mathcal F^{\beta+\frac{1}{2},\sigma}}^2] \label{Eq20}
\end{align}
for every $t\in (0, T_{\mathcal B_1, \mathcal B_2, B_A}].$

Now, we shall give an estimate for $\mathbb E|X(t)-\bar X(t)|^2$. Taking $\theta=0$ in \eqref{Eq19}, we have
\begin{align*}
\mathbb E&|X(t)-\bar X(t)|^2 
\leq    4\iota_0 \mathbb E |\xi-\bar \xi|^2+ 4\iota_0  {\bf B}(\beta,1)^2 t^{2\beta} |F_2-\bar F_2|_{\mathcal F^{\beta,\sigma}}^2\notag\\
&+4 c(E) {\bf B}(2\beta, 1)t^{2\beta} |G-\bar G|_{\mathcal F^{\beta+\frac{1}{2},\sigma}}^2 +8\iota_0^2  c_{F_1}^2 t\int_0^t   s^{-2\eta}q(s)ds.\notag
\end{align*}
Using \eqref{Eq20}, we observe that
 \begin{align*}
 t\int_0^t &  s^{-2\eta}q(s)ds\\
 \leq &C_{\mathcal B_1, \mathcal B_2, B_A} t\int_0^t   s^{-2\eta}[\mathbb E |\xi-\bar \xi|^2+ s^{2\beta}   |F_2-\bar F_2|_{\mathcal F^{\beta,\sigma}}^2+ s^{2\beta} |G-\bar G|_{\mathcal F^{\beta+\frac{1}{2},\sigma}}^2]ds\notag\\
  \leq &\frac{C_{\mathcal B_1, \mathcal B_2, B_A} t^{2(1-\eta)}\mathbb E |\xi-\bar \xi|^2 }{1-2\eta}\notag\\
  &+\frac{C_{\mathcal B_1, \mathcal B_2, B_A} t^{2(1+\beta-\eta)}
[|F_2-\bar F_2|_{\mathcal F^{\beta,\sigma}}^2+  |G-\bar G|_{\mathcal F^{\beta+\frac{1}{2},\sigma}}^2] }{1+2\beta-2\eta}.
  \end{align*}
Therefore, 
 \begin{align}
 &\mathbb E|X(t)-\bar X(t)|^2  \label{Eq21}\\
&\leq C_{\mathcal B_1, \mathcal B_2, B_A}[\mathbb E |\xi-\bar \xi|^2+ t^{2\beta}   |F_2-\bar F_2|_{\mathcal F^{\beta,\sigma}}^2+ t^{2\beta} |G-\bar G|_{\mathcal F^{\beta+\frac{1}{2},\sigma}}^2]\notag
\end{align}
for $t\in (0, T_{\mathcal B_1, \mathcal B_2, B_A}]. $ 
By \eqref{Eq20} and \eqref{Eq21}, \eqref{Eq22} has been verified.

Let us now show  \eqref{Eq23}. By substituting the estimate
$$|A^\theta S(t) (\xi-\bar \xi)|\leq \iota_{\beta-\theta} t^{\beta-\theta} |A^\beta  (\xi-\bar \xi)|$$
with $\theta=\beta$ and $\theta=\eta$ for 
$|A^\theta S(t) (\xi-\bar \xi)|\leq \iota_\theta t^{-\theta} |\xi-\bar \xi|,$ we obtain a similar result to \eqref{Eq19}:
\begin{align*}
&\mathbb E|t^\theta A^\theta[X(t)-\bar X(t)]|^2\notag\\
\leq  & 4 \iota_{\beta-\theta}^2 t^{2(\beta-\theta)} \mathbb E |A^\beta  (\xi-\bar \xi)|^2
+4\iota_\theta^2   {\bf B}(\beta,1-\theta)^2 t^{2\beta} |F_2-\bar F_2|_{\mathcal F^{\beta,\sigma}}^2\notag\\
&+4c(E)\iota_\theta^2 {\bf B}(2\beta, 1-2\theta)t^{2\beta} |G-\bar G|_{\mathcal F^{\beta+\frac{1}{2},\sigma}}^2\notag\\
&+8\iota_\theta^2 c_{F_1}^2 t^{2\theta+1}  \int_0^t   (t-s)^{-2\theta}  \mathbb E [ |A^\eta[X(s)-\bar X(s)]|^2+|A^{\beta} [X(s)-\bar X(s)]|^2] ds.\notag
\end{align*}
Using the same arguments as in verifying \eqref{Eq20}, we conclude that  \eqref{Eq23} holds true. 
 It completes the proof.
\end{proof}
\subsection{Case $\tilde\beta=0$}
This subsection investigates the critical case of the Lipschitz condition \eqref{AbetaLipschitzcondition} when $\tilde\beta=0.$  We assume that
\begin{itemize}
\item [(H1')] $F_1$ defines on the domain $ \mathcal D(A^\eta)$ and \eqref{AbetaLipschitzcondition} is valid with $\tilde\beta=0,$ i.e.  
\begin{equation} \label{AbetaLipschitzconditionBeta0}
        |F_1(x)-F_1(y)|\leq c_{F_1}  [ |A^\eta(x-y)|+|x-y|], \quad\quad x,y\in \mathcal D(A^\eta).
      \end{equation}
\end{itemize}
We can then generalize some results of Theorem \ref{semilinear evolution equationTheorem1}.
\begin{theorem}\label{semilinear evolution equationTheorem2}
Let \eqref{spectrumsectorialdomain}, \eqref{resolventnorm}, {\rm (H1')}, {\rm (H2)} and {\rm (H3)}   be satisfied.
Assume that   $\mathbb E|\xi|^2 < \infty$. Then \eqref{semilinear evolution equation} possesses a unique mild solution $X$ in the function space:
$$
X\in  \mathcal C((0,T_{F_1,F_2,G,\xi}];\mathcal D(A^\eta)),
$$
where $T_{F_1,F_2,G,\xi}$ depends only on the squared norms $|F_2|_{\mathcal F^{\beta,\sigma}}^2, |G|_{\mathcal F^{\beta+\frac{1}{2},\sigma}}^2$ and  $\mathbb E |F_1(0)|^2$. In addition, $X$ satisfies the estimate
\begin{align}
\mathbb E|X(t)|^2 \leq C_{F_1,F_2,G,\xi},    \quad\quad t\in [0,T_{F_1,F_2,G,\xi}]\label{semilinear evolution equationExpectationAbetaXSquareBeta0}
\end{align}
with some constant $C_{F_1,F_2,G,\xi}$ depending only on $|F_2|_{\mathcal F^{\beta,\sigma}}^2, |G|_{\mathcal F^{\beta+\frac{1}{2},\sigma}}^2, \mathbb E |F_1(0)|^2$  and  $\mathbb E|\xi|^2.$ 
\end{theorem}
\begin{proof}
The proof is analogous to that of Theorem \ref{semilinear evolution equationTheorem1}. For each $S\in (0,T]$, we set the Banach space:
\begin{align*}
\Xi (S)=&\{Y\in \mathcal C((0,S];\mathcal D(A^\eta)) \cap  \mathcal C([0,S];E) \text{  such that  }\\
& \sup_{0<t\leq S} t^{2\eta} \mathbb E|A^\eta Y(t)|^2+ \sup_{0\leq t\leq S}\mathbb E|Y(t)|^2 <\infty \}
\end{align*}
with norm 
$$|Y|_{\Xi (S)}=\Big[\sup_{0<t\leq S} t^{2\eta} \mathbb E|A^\eta Y(t)|^2+ \sup_{0\leq t\leq S}\mathbb E|Y(t)|^2\Big]^{\frac{1}{2}}. $$
Consider a nonempty closed subset $\Upsilon(S) $ of $\Xi (S)$ which consists of all function $Y\in \Xi (S)$ such that
\begin{equation}  \label{Upsilon(S)DefinitionBeta0}
\max\{\sup_{0<t\leq S} t^{2\eta} \mathbb E|Y(t)|^2,
 \sup_{0\leq t\leq S}\mathbb E|A^\beta Y(t)|^2\} \leq \kappa^2
\end{equation}
with $\kappa>0 $  which will be fixed appropriately. 

Similarly to the proof  of Theorem \ref{semilinear evolution equationTheorem1}, if we choose $\kappa>0 $ dependent only on $|F_2|_{\mathcal F^{\beta,\sigma}}^2, |G|_{\mathcal F^{\beta+\frac{1}{2},\sigma}}^2, \mathbb E |F_1(0)|^2$ and $ \mathbb E |A^\beta\xi|^2,$  and if $S$ is sufficiently small, then the mapping $\Phi$ defined by \eqref{DefinitionOfFunctionPhi} maps the set $\Upsilon(S) $ into itself and is contraction with respect to the norm of $\Xi (S)$.
Consequently, $\Phi$ possesses a unique fixed point $X\in \Upsilon(S)$, i.e. for every $t\in [0,S]$, $X(t)=\Phi X(t)$ a.s. This means that $X$ is a local mild solution of \eqref{semilinear evolution equation}. Following  {\bf Step 4} and {\bf Step 6} in the proof of Theorem \ref{semilinear evolution equationTheorem1}, the estimate \eqref{semilinear evolution equationExpectationAbetaXSquareBeta0} and uniqueness of the solution  are verified.
\end{proof}

By the same arguments as in Corollary \ref{semilinear evolution equationTheorem1corollary}, global mild solutions of \eqref{semilinear evolution equation} can be constructed.
\begin{corollary}[global existence]
Assume that in Theorem \ref{semilinear evolution equationTheorem2} the constant  $ C_{F_1,F_2,G,\xi}$ is independent of  $T_{F_1,F_2,G,\xi}$ for every initial value $\xi\in E$. Then \eqref{semilinear evolution equation}  possesses a unique  mild solution on the interval $[0,T].$
\end{corollary}
The next theorem shows regularity of local mild solutions for more regular initial values. The proof of the theorem is similar to that of Theorem \ref{semilinear evolution equationMoreRegular}, so we omit it.
\begin{theorem}
Let \eqref{spectrumsectorialdomain}, \eqref{resolventnorm}, {\rm (H1')}, {\rm (H2)} and {\rm (H3)}   be satisfied.
Let $F_2\in \mathcal F^{\gamma, \sigma}((0,T];E)$,  $ \max\{\beta,\frac{1}{2}-\eta\}<\gamma< \frac{1}{2}$, and  $\xi\in \mathcal D(A^\gamma)$ such that $\mathbb E|A^\gamma\xi|^2<\infty$.  Then \eqref{semilinear evolution equation} possesses a unique mild solution $X$ in the function space:
\begin{equation*}
X\in  \mathcal C((0,T_{F_1,F_2,G,\xi}];\mathcal D(A^\eta))\cap \mathcal C([0,T_{F_1,F_2,G,\xi}];\mathcal D(A^\gamma)),
\end{equation*}
where $T_{F_1,F_2,G,\xi}$ depends only on the squared norms $|F_2|_{\mathcal F^{\beta,\sigma}}^2, |G|_{\mathcal F^{\beta+\frac{1}{2},\sigma}}^2$ and $\mathbb E |F_1(0)|^2$ and $ \mathbb E |A^\gamma\xi|^2.$ In addition, $X$ satisfies the estimate
\begin{align*}
\mathbb E|X(t)|^2+t^{2\gamma}\mathbb E |A^\gamma X(t)|^2 \leq C_{F_1,F_2,G,\xi},    \quad\quad t\in [0,T_{F_1,F_2,G,\xi}], \label{semilinear evolution equationExpectationAbetaXSquareMoreRegularBeta0}
\end{align*}
with some constant $C_{F_1,F_2,G,\xi}$ depending only on $|F_2|_{\mathcal F^{\beta,\sigma}}^2, |G|_{\mathcal F^{\beta+\frac{1}{2},\sigma}}^2, \mathbb E |F_1(0)|^2$ and $ \mathbb E |A^\gamma\xi|^2.$ 
\end{theorem}
Let us finally verify continuous  dependence of solutions on initial data.
\begin{theorem}\label{continuityofsolutionsininitialdata2}
Let \eqref{spectrumsectorialdomain}, \eqref{resolventnorm}, {\rm (H1')}, {\rm (H2)} and {\rm (H3)}   be satisfied.  Let $X$ and $\bar X$ be the solutions of  \eqref{semilinear evolution equation} for the data $(F_2,G,\xi)$ and $(\bar F_2,\bar G,\bar \xi)$ in $\mathcal B_1\times \mathcal B_2\times B_A$, respectively, where $\mathcal B_1$, $ \mathcal B_2$ and $ B_A$ are defined by \eqref{mathcal B1}, \eqref{mathcal B2} and  \eqref{BABall}.
Then
\begin{align}
&t^{2\eta}\mathbb E  |A^\eta[X(s)-\bar X(s)]|^2+t^{2\eta}\mathbb E|X(s)-\bar X(s)|^2 \notag\\
& \leq C_{\mathcal B_1, \mathcal B_2, B_A}[\mathbb E |\xi-\bar \xi|^2+ t^{2\beta}   |F_2-\bar F_2|_{\mathcal F^{\beta,\sigma}}^2+ t^{2\beta} |G-\bar G|_{\mathcal F^{\beta+\frac{1}{2},\sigma}}^2],    \notag
\end{align}
and 
\begin{align}
&t^{2(\eta-\beta)}[\mathbb E  |A^\eta[X(s)-\bar X(s)]|^2+\mathbb E|X(s)-\bar X(s)|^2] \notag\\
& \leq C_{\mathcal B_1, \mathcal B_2, B_A}[\mathbb E |A^\beta(\xi-\bar \xi)|^2+    |F_2-\bar F_2|_{\mathcal F^{\beta,\sigma}}^2+ |G-\bar G|_{\mathcal F^{\beta+\frac{1}{2},\sigma}}^2]  \notag
\end{align}
for $t\in (0, T_{\mathcal B_1, \mathcal B_2, B_A}].$
\end{theorem}
As the proof of this theorem is quite analogous to that of Theorem  \ref{continuityofsolutionsininitialdata}, we may omit it.

\end{document}